\documentclass[12pt,a4paper]{amsart}
\usepackage[utf8]{inputenc}
 \usepackage[english]{babel}
\usepackage[a4paper,inner=2.8cm,outer=2.5cm,top=2.8cm,bottom=2.5cm]{geometry}     
\usepackage[all]{xy}
\usepackage{microtype}
\usepackage{booktabs}   
\usepackage{tabularx}   
\usepackage{makecell}   
\usepackage{ragged2e}
\usepackage{array} 
\usepackage{amsmath,amssymb,amscd,amsthm,amsfonts,anyfontsize,color,enumerate,fix-cm,layout,lipsum,lpic,mathrsfs,mdwlist,stmaryrd,tensor,tikz,thmtools,xspace,comment}
\usepackage{float}
\usepackage{pdfpages}
\usepackage{graphicx}
\usepackage[all]{xy}
\usepackage{mathtools}
\usepackage[bottom]{footmisc}
\usepackage{ifthen}
\usepackage{xfrac}
\usepackage{tikz-cd}
\theoremstyle{plain}                 
\newtheorem{theorem}{Theorem}[section]     
\newtheorem{proposition}[theorem]{Proposition}
\newtheorem{corollary}[theorem]{Corollary}

\newtheorem{lemma}[theorem]{Lemma}    
\theoremstyle{definition}           
\newtheorem{definition}[theorem]{Definition} 
\newtheorem{problem}[theorem]{Problem}
\newtheorem{example}[theorem]{Example}

\theoremstyle{remark}       
\newtheorem{remark}[theorem]{Remark} 
\newtheorem{conjecture}[theorem]{Conjecture} 

\usetikzlibrary{calc,decorations.markings,positioning}

\usepackage{hyperref}
\hypersetup{colorlinks=true,linkcolor=blue,citecolor=purple,filecolor=magenta,urlcolor=cyan}

\newcommand{\R}{\mathbb{R}}

\def\bea{\begin{eqnarray}}
\def\eea{\end{eqnarray}}

\newcolumntype{M}[1]{>{\centering\arraybackslash}m{#1}}

\newcommand{\boris}[1]{\textcolor{red}{#1}}

\begin{document}
\usetikzlibrary{arrows.meta}
\title[Electrical networks, Grassmannians, and cluster algebras]{Electrical networks, Grassmannians, and cluster algebras}

\author[B.~Bychkov]{B.~Bychkov}
\address{B.~B.: Department of Mathematics, University of Haifa, Mount Carmel, 3498838, Haifa, Israel}
\email{bbychkov@hse.ru}

\author[L. Guterman]{L.~Guterman}
\address{L.~G.: Einstein Institute of Mathematics, Hebrew University, Jerusalem 91904, Israel.}
\email{lazar.guterman@mail.huji.ac.il}

\author[A. Kazakov]{A.~Kazakov}
\address{A.~K.: National Research University Higher School of Economics;  Centre of Integrable Systems, P. G. Demidov Yaroslavl State University, Sovetskaya 14, 150003, Yaroslavl, Russia}
\email{anton.kazakov.4@mail.ru}
\begin{abstract}

The paper studies the problem of circular total positivity of the symmetric matrices with zero row sums. These matrices are exactly response matrices of the electrical networks. Alman, Lian and Tran described tests for circular total positivity in two related frameworks: the cluster algebra $\mathcal{CM}_n$ and the Laurent Phenomenon algebra $\mathcal{LM}_n$. 
Our first result is the construction of a seed in Scott's cluster algebra structure on the coordinate ring of the Grassmannian $\mathrm{Gr}(n-1,2n)$ that consists entirely of circular minors. We compare the cluster structure induced by this seed with $\mathcal{CM}_n$. In particular, for odd $n$ the cluster algebra structure $\mathcal{CM}_n$ is isomorphic to the cluster algebra structure on $\mathrm{Gr}(n-1,2n)$ subject to natural freezing and trivialization of certain cluster variables in their initial seeds. We use this isomorphism to relate circular total positivity to positivity in the Grassmannian.
Our second result is that the Laurent Phenomenon algebra $\mathcal{LM}_n$ is isomorphic to the coordinate ring of the noncompactified space of electrical network, or equivalently, 
to a certain localization of the grove algebra.

\medskip


\medskip


\end{abstract}
\maketitle

\tableofcontents

\section{Introduction} \label{sec:background}

\subsection{State of the art}
A matrix $A \in \mathrm{Mat}_{n \times n}(\mathbb R)$ is called  {\itshape totally positive} if all its minors are positive. The theory of totally positive matrices, initially motivated by problems in physics \cite{Gan}, currently plays an important role in algebra and geometry. 
A {\itshape positivity test} is a minimal set of minors, whose positivity guarantees the positivity of all minors, see \cite{BFZ}. It turned out that such tests can be systematically described using {\itshape cluster algebras} \cite{FZ tests and par, FZ positivity, FZ I, BFZ III}. 

A {\itshape cluster algebra} is a commutative ring generated inside an ambient field by a family of distinguished generators called {\itshape cluster variables} which are grouped into {\itshape clusters}. These generators are produced recursively from a
fixed initial family of variables by a process called {\itshape mutation} inside the ambient field. A {\itshape quiver}, whose vertices are labeled by the variables of the cluster, is associated with each cluster; and a mutation is a rule to transform the quiver. A common way to generate a cluster algebra is to start from a single {\itshape seed}, i.e. a pair of a quiver and the corresponding cluster, and generate the cluster algebra by means of all possible mutations.

The cluster algebra that encodes total positivity of matrices was constructed in \cite{BFZ III}. This became one of the foundational examples in the theory of cluster algebras \cite{FZ I, FZ II, BFZ III}. The cluster algebra $\mathcal{A}_{k,n}$ on the coordinate ring $\mathbb C[\mathrm{Gr}(k,n)]$ of the Grassmannian $\mathrm{Gr}(k,n)$ was constructed in \cite{Sc}. 

The {\itshape totally non-negative Grassmannian} $\mathrm{Gr}_{\ge0}(k,n)$ was described combinatorially in many different ways in \cite{Pos}. 
This raises the question of describing a positivity test for a point in $\mathrm{Gr}(k,n)$ in terms of the Plücker coordinates. The cluster algebra $\mathcal{A}_{k,n}$ encodes these positivity tests. There is a relation between total positivity of matrices and positivity in the Grassmannian. Namely, there is a seed in $\mathcal{A}_{n,2n}$ that coincides with a seed in the cluster algebra for matrices. Thus, tests for total positivity of matrices are induced by tests for positivity of points in the Grassmannian $\mathrm{Gr}(n,2n)$, see \cite[Remark 6.7.15]{FZW Ch6}.

A {\itshape circular minor} is a minor whose rows and columns indices, if arranged on a circle, lie in a special {\itshape circular clockwise order}. A {\itshape circular total positivity} of a matrix is a relaxation of total positivity meaning that all signed circular minors are positive. In addition to totally positive matrices, \textit{circular totally   positive} matrices also play an important role in applications, see \cite{BDV, CM, Lom}.

An \textit{electrical network} with $n$ boundary vertices induces a matrix that measures the electrical properties of the network and is called a {\itshape response matrix}. The set of response matrices for \textit{well-connected networks} (i.e. the general position networks) was completely characterized in \cite{CM} as a special class of circular totally positive matrices. Namely, these are symmetric circular totally positive matrices whose row sums equal to zero. Thus, it is natural to study positivity tests for circular minors. One such test is given by the so-called {\itshape central circular minors} as shown in \cite{KW space}. In \cite{ALT}, positivity tests on circular minors were studied systematically, and in this way, the authors constructed the Laurent Phenomenon algebra $\mathcal{LM}_n$ and the cluster algebra $\mathcal{CM}_n$ on circular minors, whose initial seeds consist of central circular minors. 

The embedding $\mathcal{Y}: E_n \to \mathrm{Gr}_{\geq 0}(n-1,2n)$ of the set of electrical networks modulo the {\itshape electrical equivalence} into the totally non-negative Grassmannian $\mathrm{Gr}_{\ge0}(n-1,2n)$ was constructed in \cite{L} and subsequently studied in \cite{CGS, BGKT, BGGK}.

In the present paper we prove that Plücker coordinates of a certain form on $\mathrm{Gr}(n-1,2n)$  and {\itshape contiguous} circular pairs are in a natural bijection. 
We next find a seed in the cluster algebra $\mathcal{A}_{n-1,2n}$ in which 
every Plücker coordinate corresponds to a central circular pair, and hence equals to a
central circular minor after restriction to Lam’s image. Also, we note that the obtained seed is closely related to the initial seed of the cluster algebra $\mathcal{CM}_n$. The initial seeds of $\mathcal{A}_{k,n}$ are given by {\itshape plabic graphs} and their {\itshape clusters} are given by {\itshape Scott's rule}. For odd $n$, we choose as {\itshape the fixed initial seed} of $\mathcal{A}_{n-1,2n}$ the plabic graph obtained by {\itshape generalized Temperley's trick} from a certain well-connected electrical network.

The first main result of the paper is that for odd $n$ the cluster algebra structure $\mathcal{A}_{n-1,2n}$ coincides with the cluster algebra structure $\mathcal{CM}_n$ after a natural freezing and subsequent  trivialization of certain variables in the initial seeds (see Theorem \ref{thm: cluster algebras coinside}). 

\begin{theorem} \label{thm: introduction}
Let $n$ be a positive odd integer. 
The cluster algebra structure  $\mathcal{A}_{n-1,2n}$ coincides with the cluster algebra structure $\mathcal{CM}_n$ subject to natural freezing and subsequently trivializing of certain cluster variables. 
\end{theorem}

We prove Theorem \ref{thm: introduction} by providing a seed in $\mathcal{A}_{n-1,2n}$ which coincides with the initial seed of $\mathcal{CM}_n$ subject to natural freezing and subsequent trivialization of certain cluster variables. The quiver underlying the fixed initial seed of $\mathcal{A}_{n-1,2n}$, as well as the initial seed of $\mathcal{CM}_n$, is centrally symmetric. In the fixed initial seed of $\mathcal{A}_{n-1,2n}$, the $n$ vertices closest to the center of symmetry, and in the initial seed of $\mathcal{CM}_n$, the unique central vertex, correspond to the empty circular pair. These are precisely the vertices corresponding to the
$n$ variables referred to in Theorem \ref{thm: introduction}.

In parallel, in the case of even $n$ we prove the following result (see Theorem \ref{thm: even case}):
\begin{theorem}
Let $n$ be an even natural number. 
There is a cluster in the cluster algebra $\mathcal{A}_{n-1,2n}$ consisting entirely of Plücker coordinates corresponding to central circular minors.
\end{theorem}

If we interpret the Plücker coordinates as Plücker coordinates of $\mathcal{Y}(e)$ and circular minors as circular minors of $M(e)$ for some well-connected electrical network $e$, then the correspondence from Theorem \ref{thm: introduction} is actually an identity. Namely, we prove that any contiguous circular minor of the response matrix of a given electrical network is equal, up to a scalar, to a Plücker coordinate of the image of this network in the Grassmannian. 
We use this property and Theorem \ref{thm: introduction} in order to show that the tests for circular total positivity are induced by positivity tests on the Plücker coordinates of points in the Grassmannian $\mathrm{Gr}(n-1,2n)$. For an $n\times n$ symmetric matrix with zero row sums, we obtain the following
result (see Corollary \ref{cor: positivity tests}):

\begin{corollary} \label{cor: introduction}
        The test for circular positivity given by the set of central circular minors, or equivalently by the initial seed of $\mathcal{CM}_n$, coincides with the test for positivity of a certain point in the Grassmannian $Gr(n-1,2n)$ given by a certain initial seed for $\mathcal{A}_{n-1,2n}$.
\end{corollary}

We apply Theorem \ref{thm: introduction} in two more directions. Firstly, we provide a new proof of the Laurent Phenomenon for contiguous circular minors based on the Laurent Phenomenon in cluster algebras (see Theorem \ref{loran-ph-circ}). Secondly, we prove that the coordinate ring of the noncompactified space of electrical networks, which was shown in \cite{GLX} to be a certain localization of the grove algebra, admits a structure of a Laurent Phenomenon algebra (see Theorem \ref{thm: LP grove algebra}): 
\begin{theorem} \label{cl-alg-1}
     The coordinate ring of the noncompactified space of electrical networks is isomorphic to $\mathcal{LM}_n$, where the latter is regarded as an algebra over Laurent polynomials in the central frozen variable of the initial seed of $\mathcal{LM}_n$. 
\end{theorem}

\subsection{Motivation}\label{sec: Motivation}
Here we describe two constructions in more detail:\\
1. The classical construction relating total positivity for matrices and positivity for Grassmannians (see \cite[Section 3]{Pos}, \cite[Section 4]{Sc}, \cite[Appendix A4]{MSp},
and\\ 
2. The new construction relates circular total positivity for matrices to positivity for Grassmannians. 
\begin{itemize}
    \item  First, embed a matrix $A=(a_{ij}) \in \mathrm{Mat}_{n \times n}$ into the Grassmannian $\mathrm{Gr}(n, 2n)$ using the following matrix:
   \begin{gather*}
        X(A)= 
        \left(
	\begin{array}{cccccccccc}
			0& \dots & 0 & 0  & 1 &a_{11} & a_{12} & a_{13} & \dots &a_{1n}\\
			0& \dots & 0 & -1  & 0 & a_{21} & a_{22} & a_{23} & \dots & a_{2n}\\
            0& \dots & 1 & 0  & 0 & a_{31} & a_{32} & a_{33} & \dots&a_{3n} \\
	\vdots&\ddots&\vdots&\vdots&\vdots&\vdots& \vdots &  \vdots &\ddots&\vdots\\
    -1&\dots & 0&0&0&a_{n1}&a_{n2}&a_{n3}&\dots &a_{nn}
		\end{array}
		\right).
    \end{gather*}
It is easy to verify that for each minor of the matrix $A$ there is a unique Plücker 
 coordinate of $X(A)$ that equals to it, see \cite[Remark 6.7.15]{FZW Ch6}.
\item  Then construct a certain reduced plabic graph 
and label its faces by Plücker coordinates using Scott's rule, see Section \ref{sec:Scottdef}. This plabic graph, combinatorially, is the same as a plabic graph defining $X(A)$, assuming $X(A)\in Gr_{\ge0}(n,2n)$; see Theorem \ref{th:plabicgr}.
\item   The obtained set of Plücker coordinates corresponds to a  test for total positivity of the matrix $A$. Moreover, the same set forms a cluster in the \textit{cluster algebra} for $\mathrm{Gr}(n, 2n)$, see Section \ref{sec:Scottdef}. 
\end{itemize}

We next present our construction relating circular total positivity for symmetric matrices with row entry sums equal to $0$ and positivity for Grassmannians.
\begin{itemize}
    \item First, embed a symmetric matrix $M = (x_{ij})$ with zero row sums equal to $0$, into the Grassmannian $\mathrm{Gr}(n-1, 2n)$ using the following matrix, see Section \ref{sec: Positivity tests for circular minors}:
\begin{gather*}
  \Phi_n(M) =
     \left(
	\begin{array}{cccccccccc}
        x_{11} & 1 & -x_{12} & 0 & x_{13} & 0 & \cdots & (-1)^n \\
        -x_{21} & 1 & x_{22} & 1 & -x_{23} & 0 & \cdots & 0 \\
        x_{31} & 0 & -x_{32} & 1 & x_{33} & 1 & \cdots & 0 \\
        \vdots &\vdots & \vdots & \vdots& \vdots &\vdots  & \ddots & \vdots \\
    \end{array} \right).
\end{gather*}
   It can be verified that each contiguous circular minor of $M$ is equal, up to a sign, to a Plücker coordinate of $\Phi_n(M)$, see Lemma \ref{lem:cir-min-gen}.
   \item Then we  construct a certain reduced plabic graph 
   and label its faces by the Plücker coordinates according to Scott's rule.
\item   The obtained set of Plücker coordinates  corresponds to a test for circular total positivity of the matrix $M$. Moreover, the same set forms a seed of the cluster algebra for $\mathrm{Gr}(n-1, 2n)$, which is closely related to the initial seed in the cluster algebra $\mathcal{CM}_n$, see Section \ref{Central circular as seed}.
\end{itemize}

The problem of obtaining tests for circular total positivity is also important due to its connection (see \cite[Theorem 1]{BDV}) to the numerical solution of the classical Calderón problem. Specifically, this method is  based on solving the discrete Calderón problem (also known  as the black box problem \cite{CM}), which consists of reconstructing the unknown conductivities of a given electrical network from its known response matrix. It has been shown (see \cite{Kaz} and \cite{Geo}) that this problem can be solved by applying the generalized chamber ansatz to a plabic graph associated with an electrical network using Temperley's trick. In Section \ref{Central circular as seed}, we provide an explicit labeling rule for the faces of the plabic graph associated with electrical networks on graphs $G_{2n+1}$; see Section \ref{sub:spclass} for the definition.   Together with the formulas from \cite[Example 3.7]{Kaz}, this yields a non-recursive solution to the discrete Calderón problem for electrical networks on $G_{2n+1}$, which are dual to the networks used in \cite{BDV} for the numerical solution of the classical Calderón problem. 

\subsection{Organization of the paper}

In Section \ref{sec: preliminaries} we provide the necessary background on the theory of electrical networks and its connection to the non-negative Grassmannian. In Section \ref{sub:circ} we discuss circular minors and prove the key Lemma \ref{lemma: bijection} and Proposition \ref{circular minors via Plücker coordinates 0}.
In Section \ref{sec:cl} we provide background on cluster algebras and present the constructions of cluster algebras needed for our work. 

The main result of the paper is Theorem \ref{thm: cluster algebras coinside}, which we prove in Section \ref{Central circular as seed}, relying on the logic described above and based on the key Lemma \ref{lemma: bijection}.
In Section \ref{sec:even n} we prove the analogous result in the more subtle case of even $n$.
In Section \ref{sec: Positivity tests for circular minors} we specialize Theorem \ref{thm: cluster algebras coinside} in order to study tests for circular total positivity; in Section \ref{sec: Laurent phenomenon for circular minors} we apply Theorem \ref{thm: cluster algebras coinside} to provide a new proof of the Laurent Phenomenon for contiguous circular minors; in Section \ref{sec: Laurent Phenomenon structure on the grove algebra} we equip a certain localization of the grove algebra with the structure of a Laurent Phenomenon algebra. 
In the main part of the paper we deal with well-connected electrical networks, which are the maximal cells in the poset of electrical networks. We discuss the non-maximal cells in Section \ref{sec:cells}. In Appendix \ref{gen-form-ap} we prove Lemma \ref{Parametrization of the top-cell} and in Appendix \ref{sec:appendix B} we prove a technical lemma about the reachability of contiguous circular pairs in $\mathcal{CM}_n$.

\subsection{Acknowledgments} 
We thank Boris~Feigin, Vassily~Gorbounov, Pavlo~Pylyavskyy, Dmitry~Talalaev and Alek~Vainshtein for valuable discussions.
B.~B. was partially supported by the ISF grant 2848/25. L.~G. was partially supported by the ISF grant 687/24. The work of Anton Kazakov on Sections  \ref{sub:circ}, \ref{sec: main results} was supported by RSF (project No. 26-11-00379, https://rscf.ru/project/26-11-00379/), which finances the work of the author at P. G. Demidov Yaroslavl State University. The work of A.~K. on Section \ref{sec: preliminaries}, \ref{sec:cl} is an output of a research project implemented as part of the Basic Research Program at the National Research University Higher School of Economics (HSE University).

\section{Preliminaries} \label{sec: preliminaries}

\subsection{Electrical networks} \label{sec:networks} 
We start with a brief introduction to the theory of (circular) electrical networks following \cite{CM} and \cite{K}.
\begin{definition} 
An {\itshape electrical network} $e(\Gamma,  \omega)$ is a  planar graph $\Gamma=(V, E)$ embedded into a disk, which satisfies the following conditions:  
\begin{itemize}
    \item The nodes are partitioned into the set of interior nodes
$V_I$ and the set of boundary nodes $V_B$; each boundary node lies on the boundary circle;
    \item Boundary nodes are numbered clockwise from $1$ to $n:=|V_B|.$ Interior nodes are enumerated arbitrarily from $n+1$ to $|V|$;
    \item Every edge $v_iv_j$ of $\Gamma$ is assigned a positive weight  $\omega_{ij}$, representing its conductance. We denote the set of all weights by $\omega.$
\end{itemize}

\end{definition}

Consider an electrical network $e(\Gamma,  \omega)$ and prescribe boundary voltages  $\textbf{U}:V_{B}\to \mathbb{R}$. Then these boundary voltages extend uniquely to a function $U: V \to \mathbb{R}$ on all vertices, which is determined from Kirchhoff's and Ohm's laws:
 \begin{equation*}
     \sum \limits_{j \in V}\omega_{ij}\bigr(U(i)-U(j)\bigl)=0, \ \forall i \in V_I,
 \end{equation*}
 The boundary current vector $\textbf{I}=\{I_1, \dots, I_n\}$ associated with a prescribed boundary voltage vector $\textbf{U}=\{U(1), \dots, U(n)\}$ is defined by
 \begin{equation*}
    I_k:=\sum \limits_{j \in V}\omega_{kj}\bigr(U(k)-U(j)\bigl), \ k\in \{1, \dots, n\}.
 \end{equation*}

A fundamental result in the theory of electrical networks is that boundary voltages and currents are linearly related to each other.  
\begin{theorem} \textup{\cite[Theorem 3.2]{CM}} \label{gen-theorem}
Consider an electrical network $e(\Gamma, \omega).$ Then there is a  matrix $M_R(e)=(x_{ij}) \in Mat_{n \times n}(\mathbb{R})$ such that the following holds:
$$M_R(e)\textbf{U}=\textbf{I}.$$

This matrix satisfies the following conditions:
\begin{itemize}
    \item The matrix $M_R(e)$ is symmetric;
    \item The sum of the entries of each row $($column$)$ equals $0;$
    \item All off-diagonal entries $x_{ij}$ of $M_R(e)$ are non-positive.
\end{itemize}
\end{theorem} 
The matrix $M_R(e)$ is called the {\itshape response matrix} of the network $e(\Gamma, \omega)$.

Two electrical networks are called {\itshape equivalent} if they can be transformed into each other by a sequence of electrical transformations, see  Fig. \ref{fig:el_trans}. 

\begin{theorem} \textup{\cite[Theorem 1]{CM}, \cite[Theorem 4]{CGV}}
   Two electrical networks are equivalent if and only if they have the same  response matrix. An electrical network is uniquely determined by its response matrix $M_R(e)$ up to electrical transformations.
\end{theorem}
\begin{figure}[H]
    \centering
    \includegraphics[scale=0.2]{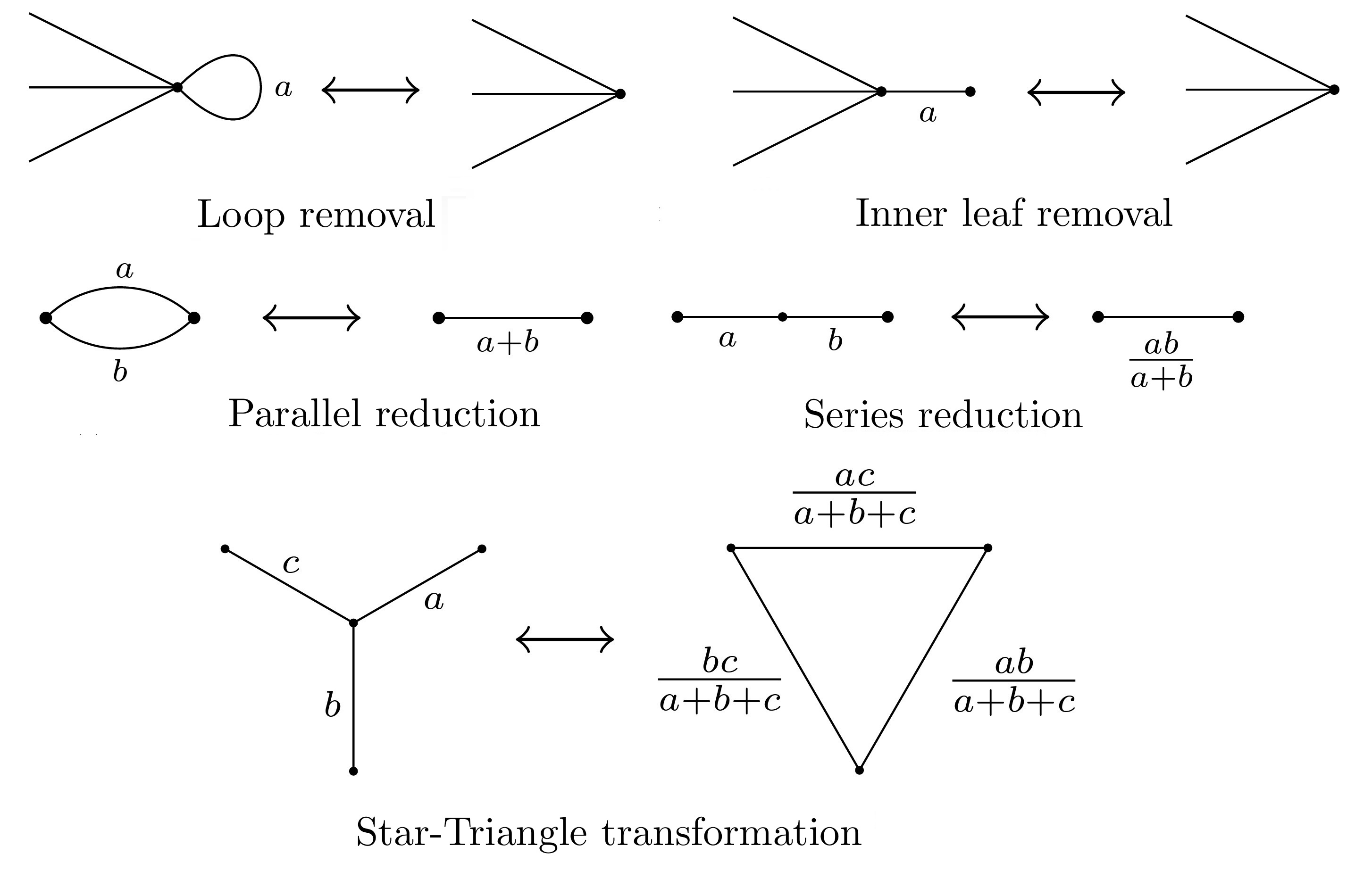}
 \hspace{-1.5cm}   \caption{Electrical transformations}
    \label{fig:el_trans}
\end{figure}

We denote by $E_n$ the set of all electrical networks with $n$ boundary nodes  up to electrical transformations. 
It is natural to consider generic electrical networks, namely networks that cannot be simplified by electrical transformations and have the maximum possible number of edges. This motivates the definition of well-connected networks:

\begin{definition}   
\label{def:well-con} 
    An electrical network $e(\Gamma, \omega) \in E_n$  is called {\itshape well-connected}  
    if it is equivalent to the standard network, see Fig. \ref{fig:standart}; for more details, see \cite[Section 4.4.4]{K}.
\end{definition}
\begin{figure}[H]
     \hspace*{-7mm}
     \includegraphics[scale=0.21]{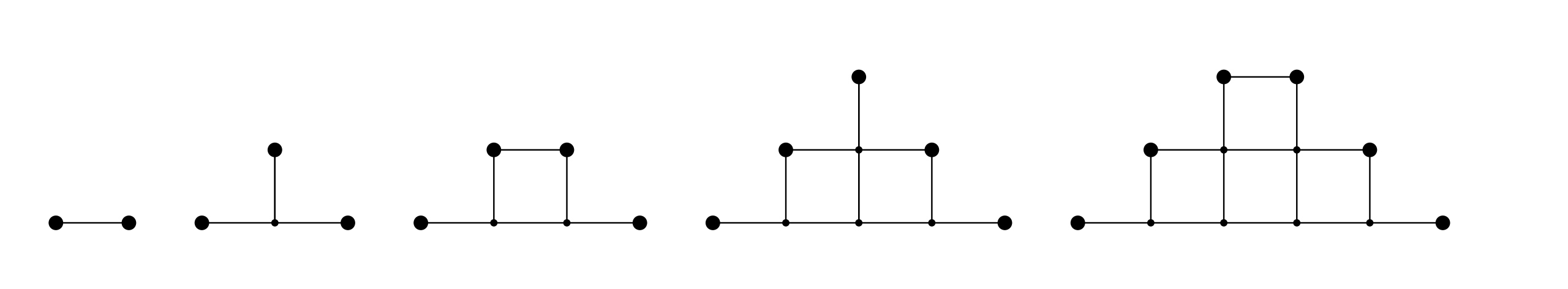}
     \caption{Standard networks, boundary nodes are bold } \label{fig:standart}
     \end{figure}
\begin{remark}
    Note that standard electrical networks are just one particular choice of a representative within the equivalence class of well-connected networks. A different, more symmetric, representative will be used in Section \ref{sub:spclass}. Well-connected networks have a different combinatorial interpretation in terms of connections, see \cite[Section 3.7]{CM}. Later we will find out that as Theorem \ref{Set of response matrices all network well conected} asserts, response matrices of
well-connected networks are
precisely the circular totally
positive matrices among
response matrices; arbitrary
response matrices are
circular totally non-negative.
\end{remark}

\subsection{Embedding of \texorpdfstring{$E_n$}{En} into totally the non-negative Grassmannian} \label{sec: lam emb} 
In this section we will describe an embedding of $E_n$ into $\mathrm{Gr}_{\geq0}(n-1,2n).$
\begin{definition}
The totally non-negative Grassmannian $\mathrm{Gr}_{\geq 0}(k, m)$ is the subset of all points of the complex Grassmannian $\mathrm{Gr}(k, m),$ whose Plücker coordinates $\Delta_I, \ I \subset \binom{m}{k} $ can be chosen to be real and non-negative.
\end{definition}

\begin{theorem} \textup{\cite[Theorem 3.3]{BGKT}} \label{th: main_gr}
 Consider an  electrical network  $e(\Gamma, \omega) \in E_n$ with a response matrix $M_R(e)=(x_{ij}).$ Then there is an injective map $\mathcal{L}: E_n\to \mathrm{Gr}_{\geq 0}(n-1,2n)$ such that $\mathcal{L}(e)$ is the row space of the $(n\times 2n)$ matrix
\begin{equation*}
\Omega_n(e)=\left(
\begin{array}{cccccccc}
x_{11} & 1 & -x_{12} & 0 & x_{13} & 0 & \cdots & (-1)^n\\
-x_{21} & 1 & x_{22} & 1 & -x_{23} & 0 & \cdots & 0 \\
x_{31} & 0 & -x_{32} & 1 & x_{33} & 1 & \cdots & 0 \\
\vdots & \vdots &  \vdots &   \vdots &  \vdots & \vdots & \ddots & \vdots 
\end{array}
\right)    
\end{equation*}
In particular, we have:
\begin{itemize}
    \item 
    The dimension of the row space of $\Omega_n(e)$ is equal to $n-1$;
    \item Every $n-1 \times n-1$ minor of  $\Omega_n(e)$  is non-negative.
   \item   Plücker coordinates of the point $\mathcal{L}(e)$ correspond  to $n-1 \times n-1$ minors  of the matrix $\Omega'_n(e)$ obtained from $\Omega_n(e)$  by deleting the first row.
\end{itemize}
\end{theorem}

\begin{theorem} \textup{\cite[Theorem 5.8]{L}} \label{Image of the Lam's embedding}
The image of the set of  well-connected networks under the map $\mathcal{L}$    lies within the top cell of the totally non-negative Grassmannian $\mathrm{Gr}_{\geq 0}(n-1,2n)$, i.e. in the totally positive part  $\mathrm{Gr}_{> 0}(n-1, 2n):=\{X \in \mathrm{Gr}(n-1, 2n): \forall I \ \Delta_I(X)>0\}.$
\end{theorem}

We now describe the  combinatorial interpretation of the Plücker coordinates of the points  $\mathcal{L}(e)$.

\begin{definition} \label{groves, ncp}
A {\itshape grove} $F$ on $\Gamma$ is a spanning forest, that is, an acyclic subgraph that contains all vertices, such that each connected component $F_i\subset F$ contains at least one boundary vertex. The {\itshape boundary partition} $\sigma(F)$ is the set partition of the set $\{\bar{1},\bar{2},\ldots,\bar{n}\}$ that specifies which boundary vertices lie in the same connected component of $F$. Note that since $\Gamma$ is planar, $\sigma(F)$ must be a {\itshape non-crossing partition}. 

We often write set partitions in the form $\sigma=(\overline{a},\overline{b},\overline{c}|\overline{d},\overline{e}|\overline{f},\overline{g}|\overline{h})$. We also use a notation $[_{\overline{p}_1}^{\overline{q}_1}|_{\overline{p}_2}^{\overline{q}_2}|\ldots|_{\overline{p}_k}^{\overline{q}_k}|\overline{o}_1|\overline{o}_2,\ldots \overline{o}_l]$ meaning $(\overline{p}_1,\overline{q}_1|\overline{p}_2,\overline{q}_2|\ldots |\overline{p}_k,\overline{q}_k|\overline{o}_1|\overline{o}_2|\ldots|\overline{o}_l)$. We denote by $\mathcal{NC}_n$ the
set of non-crossing partitions on the set $\{\overline{1},\overline{2},\ldots,\overline{n}\}$.
Each non-crossing partition $\sigma$ on the set $\{\overline{1},\overline{2},\ldots,\overline{n}\}$ has a dual non-crossing partition on $\{\widetilde{1},\widetilde{2},\ldots,\widetilde{n}\}$
 where, by convention,  $\widetilde{i}$ lies between $\overline{i}$ and $\overline{i+1}$ and $\widetilde{n}$ lies between $\overline{n}$ and $\overline{1}$.

 We often consider the non-crossing partition of $[2n]$ defined by merging the initial partition $\sigma$ and its dual $\widetilde\sigma$, using the identification $\{1,2,\ldots,2n-1,2n\} =\{\overline{1},\widetilde{1},\overline{2},\widetilde{2},\ldots,\overline{n},\widetilde{n}\}$. We denote this partition $(\sigma|\widetilde\sigma)$, see Fig. \ref{concordance example}.

\end{definition}
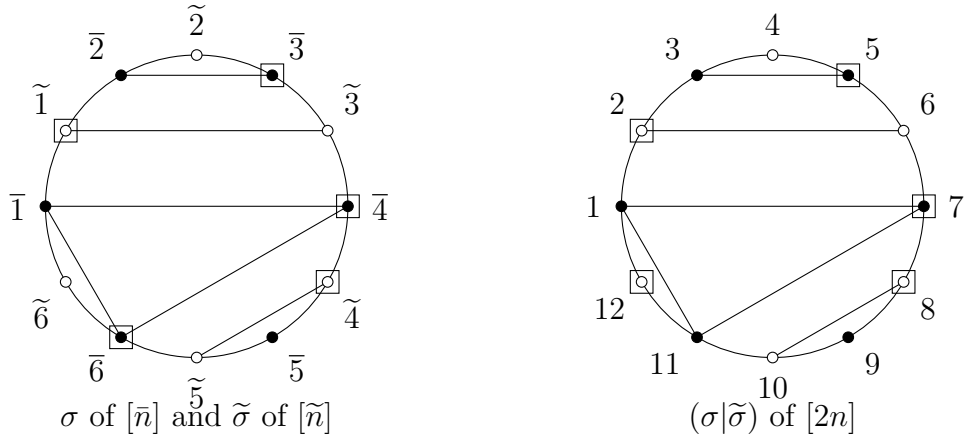
\begin{figure}[ht]
\centering
\begin{tikzpicture}
    \draw (0,0) circle (2);

    \draw (0:2) -- node [left] {}(180:2);
    \draw (30:2) -- node [left] {}(150:2);
    \draw (60:2) -- node [left] {}(120:2);
    \draw (0:2) -- node [left] {}(-120:2);
    \draw (-120:2) -- node [left] {}(180:2);
    \draw (-30:2) -- node [left] {}(-90:2);

    \filldraw[fill=black] (0:2) node [right=5pt] {$\overline{4}$} circle (2pt);
    \filldraw[fill=white] (30:2) node [above right=3pt] {$\widetilde{3}$} circle (2pt);
    \filldraw[fill=black] (60:2) node [above right=3pt] {$\overline{3}$} circle (2pt);
    \filldraw[fill=white] (90:2) node [above=3.5pt] {$\widetilde{2}$} circle (2pt);
    \filldraw[fill=black] (120:2) node [above left=3pt] {$\overline{2}$} circle (2pt);
    \filldraw[fill=white] (150:2) node [above left=3pt] {$\widetilde{1}$} circle (2pt);
    \filldraw[fill=black] (180:2) node [left=3.5pt] {$\overline{1}$} circle (2pt);
    \filldraw[fill=white] (-30:2) node [below right=3pt] {$\widetilde{4}$} circle (2pt);
    \filldraw[fill=black] (-60:2) node [below right=3pt] {$\overline{5}$} circle (2pt);
    \filldraw[fill=white] (-90:2) node [below=3.5pt] {$\widetilde{5}$} circle (2pt);
    \filldraw[fill=black] (-120:2) node [below left=3pt] {$\overline{6}$} circle (2pt);
    \filldraw[fill=white] (-150:2) node [below left=3pt] {$\widetilde{6}$} circle (2pt);

    \draw (0:2) +(-0.15,-0.15) rectangle +(0.15,0.15);
    \draw (-30:2) +(-0.15,-0.15) rectangle +(0.15,0.15);
    \draw (-120:2) +(-0.15,-0.15) rectangle +(0.15,0.15);
    \draw (150:2) +(-0.15,-0.15) rectangle +(0.15,0.15);
    \draw (60:2) +(-0.15,-0.15) rectangle +(0.15,0.15);

    \draw (0,-2.8) node {$\sigma$ of $[\bar n]$ and $\widetilde{\sigma}$ of $[\widetilde{n}]$};
\end{tikzpicture} \phantom{aaaaaaaaaa}
\begin{tikzpicture}
    \draw (0,0) circle (2);

    \draw (0:2) -- node [left] {}(180:2);
    \draw (30:2) -- node [left] {}(150:2);
    \draw (60:2) -- node [left] {}(120:2);
    \draw (0:2) -- node [left] {}(-120:2);
    \draw (-120:2) -- node [left] {}(180:2);
    \draw (-30:2) -- node [left] {}(-90:2);

    \filldraw[fill=black] (0:2) node [right=5pt] {$7$} circle (2pt);
    \filldraw[fill=white] (30:2) node [above right=3pt] {$6$} circle (2pt);
    \filldraw[fill=black] (60:2) node [above right=3pt] {$5$} circle (2pt);
    \filldraw[fill=white] (90:2) node [above=3.5pt] {$4$} circle (2pt);
    \filldraw[fill=black] (120:2) node [above left=3pt] {$3$} circle (2pt);
    \filldraw[fill=white] (150:2) node [above left=3pt] {$2$} circle (2pt);
    \filldraw[fill=black] (180:2) node [left=3.5pt] {$1$} circle (2pt);
    \filldraw[fill=white] (-30:2) node [below right=3pt] {$8$} circle (2pt);
    \filldraw[fill=black] (-60:2) node [below right=3pt] {$9$} circle (2pt);
    \filldraw[fill=white] (-90:2) node [below=3.5pt] {$10$} circle (2pt);
    \filldraw[fill=black] (-120:2) node [below left=3pt] {$11$} circle (2pt);
    \filldraw[fill=white] (-150:2) node [below left=3pt] {$12$} circle (2pt);

    \draw (0:2) +(-0.15,-0.15) rectangle +(0.15,0.15);
    \draw (60:2) +(-0.15,-0.15) rectangle +(0.15,0.15);
    \draw (150:2) +(-0.15,-0.15) rectangle +(0.15,0.15);
    \draw (-150:2) +(-0.15,-0.15) rectangle +(0.15,0.15);
    \draw (-30:2) +(-0.15,-0.15) rectangle +(0.15,0.15);

    \draw (0,-2.8) node {$(\sigma|\widetilde{\sigma})$ of $[2n]$};
\end{tikzpicture}
    \caption{Non-crossing partition $\sigma=(\bar1,\bar4,\bar6|\bar2,\bar3|\bar5)$ with its dual non-crossing partition  $\widetilde{\sigma}=(\tilde1,\tilde3|\tilde2|\tilde4,\tilde5|\tilde6)$ and merged partition $(\sigma|\widetilde{\sigma})$.}
    \label{concordance example}
\end{figure}

\begin{definition}
For a non-crossing partition $\sigma$ the {\itshape grove measurement} associated with $\sigma$ is defined by
$$L_{\sigma}:=\sum\limits_{F|\sigma(F)=\sigma}wt(F),$$
where the sum is over all groves with boundary partition $\sigma$, and $wt(F)$ is the product of the weights of the edges in $F$. We denote by $L_{unc}$ the grove measurement $L_{\bar1|\bar2|\ldots|\bar n}$. For $\sigma=(\bar p_1,\bar q_1|\bar p_2,\bar q_2|\ldots |\bar p_k,\bar q_k|\bar o_1|\bar o_2|\ldots|\bar o_l)$ we also write $L_\sigma$ as $L[_{\bar p_1}^{\bar q_1}|_{\bar p_2}^{\bar q_2}|\ldots|_{\bar p_k}^{\bar q_k}|\bar o_1|\bar o_2,\ldots, \bar o_l]$.
\end{definition}

\begin{definition} \label{def-ordcon}
We say that an $(n-1)$-element subset $I\subset \{1,\ldots,2n\}$ is {\itshape concordant} with a non-crossing partition $\sigma$ if each part of $\sigma$ and each part of the dual partition $\widetilde{\sigma}$ contains exactly one element not in $I$. In this situation, we also say that $\sigma$ or $(\sigma,\widetilde{\sigma})$ is concordant with $I$, see Fig. \ref{concordance example} for an example. 
\end{definition}

Let $e(\Gamma, \omega) \in E_n$ and define boundary measurements $\Delta_I$ for all $I\subset\{1,\ldots,2n\},\, |I|=n-1$ as
\begin{equation} \label{form:pl-gr}
\Delta_I:=\sum\limits_{(\sigma,I)}L_{\sigma},
\end{equation}
where the sum is over all $\sigma$ that are concordant with $I$.

\begin{theorem} \textup{\cite[Theorem 5.8]{L}, \cite[Theorem 3.3]{BGKT} \label{Emb}}
The collection of $\Delta_I$, defined in \eqref{form:pl-gr}, viewed as homogeneous Plücker coordinates, defines a point in $\mathrm{Gr}_{\geq 0}(n-1,2n)$. Thus, we obtain an injective map $\mathcal{Y}: E_n \to \mathrm{Gr}_{\geq 0}(n-1,2n)$.

Moreover, the maps $\mathcal{L}$ and $\mathcal{Y}$ agree pointwise, and 

$$\Delta_I\bigr(\Omega'_n(e)\bigl)=\frac{\Delta_I}{L_{unc}}. $$

\end{theorem}
\begin{remark}
Recall that we consider electrical networks up to electrical transformations. It can be verified that all grove measurements $L_{\sigma}$ are multiplied by the same nonzero factor under these transformations, see  \cite[Proposition $4.4$]{L}. 
Consequently, the projective point defined by the collection of $\Delta_I$ is
invariant under electrical transformations.
\end{remark}
We call the map $\mathcal{Y}: E_n \to \mathrm{Gr}_{\geq 0}(n-1,2n)$ {\itshape Lam's embedding}. By Theorem \ref{Emb} we will also use the notation $\mathcal{L}$ for Lam's embedding.

\begin{example} \label{bulcon-ex}
  By Definition \ref{def-ordcon} for any $n$  the following holds:
  \begin{itemize}
      \item For any $I \subset [2n], \ |I|=n-1$ containing only even indices we have  $\Delta_{I}=L_{unc}$, in particular, $\Delta_{246\dots2n-2}=L_{unc};$
      \item For any $I \subset [2n], \ |I|=n-1$ containing only odd indices we have  $\Delta_{I}=L_{ \overline{1  2 3}\ldots  \overline{n}},$ in particular, $\Delta_{135\dots 2n-3}=L_{\overline{1}\overline{2}\overline{3}\dots \overline{n}}.$
  \end{itemize} 
\end{example}

\subsection{The image of Lam's embedding}
In this section we describe the image of Lam's embedding using the geometry encoded by combinatorics of concordance subsets. 
\begin{definition} \label{def:concordance}
Let $I$ be a $(n-1)$-element subset of $\{1,\ldots,2n\}$ and $\sigma\in \mathcal{NC}_n$.

Define a $(\binom{2n}{n-1}\times C_n)$-matrix $A_n=(a_{I\sigma})$, where $I$ is a $(n-1)$-element subset of $\{1,\ldots,2n\}$, $\sigma\in \mathcal{NC}_n$ and $C_n$ is the $n$-th Catalan number as follows:

$$a_{I\sigma}=\begin{cases}
1, \text{\ if\ }  \sigma\ \text{is concordant with}\ I; \\
0, \text{\ otherwise.}
\end{cases}$$
\end{definition}

Let the {\itshape concordance space} $H$ be the column space of $A_n$, it is a subspace of $\bigwedge^{n-1}\mathbb{R}^{2n}$ and $\mathrm{dim}\, H=C_n$. Denote by $\mathbb PH$ the projectivization of $H$. 
\begin{definition}
    A {\itshape cactus network} is an electrical network, whose conductivities are allowed to be zero or infinite.  Due to Ohm's law, setting an edge conductance to be zero (infinite) means that we can delete (respectively contract) this edge. Denote by $\overline{E}_n$ the set of cactus networks with $n$ boundary vertices up to electrical equivalence. 

     The identity \eqref{form:pl-gr} extends naturally to any cactus electrical network.
    \end{definition}

\begin{theorem} \textup{\cite[Theorem 5.8]{L}} \label{Image of the Lam's embedding surjectivity}
The extension of Lam's embedding to $\overline{E}_n$, $\mathcal{L}:\overline{E}_n\to \mathrm{Gr}_{\ge0}(n-1,2n)$ has the image $\mathrm{Gr}_{\ge0}(n-1,2n)\cap\mathbb PH$. This map is a bijection from $\overline{E}_n$ to $\mathrm{Gr}_{\ge0}(n-1,2n)\cap\mathbb PH$. 

Moreover, the image of well-connected networks under Lam's embedding  is exactly  $\mathrm{Gr}_{> 0}(n-1,2n)\cap\mathbb P H$.
\end{theorem}

Denote by $e_i$ the standard basis vectors of $\mathbb R^{2n}$. For $I=(i_1,i_2,\ldots,i_{n-1})$ we denote by $e_I$ the standard vector $e_{i_1}\wedge e_{i_2}\wedge\ldots\wedge e_{i_{n-1}}$. Thus, the vectors of the form $e_I=e_{i_1}\wedge e_{i_2}\wedge \ldots \wedge e_{i_{n-1}}$, where  $i_1<i_2<\ldots<i_{n-1}$ form a basis of the space $\bigwedge^{n-1}\mathbb{R}^{2n}$; and vectors $w_\sigma=\sum\limits_I a_{I \sigma}e_I\in\bigwedge^{n-1}\mathbb{R}^{2n}$ form a basis of the space $H$. In these notations Theorem \ref{Image of the Lam's embedding surjectivity} can be rewritten as follows. 

\begin{theorem} \textup{\cite[Theorem 5.10]{L}} \label{th: the_first_ph}
Any point $X\in \mathrm{Gr}(n-1,2n)\cap\mathbb P H$ has the following form: 
\begin{equation}\label{eq:Grovesum}
    X=\sum _{\sigma \in \mathcal{NC}_n}L_{\sigma}w_\sigma.
\end{equation}

Moreover, if  $X \in  \mathrm{Gr}_{>0}(n-1,2n)\cap\mathbb P H $, then $L_{\sigma}>0$ for any $\sigma$. 
\end{theorem}

The space $\mathbb P H$ has a more geometric description (see \cite[Section 5.2]{BGKT}), which helps relate the combinatorics of the concordance subsets to the geometry of the Lagrangian Grassmannian.
\begin{theorem}  \textup{\cite[Theorem 1.6]{CGS}, \cite[Lemma 5.8, Theorem 5.9]{BGKT}, \cite[Theorem 6.3]{GLX}} \label{theorem:to def of x}
   The complex projective variety  $\mathrm{Gr}(n-1,2n)\cap\mathbb P H$ is isomorphic to the Lagrangian Grassmannian $\mathrm{LG}(n-1,2n-2).$ 
\end{theorem}

\begin{remark}\label{rem:Lunc}
    Given a point $X\in\mathrm{Gr}(n-1,2n)\cap\mathbb PH$ we can represent it as $X=\sum_I\Delta_Ie_I$ and as $X=\sum_\sigma L_\sigma w_\sigma$. For every $I\in\binom{[2n]}{n-1}$ such that $I\subset\{2,4,\ldots,2n\}$ we have $\Delta_I=L_{unc}$. One can obtain this by substituting one expression of $X$ into the other and using the definition of concordance vectors.
\end{remark}

\begin{remark}\label{rem:Lunc2}
    By Theorem \ref{Image of the Lam's embedding surjectivity} $\mathcal{L}(\overline{E}_n)=\mathrm{Gr}_{\ge0}(n-1,2n)\cap\mathbb PH$. For $X\in \mathrm{Gr}_{\ge0}(n-1,2n)\cap\mathbb PH$ the condition that there exists $e\in E_n$ such that $\mathcal{L}(e)=X$ is equivalent to $L_{unc}\ne 0$, where $L_{unc}$ is the coefficient of $w_{\bar1|\bar2|\bar3|\ldots|\bar n}$ in $X=\sum_\sigma L_\sigma w_\sigma$. In other words, the condition $L_{unc}\ne 0$ cuts out the noncompactified space of electrical networks from the space of cactus networks.
\end{remark}

\section{Circular minors, groves, and Plücker coordinates}
\label{sub:circ}
\subsection{Circular minors} \label{sub:circ-min}
The main goal of this section is to relate certain minors of a response matrix of an electrical network $e$ to the Plücker coordinates of the point $\mathcal{L}(e) $ defined by Lam's embedding. 

\begin{definition}
Let $P=(\overline{p}_1,\ldots,\overline{p}_k)$ and $Q=(\overline{q}_1,\ldots,\overline{q}_k)$ be disjoint ordered $k$-tuples of boundary nodes arranged on a circle. We say that $(P;Q)=(\overline{p}_1,\ldots,\overline{p}_k;\overline{q}_1,\ldots,\overline{q}_k)$ is a {\itshape circular pair} if $\overline{p}_1,\ldots, \overline{p}_k,\overline{q}_k,\ldots,\overline{q}_1$ occur in clockwise order around the circle, i.e. lie on the circle in the specified order modulo $n$.
 A {\itshape contiguous circular pair} is a circular pair of the form\\ $(\overline{a},\overline{a+1},\ldots,\overline{a+y-1};\overline{b+y-1},\ldots,\overline{b+1},\overline{b})$. A {\itshape semicontiguous circular pair} is a circular pair $(P;Q)$ where only one of $P$ or $Q$ is contiguous, i.e. the arc $P$ or $Q$ does not contain additional nodes. Let $(P;Q)$ be a circular pair and let  $M = (m_{ij}), 1\leq i,j\leq n$, be an arbitrary matrix. The {\itshape circular minor} associated with a circular pair $(P;Q)$ is the determinant of the circular submatrix $M(P;Q)$ whose rows are taken in the order $\overline{p}_1,\ldots,\overline{p}_k$ and whose columns are taken in the order $\overline{q}_1,\ldots,\overline{q}_k$ of the matrix $M$. A {\itshape contiguous} (resp. {\itshape semicontiguous}) {\itshape circular minor} is a circular minor that corresponds to a contiguous (resp. semicontiguous) circular pair.
 We use the term \emph{minor} both for a submatrix and for its determinant whenever it does not lead to confusion.
\end{definition}

\begin{example}
Consider a matrix
\begin{equation*} 
M=\left(\begin{matrix}
131 & -66 & -12 & -53  \\
-66  & 76 & -8 & -2  \\
-12 & -8 & 24 & -4  \\
-53 & -2 & -4 & 59    
\end{matrix}\right)
\end{equation*}    
and  circular pairs $(\overline{2}, \overline{3}; \overline{1}, \overline{4})$ and $(\overline{1}, \overline{2}; \overline{4}, \overline{3})$ then we have the following positive circular minors:
\begin{equation*} 
\det M(\overline{2}, \overline{3};\overline{1}, \overline{4})=\begin{vmatrix}
-66 & -2   \\
-12 & -4      
\end{vmatrix}, \
\det M(\overline{1}, \overline{2};\overline{4}, \overline{3})=\begin{vmatrix}
-53 & -12   \\
-2 & -8      
\end{vmatrix}.
\end{equation*}
\end{example}
Define {\itshape circular totally non-negative} matrix to be a $n\times n$ matrix $M$  that satisfies the following condition: for every $k$ if $M(P;Q)$ is  a circular $k \times k$ submatrix, then 
$$(-1)^k\det M(P;Q) \geq 0.$$ 
The next theorem provides a characterization of the set of response matrices:
\begin{theorem}\textup{\cite[Theorem 3 and Theorem 4.2]{CIM}} \label{Set of response matrices all network well conected} 
  A matrix $M \in Mat_{n \times n}(\mathbb{R})$ is a response matrix of an electrical network if and only if the following holds:
  \begin{itemize}
      \item $M$ is symmetric with zero row sums;
      \item   $M$ is circular totally positive, i.e. if $M(P;Q)$ is  a circular $k \times k$ submatrix, then $(-1)^k\det M(P;Q) \ge 0.$
  \end{itemize}
  Moreover, the strict inequality holds if and only if the electrical network is well-connected.
\end{theorem}

We now state an analogue of the all-minor matrix-tree Theorem (see \cite{Cha}, \cite{Moon} and \cite{Wa}) for response matrices.  
\begin{theorem} \textup{\cite[Theorem 8.1]{KW}}\label{minors via groves} 
  Consider an electrical network $e(\Gamma, \omega) \in E_n$ and its response matrix $M=M_R$. Let $O, R, S, T$ be a partition of the boundary vertices  with $|R|=|S|=k$. Then,

    $$(-1)^{|S\cup T|}\det M_{S\cup T}^{R\cup T}=(-1)^{|T|}\sum\limits_{\rho\in S_k}(-1)^{sgn(\rho)}\frac{L[_{\overline{r}_1}^{\overline{s}_{\rho (1)}}|\cdots|_{\overline{r}_k}^{\overline{s}_{\rho(k)}}|\overline{o}_1|\cdots|\overline{o}_l]}{L_{unc}}.$$
    Here the vertices in $T$ are treated as internal vertices in the corresponding grove measurements.
\end{theorem}

\begin{lemma} \label{lemma: bijection}
   
     Let $k\ge 1$. Given a contiguous set $P=\{\overline{p}_1,\ldots,\overline{p}_k\}$, the following objects are in a bijection:
     \begin{enumerate}
     \item The set of contiguous circular pairs $(P;Q)$, where $P=\{\overline{p}_1,\ldots,\overline{p}_k\}$ and $Q=\{\overline{q}_1,\ldots,\overline{q}_k\}$.
     
     \item The set of non-crossing partitions $[_{\overline{p}_1}^{\overline{q}_1}|\cdots|_{\overline{p}_k}^{\overline{q}_k}|\overline{o}_1|\cdots|\overline{o}_l]$, where $\{\overline{o}_1,\ldots,\overline{o}_l\}:=[\bar n]\setminus (P\cup Q)$.
     
     \item The subsets $I\subset[2n],\ |I|=n-1$ of the following form: $I=J_l\cup \widehat P\cup J_r$, where $\widehat P=\{2p_1-1,2p_1,\ldots,2p_k-2,2p_k-1\}$; the subsets $J_l$ and $J_r$ consist of even integers; $\max J_l=2p_1-2$; $\min J_r=2p_k$ (see Fig. \ref{fig: induction in an area}). In other words, $I$ is of the form:
     \end{enumerate}

     \begin{equation} \label{eq: I for lemma on Lam minors}
    I=\{2p_1-1,2p_1,\ldots,2p_k-2,2p_k-1\}\cup\{2p_k,2p_k+2,\ldots,2q_k-4\}\cup
\end{equation}
$$\cup\{2p_1-2,2p_1-4,\ldots,2q_1+2\}.$$

    Informally, $I$ consists of a contiguous segment of integers $\widehat P=\{2p_1-1,2p_1,\ldots,2p_k-2,2p_k-1\}$ possibly surrounded by contiguous segments of even integers $J_l$ and $J_r$; moreover, these segments of even indices are attached to the segment $\widehat P$. See Fig. \ref{I which is concordant with central circular pair} for an illustration.
\end{lemma}

\begin{proof}
    Note that a subset $I$ that satisfies the assumptions of the lemma is concordant with the partition $[_{\overline{p}_1}^{\overline{q}_1}|\cdots|_{\overline{p}_k}^{\overline{q}_k}|\overline{o}_1|\cdots|\overline{o}_l]$, where the set $\{\overline{q}_1,\ldots,\overline{q}_k\}$ is the set of all odd numbers on the arc $(\max J_r+2,\min J_l-2)$.

    We now prove that $I$ is concordant only with $[_{\overline{p}_1}^{\overline{q}_1}|\cdots|_{\overline{p}_k}^{\overline{q}_k}|\overline{o}_1|\cdots|\overline{o}_l]$. Suppose that some $\sigma$ is concordant with $I$.

    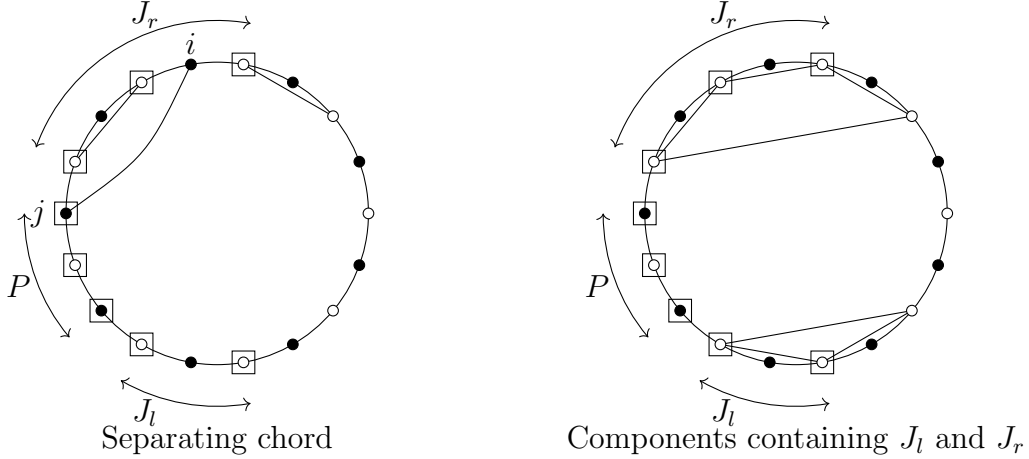
\begin{figure}[ht]
\centering
\begin{tikzpicture}
    \draw (0,0) circle (2);

    \draw (40:2) -- node [left] {}(80:2);
    \draw (160:2) -- node [left] {}(120:2);


    \draw (100:2) .. controls (140:1.2) .. (180:2);

    \draw (80:2) +(-0.15,-0.15) rectangle +(0.15,0.15);
    \draw (120:2) +(-0.15,-0.15) rectangle +(0.15,0.15);
    \draw (160:2) +(-0.15,-0.15) rectangle +(0.15,0.15);
    \draw (200:2) +(-0.15,-0.15) rectangle +(0.15,0.15);
    \draw (240:2) +(-0.15,-0.15) rectangle +(0.15,0.15);
    \draw (280:2) +(-0.15,-0.15) rectangle +(0.15,0.15);
    \draw (180:2) +(-0.15,-0.15) rectangle +(0.15,0.15);
    \draw (220:2) +(-0.15,-0.15) rectangle +(0.15,0.15);
    \draw [<->] (80:2.55) arc (80:160:2.55);
    \draw [<->] (180:2.55) arc (180:220:2.55);
    \draw [<->] (240:2.55) arc (240:280:2.55);

    
    \filldraw[fill=white] (40:2) node [right] {} circle (2pt);
    \filldraw[fill=white] (80:2) node [right] {} circle (2pt);
    \filldraw[fill=white] (120:2) node [right] {} circle (2pt);
    \filldraw[fill=white] (160:2) node [right] {} circle (2pt);
    \filldraw[fill=white] (200:2) node [right] {} circle (2pt);
    \filldraw[fill=white] (240:2) node [right] {} circle (2pt);
    \filldraw[fill=white] (280:2) node [right] {} circle (2pt);
    \filldraw[fill=white] (320:2) node [right] {} circle (2pt);
    \filldraw[fill=white] (360:2) node [right] {} circle (2pt);

    \filldraw[fill=black] (20:2) node [right] {} circle (2pt);
    \filldraw[fill=black] (60:2) node [right] {} circle (2pt);
    \filldraw[fill=black] (100:2) node [above] {$i$} circle (2pt);
    \filldraw[fill=black] (140:2) node [right] {} circle (2pt);
    \filldraw[fill=black] (180:2) node [left=3.5pt] {$j$} circle (2pt);
    \filldraw[fill=black] (220:2) node [right] {} circle (2pt);
    \filldraw[fill=black] (260:2) node [right] {} circle (2pt);
    \filldraw[fill=black] (300:2) node [right] {} circle (2pt);
    \filldraw[fill=black] (340:2) node [right] {} circle (2pt);

    \draw (200:2.8) node {$P$};
    \draw (110:2.8) node {$J_r$};
    \draw (250:2.8) node {$J_l$};
    
    \draw (0,-3) node {Separating chord};
\end{tikzpicture}\phantom{aaaaaaaaaa}
\begin{tikzpicture}
    \draw (0,0) circle (2);

    \draw (40:2) -- node [left] {}(80:2);
    \draw (80:2) -- node [left] {}(120:2);
    \draw (160:2) -- node [left] {}(120:2);
    \draw (160:2) -- node [left] {}(40:2);

    \draw (320:2) -- node [left] {}(280:2);
    \draw (280:2) -- node [left] {}(240:2);
    \draw (240:2) -- node [left] {}(320:2);

    
    \filldraw[fill=white] (40:2) node [right] {} circle (2pt);
    \filldraw[fill=white] (80:2) node [right] {} circle (2pt);
    \filldraw[fill=white] (120:2) node [right] {} circle (2pt);
    \filldraw[fill=white] (160:2) node [right] {} circle (2pt);
    \filldraw[fill=white] (200:2) node [right] {} circle (2pt);
    \filldraw[fill=white] (240:2) node [right] {} circle (2pt);
    \filldraw[fill=white] (280:2) node [right] {} circle (2pt);
    \filldraw[fill=white] (320:2) node [right] {} circle (2pt);
    \filldraw[fill=white] (360:2) node [right] {} circle (2pt);

    \filldraw[fill=black] (20:2) node [right] {} circle (2pt);
    \filldraw[fill=black] (60:2) node [right] {} circle (2pt);
    \filldraw[fill=black] (100:2) node [right] {} circle (2pt);
    \filldraw[fill=black] (140:2) node [right] {} circle (2pt);
    \filldraw[fill=black] (180:2) node [right] {} circle (2pt);
    \filldraw[fill=black] (220:2) node [right] {} circle (2pt);
    \filldraw[fill=black] (260:2) node [right] {} circle (2pt);
    \filldraw[fill=black] (300:2) node [right] {} circle (2pt);
    \filldraw[fill=black] (340:2) node [right] {} circle (2pt);
    \draw [<->] (80:2.55) arc (80:160:2.55);
    \draw [<->] (180:2.55) arc (180:220:2.55);
    \draw [<->] (240:2.55) arc (240:280:2.55);


    \draw (80:2) +(-0.15,-0.15) rectangle +(0.15,0.15);
    \draw (120:2) +(-0.15,-0.15) rectangle +(0.15,0.15);
    \draw (160:2) +(-0.15,-0.15) rectangle +(0.15,0.15);
    \draw (200:2) +(-0.15,-0.15) rectangle +(0.15,0.15);
    \draw (240:2) +(-0.15,-0.15) rectangle +(0.15,0.15);
    \draw (280:2) +(-0.15,-0.15) rectangle +(0.15,0.15);
    \draw (180:2) +(-0.15,-0.15) rectangle +(0.15,0.15);
    \draw (220:2) +(-0.15,-0.15) rectangle +(0.15,0.15);

    \draw (200:2.8) node {$P$};
    \draw (110:2.8) node {$J_r$};
    \draw (250:2.8) node {$J_l$};
    
    \draw (0,-3) node {Components containing $J_l$ and $J_r$};
\end{tikzpicture}
    \caption{$\sigma$ is concordant with $I$.}
    \label{fig: induction in an area}
\end{figure}

\textbf{Step 1.} First we prove that all elements of $J_r$ belong to the same component of the partition $\sigma$.

    Assume the converse, then there exists an odd $i\in[\min J_r;\max J_r]$ separating two components containing elements of $J_r$. It follows that $i$ does not belong to a component of size $1$. By the form of $I$, $i\notin I$. Putting these facts together, we conclude that since $\sigma$ is concordant with $I$, $i$ lies in a component that contains some elements of $\widehat P$. Among all chords $(i,j)$ connecting $i$ with element $j\in \widehat P$, choose the one with the largest $j$. This chord cuts the circle into two regions: one containing $\min J_r$ and the other. Note that by the construction, all even elements of the first region belong to $I$. Then $\sigma$ is not concordant with $I$ because $I$ contains a whole part of $\sigma$, see Fig. \ref{fig: induction in an area}.

    \textbf{Step 2.} Similarly, all elements of $J_l$ belong to the same component of $\sigma$.

    \textbf{Step 3.}
    Since $\mathrm{min}\,J_r=2p_k$ and $\mathrm{max}\,J_l=2p_1-2$, we have that the elements of the arc $P$ are in bijection with elements of the arc $(\max J_r+3;\min J_l-3)$.
    It follows that the only way to match the elements of these two arcs is by parallel series of chords as in $[_{\overline{p}_1}^{\overline{q}_1}|\cdots|_{\overline{p}_k}^{\overline{q}_k}|\overline{o}_1|\cdots|\overline{o}_l]$.
Thus, the subset $I$ is concordant only with the partition $[_{\overline{p}_1}^{\overline{q}_1}|\cdots|_{\overline{p}_k}^{\overline{q}_k}|\overline{o}_1|\cdots|\overline{o}_l]$, so we proved the bijection $(2)\leftrightarrow (3)$.


Finally, note that the partition $[_{\overline{p}_1}^{\overline{q}_1}|\cdots|_{\overline{p}_k}^{\overline{q}_k}|\overline{o}_1|\cdots|\overline{o}_l]$ defines the contiguous circular pair uniquely by restriction to the set of odd indices that belong to components of size greater than $1$, which proves the bijection $(1)\leftrightarrow (2)$.
\end{proof}

\begin{remark} \label{rem: empty circular pair}
    The correspondence from Lemma \ref{lemma: bijection} can be naturally extended to the case $k=0$ but is no longer a bijection. In that case the empty circular pair $(\emptyset;\emptyset)$ corresponds to the non-crossing partition $unc:=(\bar 1|\bar 2|\ldots|\bar n)$ and to all subsets $I\subset\{2,4,\ldots,2n\},\ |I|=n-1$.
\end{remark}

\begin{proposition} \label{circular minors via Plücker coordinates 0} 
Let $e\in E_n$, let $M^Q_P$ be a contiguous minor
with $(P;Q)=(\overline{p}_1,\ldots,\overline{p}_k;\overline{q}_1,\ldots,\overline{q}_k)$ 
and let a set $O$ be $\{\overline{o}_1,\ldots,\overline{o}_l\}:=[\bar n]\setminus (P\cup Q)$. Let $I\subset [2n]$ be given by \eqref{eq: I for lemma on Lam minors}. Then,


$$(-1)^k\det M_P^Q=\frac{\Delta_I}{L_{unc}} = \Delta_I(\Omega'(e)).$$ 
Equivalently
$$
\Delta_I = L[_{\overline{p}_1}^{\overline{q}_1}|\cdots|_{\overline{p}_k}^{\overline{q}_k}|\overline{o}_1|\cdots|\overline{o}_l].
$$
\end{proposition}

\begin{proof}

Since the pair $(P;Q)$ is circular, i.e. $\overline{p}_1,\ldots, \overline{p}_k,\overline{q}_k,\ldots,\overline{q}_1$ are in cyclic clockwise order, $P\cap Q=\emptyset$. Thus we can apply Theorem \ref{minors via groves} with $R=P,\ S=Q,\ T=\emptyset,\ O=[\bar n]\setminus (P\cup Q)=\{\overline{o}_1,\ldots,\overline{o}_l\}$:
\begin{equation}
\label{eq: Plücker=Grove}
(-1)^k\det M_P^Q=(-1)^0\sum\limits_{\rho\in S_k}(-1)^{sgn(\rho)}\frac{L[_{\overline{r}_1}^{\overline{s}_{\rho (1)}}|\cdots|_{\overline{r}_k}^{\overline{s}_{\rho(k)}}|\overline{o}_1|\cdots|\overline{o}_l]}{L_{unc}}.
\end{equation}

Since \(\Gamma\) is planar, a grove on \(\Gamma\) can induce only a
non-crossing boundary partition. Therefore, in the right-hand side of \eqref{eq: Plücker=Grove}, all terms corresponding to crossing pairings between \(P\) and
\(Q\) vanish.

Since \((P;Q)\) is a circular pair, the elements $\overline{p}_1,\ldots,\overline{p}_k,\overline{q}_k,\ldots,\overline{q}_1$
occur in this clockwise cyclic order. Hence there is exactly one matching
between \(P\) and \(Q\) that gives a non-crossing partition, namely
$
\overline{p}_i \leftrightarrow \overline{q}_i,\; i=1,\ldots,k.
$
Thus the only
non-zero term in \eqref{eq: Plücker=Grove} is the one corresponding to \(\rho=\mathrm{id}\),
and we obtain
\begin{equation} \label{formula:from-min-to-groves}
    (-1)^k\det M_P^Q=\frac{L[_{\overline{p}_1}^{\overline{q}_1}|\cdots|_{\overline{p}_k}^{\overline{q}_k}|\overline{o}_1|\cdots|\overline{o}_l]}{L_{unc}}.
\end{equation}

Now let \(I\subset[2n]\), \(|I|=n-1\), be the subset defined by \eqref{eq: I for lemma on Lam minors}.
By Lemma \ref{lemma: bijection}, the only non-crossing partition concordant with \(I\) is
$
\sigma=
\left[_{\overline{p}_1}^{\overline{q}_1}|\cdots|_{\overline{p}_k}^{\overline{q}_k}|\overline{o}_1|\cdots|\overline{o}_l
\right].
$
Therefore, by the definition of the boundary measurement \(\Delta_I\),
\[
\Delta_I
=
\sum_{\tau\text{ concordant with }I} L_\tau
=
L[_{\overline{p}_1}^{\overline{q}_1}|\cdots|_{\overline{p}_k}^{\overline{q}_k}|\overline{o}_1|\cdots|\overline{o}_l].
\]

Combining this identity with \eqref{formula:from-min-to-groves}, we get
\[
(-1)^k\det M^Q_P
=
\frac{
L[_{\overline{p}_1}^{\overline{q}_1}|\cdots|_{\overline{p}_k}^{\overline{q}_k}|\overline{o}_1|\cdots|\overline{o}_l]}
{L_{\mathrm{unc}}}
=
\frac{\Delta_I}{L_{\mathrm{unc}}}.
\]
Finally, by Theorem \ref{Emb},
\[
\frac{\Delta_I}{L_{\mathrm{unc}}}
=
\Delta_I(\Omega'_n(e)).
\]
This proves the claim.

\begin{figure}[ht]
\centering
\begin{tikzpicture}[scale=0.8]
    \draw (0,0) circle (3);

    \draw (150:3) -- node [left] {}(30:3);
    \draw (165:3) -- node [left] {}(15:3);
    \draw (210:3) -- node [left] {}(-30:3);
    \draw (180:3) -- node [left] {}(0:3);
    
    \draw (135:3) -- node [left] {}(105:3);
    \draw (105:3) -- node [left] {}(75:3);
    \draw (75:3) -- node [left] {}(45:3);
    \draw (135:3) -- node [left] {}(45:3);

    \draw (-135:3) -- node [left] {}(-105:3);
    \draw (-105:3) -- node [left] {}(-75:3);
    \draw (-75:3) -- node [left] {}(-45:3);
    \draw (-135:3) -- node [left] {}(-45:3);

    \filldraw[fill=black] (0:3) node [right] {$\ \overline{q}_2$} circle (2pt);
    \filldraw[fill=black] (30:3) node [right] {$\ \overline{q}_k$} circle (2pt);
    \filldraw[fill=black] (60:3) node [above] {} circle (2pt);
    \filldraw[fill=black] (90:3) node [above] {} circle (2pt);
    \filldraw[fill=black] (120:3) node [left] {} circle (2pt);
    \filldraw[fill=black] (150:3) node [left] {$\overline{p}_k\ $} circle (2pt);
    \filldraw[fill=black] (180:3) node [left] {$\overline{p}_2\ $} circle (2pt);
    \filldraw[fill=black] (210:3) node [left] {$\overline{p}_1\ $} circle (2pt);
    \filldraw[fill=black] (240:3) node [below] {} circle (2pt);
    \filldraw[fill=black] (270:3) node [left] {} circle (2pt);
    \filldraw[fill=black] (300:3) node [left] {} circle (2pt);
    \filldraw[fill=black] (330:3) node [right] {$\ \overline{q}_1$} circle (2pt);

    \filldraw[fill=white] (15:3) node [right] {} circle (2pt);
    \filldraw[fill=white] (45:3) node [right] {} circle (2pt);
    \filldraw[fill=white] (75:3) node [above] {} circle (2pt);
    \filldraw[fill=white] (105:3) node [above] {} circle (2pt);
    \filldraw[fill=white] (135:3) node [left] {} circle (2pt);
    \filldraw[fill=white] (165:3) node [left] {} circle (2pt);
    \filldraw[fill=white] (195:3) node [right] {} circle (2pt);
    \filldraw[fill=white] (225:3) node [below] {} circle (2pt);
    \filldraw[fill=white] (255:3) node [below] {} circle (2pt);
    \filldraw[fill=white] (285:3) node [left] {} circle (2pt);
    \filldraw[fill=white] (315:3) node [left] {} circle (2pt);
    \filldraw[fill=white] (345:3) node [left] {} circle (2pt);

    \filldraw[fill=black] (0,-0.75) node [left] {} circle (2pt);
    \filldraw[fill=black] (-0.5,-0.75) node [left] {} circle (2pt);
    \filldraw[fill=black] (0.5,-0.75) node [left] {} circle (2pt);

    \draw (75:3) +(-0.15,-0.15) rectangle +(0.15,0.15);
    \draw (105:3) +(-0.15,-0.15) rectangle +(0.15,0.15);
    \draw (135:3) +(-0.15,-0.15) rectangle +(0.15,0.15);
    \draw (150:3) +(-0.15,-0.15) rectangle +(0.15,0.15);
    \draw (165:3) +(-0.15,-0.15) rectangle +(0.15,0.15);
    \draw (180:3) +(-0.15,-0.15) rectangle +(0.15,0.15);
    \draw (195:3) +(-0.15,-0.15) rectangle +(0.15,0.15);
    \draw (210:3) +(-0.15,-0.15) rectangle +(0.15,0.15);
    \draw (225:3) +(-0.15,-0.15) rectangle +(0.15,0.15);
    \draw (255:3) +(-0.15,-0.15) rectangle +(0.15,0.15);
    \draw (285:3) +(-0.15,-0.15) rectangle +(0.15,0.15);

    \draw (0,-2.8) node {};
\end{tikzpicture}
    \caption{The non-crossing partition $L[_{\overline{p}_1}^{\overline{q}_1}|\cdots|_{\overline{p}_k}^{\overline{q}_k}|\overline{o}_1|\cdots|\overline{o}_l]$ and a subset $I$, which is concordant only with this partition.}
    \label{I which is concordant with central circular pair}
\end{figure}
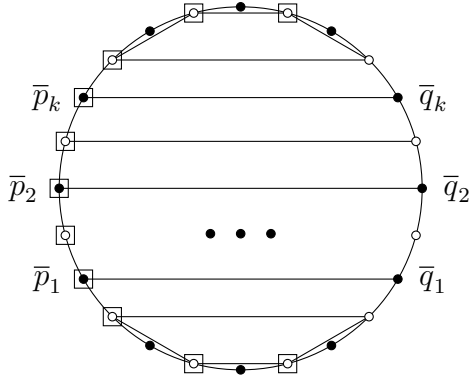

\end{proof}

Proposition \ref{circular minors via Plücker coordinates 0} admits the following generalization:
\begin{proposition} \label{circ-con-gro1}
\hskip 1pt
\begin{itemize}
    \item If the minor $M^Q_P$ is semicontiguous, then 
    $$(-1)^k\det M_P^Q=\frac{\Delta_I}{L_{unc}},$$
  for some $I\subset[2n]$, $|I|=n-1$. 
  \item For any circular minor $M^Q_P$
  with $(P;Q)=(\overline{p}_1,\ldots,\overline{p}_k;\overline{q}_1,\ldots,\overline{q}_k)$ we have:
  \begin{equation} \label{form:gen-min-grove}
    (-1)^k\det M_P^Q=\frac{L[_{\overline{p}_1}^{\overline{q}_1}|\cdots|_{\overline{p}_k}^{\overline{q}_k}|\overline{o}_1|\cdots|\overline{o}_l]}{L_{unc}}.  
  \end{equation}
\end{itemize}
\end{proposition} 

\begin{proof}
    The proof of the second assertion repeats the proof of the corresponding part in the proof of Proposition \ref{circular minors via Plücker coordinates 0}. To deduce the first assertion from the second one, we need to find a subset $I\in[2n]$ which is concordant only with the grove defined by the circular minor $M_P^Q$. The argument is similar to that of Lemma \ref{lemma: bijection} but is technically more involved. In this case, the subset $I$ corresponding to the minor $M_P^Q$ with contiguous $Q$ is equal to:
    $$I=\{2p_1-1,2p_2-1,\ldots,2p_k-1\}\cup\{2p_1,2p_1+2,2p_1+4,\ldots,2p_k-2\}\cup$$
    $$\cup\{2p_k,2p_k+2,\ldots,2q_k-4\}\cup\{2p_1-2,2p_1-4,\ldots,2q_1+2\},$$
    where the difference with \eqref{eq: I for lemma on Lam minors} is that $I$ does not contain the whole segment $[2p_1-1;2p_k-1]$, but only $P$ and all even numbers from this segment. Moreover, every subset $I\subset[2n]$, $|I|=n-1$ of this form is concordant with the only one grove.
\end{proof}

\subsection{Central circular minors} \label{sec: Central circular minors}

\begin{definition}\label{def:Dn}
For an odd $n$ and every $1\le k\le \frac{n-1}{2}$ and $1\leq j\leq n$ with all indices understood modulo $n$, a circular pair of the form 
$$\left(\overline{j},\overline{j+1},\ldots,\overline{j+k-1};\overline{\frac{n-1}{2}+k+j-1},\overline{\frac{n-1}{2}+k+j-2},\ldots,\overline{\frac{n-1}{2}+j}\right)$$
is called {\itshape central}. We denote the set of central circular pairs on $[\overline{n}]$ by $D_n$. The minors corresponding to the central pairs are called {\itshape central}. 
\end{definition}

There is an equivalent way to define central minors, which is more technical but allows one to define central circular pairs simultaneously for odd and even $n$. This requires some preparatory notation. A contiguous circular pair is denoted by $CM_{a,b,y}:=(\overline{a},\overline{a+1},\ldots,\overline{a+y-1};\overline{b+y-1},\ldots,\overline{b+1},\overline{b})$. Then a {\itshape central pair} is a contiguous pair of the form $CM_{x,y}:=CM_{a,b,y}$ with 
\begin{equation}\label{eq: central minors}
a=\left\lfloor\frac{x-y}{2}\right\rfloor\ \ \ \ \ b=\left\lfloor\frac{x-y+n-(n-1\ \mathrm{mod}\ 2)}{2}\right\rfloor,
\end{equation}
where $1\le x\le n$, $1\le y<n/2$, or $y=n/2$ and $x+y$ is odd. Denote the set of central pairs on $[\overline{n}]$ by $D_n$. We also denote by $CM_n$ the analogous set, the only difference being that $x$ ranges over $1\le x\le 2n$. For further details, see \cite[Section 4.1]{KW space}.

There are $\binom{n}2$ central circular pairs. These pairs are called central pairs since, if $n$ points are arranged on a circle, the $k$ chords connecting each point of $P$ to the corresponding opposite point of $Q$, are parallel chords which are “central” in the sense that they are as close to a fixed {\itshape central diagonal} of the circle as possible while remaining disjoint from each other. 
The same pictorial definition works for even $n$; however, its definition in terms of concrete sets $P$ and $Q$ is more difficult. For more context on central minors see \cite[Section 4.5.3]{K}. The central circular minors were called {\itshape small central minors} in \cite[Section 4.1]{KW space} and for an odd $n$ {\itshape diametric solid minors} in \cite[Section 6.1]{ALT}. The central circular pairs corresponding to $k=\frac{n-1}{2}$ are called {\itshape maximal}.

We will also use the statistics defined in \cite[Definition 6.2.7]{ALT}. Let $(P;Q)=(\overline{p}_1,\ldots,\overline{p}_k;\overline{q}_1,\ldots,\overline{q}_k)$ be a non-empty circular pair. For $\overline{a},\overline{b}\in[\overline{n}]$, let $d(\overline{a},\overline{b})$ be the number of boundary vertices on the clockwise arc from $\overline{a}$ to $\overline{b}$, including both endpoints. Write $d_1=d_1(P;Q)=d(p_k,q_k)$, and $d_2=d_2(P;Q)=d(q_1,p_1)$. The statistics 
$D(P;Q), T(P;Q)$, and $k(P;Q)$ are defined by:
$$D(P;Q):=d_1(P;Q)-d_2(P;Q)=d(p_k,q_k)-d(q_1,p_1),$$
$$T(P;Q):=\begin{cases}
    \frac{p_1+q_1}{2}\ (\mathrm{mod}\ n)\ \mathrm{ if } \ p_1<q_1 \\
    \frac{p_1+q_1+n}{2}\ (\mathrm{mod}\ n)\ \mathrm{if}\ p_1>q_1
\end{cases}$$
$$k(P;Q):=|P|,\ \mathrm{that\ is,\ the\ size\ of\ } (P;Q).$$

\begin{remark} \label{remark: central minors}
    As noted in \cite[Remark 6.1.5]{ALT}, $D_n$ consists of the contiguous circular pairs with $|d_1-d_2|=1$ when $n$ is odd, and $CM_n$ consists of some of the contiguous circular pairs with either $|d_1-d_2|=0$ or $|d_1-d_2|=2$ when $n$ is even. More precisely, for odd $n$, $D_n=CM_n$, and there is a unique central pair along any central diagonal. For even $n$ the situation is more complicated: if $x+y$ is odd, then everything is similar to the odd case. If $x+y$ is even, then $CM_{x,y}$ is off-center by one and $CM_{x+n,y}$ is off-center by one in the other direction; they both belong to $CM_n$ but only one of them belongs to $D_n$. In terms of $|d_1-d_2|$ this can be rephrased as follows: a central diagonal passing through nodes gives a unique central pair with $|d_1-d_2|=0$; whereas a central diagonal that does not pass through any node gives two contiguous pairs with $|d_1-d_2|=2$ belonging to $CM_n$ but only one of them is central. The formula \eqref{eq: central minors} implements the specific choice of the central one among the two.
\end{remark}

The following theorem is the first test for circular total positivity. We will later state a description of more tests using cluster algebras.
\begin{theorem}\textup{\cite[Theorem 10]{KW space}} \label{thm: central form test}
    Any circular minor of a symmetric $n\times n$ matrix is a subtraction-free rational function of the central circular minors.
\end{theorem}
Note that this test is minimal, see \cite[Section 4.5.3]{K}.

\section{Cluster algebras}
\label{sec:cl}
\subsection{Background on cluster algebras}
In this section we present the necessary background on cluster algebras following \cite{FZW Ch1-3, FZW Ch4-5, Sc}. In this paper, we restrict our attention to skew-symmetric cluster algebras of geometric type.

A {\itshape quiver} is a finite oriented graph, consisting of vertices and directed edges (called {\itshape arrows}). We allow multiple edges, but we do not allow loops (i.e. an arrow cannot connect a vertex to itself) and oriented $2$-cycles (i.e. no arrows of opposite orientation may connect the same pair of vertices). In what follows, we will need a slightly richer notion, with some vertices in a quiver designated as {\itshape frozen}. The remaining vertices are called {\itshape mutable}. We will always assume that there are no edges between pairs of frozen vertices.

Let $Q$ be a quiver as above with $m$ vertices, $n$ of which are mutable. Let $k$ be a mutable vertex in a quiver $Q$. The {\itshape quiver mutation} $\mu_k$ transforms $Q$ into a new quiver $Q'=\mu_k(Q)$ via a sequence of three steps:
\begin{enumerate}
    \item For each oriented two-arrow path $i\to k\to j$, add a new arrow $i\to j$ (unless both $i$ and $j$ are frozen, in which case do nothing).
    \item Reverse the direction of all arrows incident to the vertex $k$.
    \item  Repeatedly remove oriented $2$-cycles until unable to do so.
\end{enumerate}

 As an {\itshape ambient field} for a cluster algebra, we take a field $\mathcal{F}$ of rational functions over $\mathbb C$ in $m$ independent variables.

A {\itshape seed} of geometric type in $\mathcal{F}$ is a pair $(\tilde{\textbf{x}},Q)$ where $\tilde{\mathbf{x}}$ is an $m$-tuple of elements of $\mathcal{F}$ forming a {\itshape free generating set}; that is, $x_1,\ldots,x_m$ are algebraically independent, and $\mathcal{F}=\mathbb C(x_1,\ldots,x_m)$, and $Q$ is a quiver with $m$ vertices, $n$ of which are mutable. We use the following terminology: $\tilde{\mathbf{x}}$ is called the {\itshape extended cluster} of the seed $(\tilde{\textbf{x}},Q)$; the $n$-tuple $\mathbf{x}=(x_1,\ldots,x_n)$ is called the {\itshape cluster} of this seed; the elements $x_1,\ldots,x_n$ are called its {\itshape cluster variables}; and the remaining elements $x_{n+1},\ldots,x_m$ of $\tilde{\mathbf{x}}$ are called {\itshape frozen variables} (or {\itshape coefficient variables}).

Let $(\tilde{\textbf{x}},Q)$ be a seed. Take an index $k\in\{1,\ldots,n\}$. The {\itshape seed mutation} $\mu_k$ in direction $k$ transforms $(\tilde{\textbf{x}},Q)$ into the new seed $\mu_k(\tilde{\textbf{x}},Q)=(\tilde{\textbf{x}}', Q')$ defined as follows:

$\bullet$ \ $Q'=\mu_k(Q)$.

$\bullet$ \ the extended cluster $\tilde{\mathbf{x}}'=(x_1',\ldots,x_m')$ is given by $x_j'=x_j$ for $j\ne k$, whereas $x_k'\in\mathcal{F}$ is determined by the {\itshape exchange relation}
\begin{equation} \label{eq: exchange relation}
x_kx_k'=\prod\limits_{i\to k}x_i^{b_{ik}}+\prod\limits_{k\to i}x_i^{b_{ki}},
\end{equation}
where $b_{ik}$ is the number of arrows from the vertex $i$ to the vertex $k$ in $Q$. In other words, the first (resp., second) monomial on the right-hand side of \eqref{eq: exchange relation} corresponds to the arrows pointing toward (resp., away from) the vertex $k$.


Choose an initial seed $(\tilde{\mathbf{x}},Q)$ and let $\mathcal{S}$ be the unique family of seeds containing the {\itshape initial seed} $(\tilde{\mathbf{x}},Q)$ and closed under mutations. The {\itshape cluster algebra} $\mathcal{A}=\mathcal{A}(\mathcal{S})$ is the $\mathbb C[x_{n+1},\ldots,x_m]$-subalgebra of $\mathcal{F}$ generated by all cluster variables from $\mathcal{S}$. The number $n$ is called the {\itshape rank} of the cluster algebra $\mathcal{A}$. 

We will need the two following operations on cluster algebras. 

Let $(\tilde{\mathbf{x}},Q)$ be a seed of rank $n$, and let $x_i\in \mathbf{x}$ be a cluster variable. {\itshape Freezing} the variable $x_i$, or
equivalently freezing the vertex $i$, is a transformation of the seed that reclassifies $i$ and $x_i$ as frozen, and accordingly removes all arrows connecting frozen vertices to each other. (In addition, this would typically require a change of indexing, provided we want to keep using the smaller indices $1,\ldots,n-1$ for the mutable variables). More generally, we can freeze any subset of cluster variables. The order of freezing does not matter. Note that freezing commutes with mutations at the remaining mutable vertices, \cite[Lemma 4.2.2]{FZW Ch4-5}. 

Let $(\tilde{\mathbf{x}},Q)$ be a seed with an $m$-element extended cluster $\tilde{\mathbf{x}}$ and let $x_i\in\tilde{\mathbf{x}}$ be a frozen variable. {\itshape Trivialization} at the index $i$ (or, of the variable $x_i$) is a transformation of the seed that removes $x_i$ from $\tilde{\mathbf{x}}$, and removes the index $i$ from the index set, and removes the frozen vertex $i$ from  $Q$, together with all arrows incident to it. (As in the case of freezing, a renumbering may be required if we want to use the indices $1,\ldots,m-1$ after trivialization.) More generally, we can trivialize any subset of frozen variables; the order of operations does not matter. Note that trivialization commutes with seed mutation, \cite[Lemma 4.3.2]{FZW Ch4-5}. Note also that trivialization of a frozen variable $x_i$ can be interpreted as setting $x_i=1$.

Cluster algebras have many remarkable properties, one of them is the Laurent Phenomenon.
\begin{theorem} \cite[Theorem 3.3.1 and Theorem 3.3.6]{FZW Ch1-3} \label{thm: laurent phenomenon}
    Let $\mathcal{A}$ be a cluster algebra and let $\tilde{\mathbf{x}}$ be an extended cluster whose cluster is $\mathbf{x}$. Any cluster variable $x$ in $\mathcal{A}$ can be expressed as a Laurent polynomial in the variables $\tilde{\mathbf{x}}$ whose denominator is a monomial in the mutable variables of  $\mathbf{x}$.
\end{theorem}

\subsection{Postnikov's theory} \label{sub:postt}
Scott's classical  construction of the cluster algebra on the Grassmannian begins with Postnikov arrangements, see Section \ref{sec:Scottdef}. To construct such seeds for electrical network, we need to use the so-called generalized Temperley's trick.

The key result of \cite{Pos} is that each point $X \in \mathrm{Gr}_{\geq 0}(k, n)$ can be parametrized by plabic graphs:
\begin{definition}  \label{def:plabicgraph}
A planar bipartite graph $G$ together with a weight function $w:E(G)\rightarrow \mathbb R_{> 0}$ is called a {\itshape plabic graph} $N(G, w)$ if the following holds:
\begin{itemize}
\item $G$ is embedded into a disk $D$, and the set of its nodes is divided into two subsets: the set of  interior nodes $V_I$ and the set of boundary nodes $V_B;$ 
    \item The boundary nodes $V_B$ lie on a boundary circle $\partial D$ and are numbered clockwise from $1$ to $n:=|V_B|$;
    \item The degrees of the boundary vertices are all equal to one;
    \item Each vertex of $G$ is colored  either black or white color and each edge is incident to the vertices of different colors.
\end{itemize}
Note that on the figures we do not specify the colors of boundary nodes, and we do not indicate edges of weight $1.$   
\end{definition}

\begin{theorem}\textup{\cite[Theorem 4.8]{Pos}, \cite[Theorem 3.2]{L}} \label{th:plabicgr}
 There is  a bijection between the set of weighted  plabic graphs   and the totally non-negative Grassmannian $\mathrm{Gr}_{\geqslant 0} (k,n)$,  where the plabic graphs are considered up to Postnikov moves and gauge transformations, see \cite[Section 12]{Pos}.  
\end{theorem}

\begin{definition}
    An {\itshape oriented medial graph}  of a plabic graph $N(G, w)$  is a  graph $G_M$ whose nodes are defined as follows:
    \begin{itemize}
        \item Boundary nodes of $G_M$ coincide with boundary nodes of $N(G, w)$;
        \item Interior nodes  of  $G_M$   are the midpoints of the  edges connecting interior nodes of $G$.
    \end{itemize}
    Two nodes of  $G_M$ are connected by an edge if the edges of the original graph $G$ are adjacent. The edges of $G_M$ are oriented as follows:
     \begin{itemize}
        \item Clockwise around each interior white node of $G$;
        \item Counterclockwise around  each interior black node of $G$.
    \end{itemize}
    Note that the orientation rule is well-defined. Moreover, since all interior nodes of $G_M$ have degree four, we can correctly define the {\itshape strands} of $G_M$ as  the oriented paths which pass straight through every degree-four node.  
\end{definition}

\begin{definition} \label{def: zig-zag path}
    A collection of strands naturally gives rise to a {\itshape strand permutation} $\tau(N)$, which is defined  by the following rule: if a strand starts at a boundary node labeled by $i$ and  ends  at a boundary node labeled by $j$, we have that $\tau(N)(i)=j.$

    It is easy to see that each strand induces an oriented {\itshape zig-zag path} on the edges. The boundary node at which the strand starts (respectively ends) coincides with the strand's starting (ending) boundary node; if we reach a black node $v$ we follow the edge $vw$ obtained by turning maximally right at $v$; if we reach a white  node $v$, we turn maximally left at $v$. 
\end{definition}

\begin{figure}[H]
     \centering
     \includegraphics[scale=0.14]{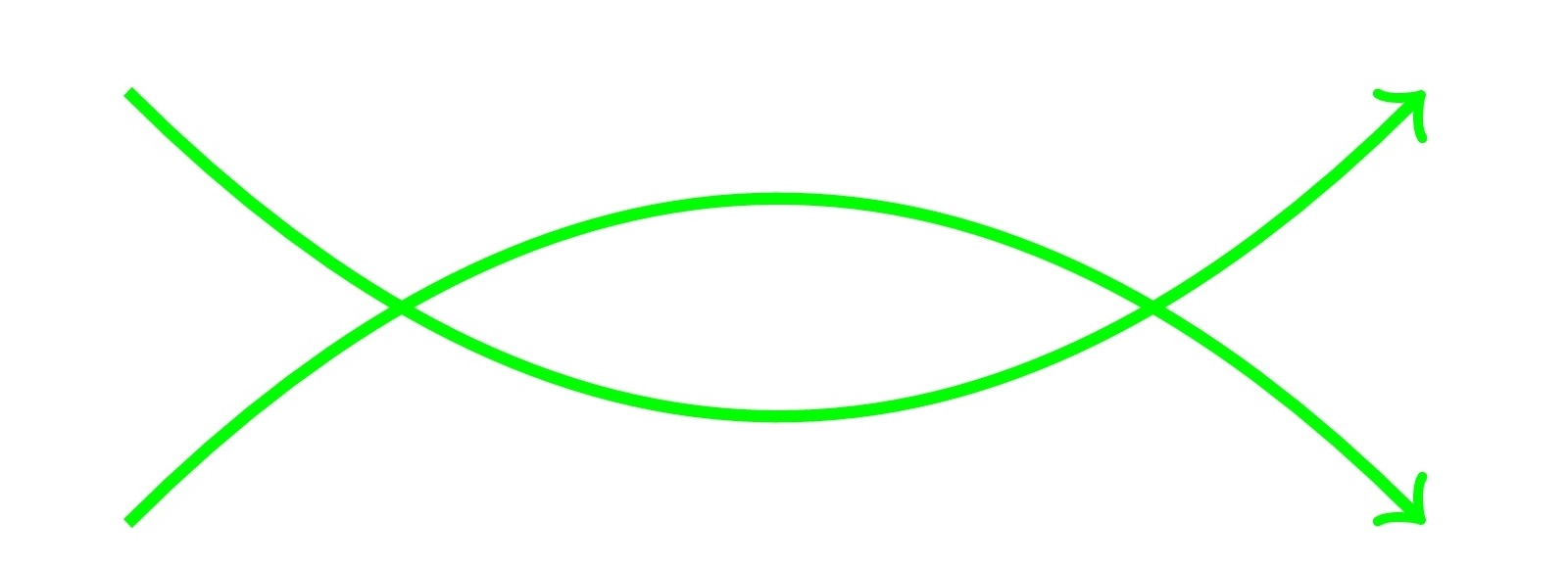}
     \caption{A forbidden oriented lens} \label{fig:forbid}
     \end{figure}

\begin{definition}  \label{def:minplabic}
   A plabic graph $N(\Gamma, \omega)$ is called {\itshape reduced} (or sometimes {\itshape minimal}) if the following holds:
    \begin{itemize}
        \item Each strand does not intersect itself;
        \item Strands do not form  closed cycles  apart from cycles attached to isolated boundary nodes;
        \item Strands do not form oriented lenses, see Fig. \ref{fig:forbid}.
    \end{itemize}   
\end{definition}

\subsection{The generalized Temperley's trick} \label{subsec:temptrick}
The generalized Temperley's trick of an electrical network is a planar bicolored graph (or plabic graph or Lam model) which is obtained by combining a graph $\Gamma$ and its dual graph $\Gamma^{*};$ all nodes of the graphs $\Gamma$ and $\Gamma^{*}$ are black, and additional nodes, which lie on intersections  of edges $\Gamma$ and $\Gamma^{*}$, are white; each edge coming from an edge of $\Gamma$ has a weight equal to the conductance of this edge, and all additional edges have weight $1$.

\begin{definition}  \label{temp_gen}
For each electrical network $e(\Gamma, \omega)\in E_n$ we  define a plabic graph  $N(e)$, which we call {\itshape generalized Temperley's trick} of $e$,
 by the following construction. The boundary nodes of $N(e)$ are defined as follows:
\begin{itemize}
    \item If $\Gamma$ has boundary nodes $\{\bar 1, \bar 2, . . . , \bar n\}$, then $N(e)$ will have white boundary nodes $\{1, 2, . . . , 2n\}$, where boundary node $\bar i$ is identified with $2i-1$ and the node $2i$ is identified with the node $\tilde i$ used to label dual non-crossing partitions;
    \item Each boundary node has a degree equal to  $1$.
\end{itemize}
The interior nodes of $N(e)$ are defined as follows:
\begin{itemize}
    \item We have a black interior node $b_v$ for each interior node $v$ of $\Gamma$;
    \item We have a black interior node $b_F$ for each interior face $F$ of $\Gamma$;
    \item We have a white interior node $w_e$ placed at the midpoint of each interior edge $e$ of $\Gamma;$
    \item  For each boundary node $\bar i$, we have a black interior node  $b_i$.
\end{itemize}  
The edges of $N(e)$ are defined as follows: 
\begin{itemize}
\item  If $v$ is a node of an edge $e$ in $\Gamma$, then $b_v$ and $w_e$ are joined, and the weight of this edge is equal
to the weight $\omega(e)$ of $e$ in $\Gamma$, 
\item  If $e$ borders $F$, then $w_e$ is joined to $b_F$ by an edge with weight $1$, 
\item  The node $b_i$ is joined by an edge with weight $1$ to the boundary node $2i - 1$ in $N(e)$, and $b_i$
is also joined by an edge with weight $1$ to $w_e$ for any edge $e$ incident
to $\bar i$ in $\Gamma$, 
\item  Even boundary nodes $2i$ in $N(e)$ are joined by an edge with weight $1$ to the face vertex $b_F$ of the face $F$ that they lie in.
\end{itemize}
\end{definition}

Each electrical network can be transformed into a minimal network using electrical transformations, see \cite[Corollary 9.4]{CM}. Moreover, minimal electrical networks are related to reduced plabic graphs through generalized Temperley's trick.

\begin{figure}[h!]
    \centering
    \includegraphics[width=0.35\textwidth]{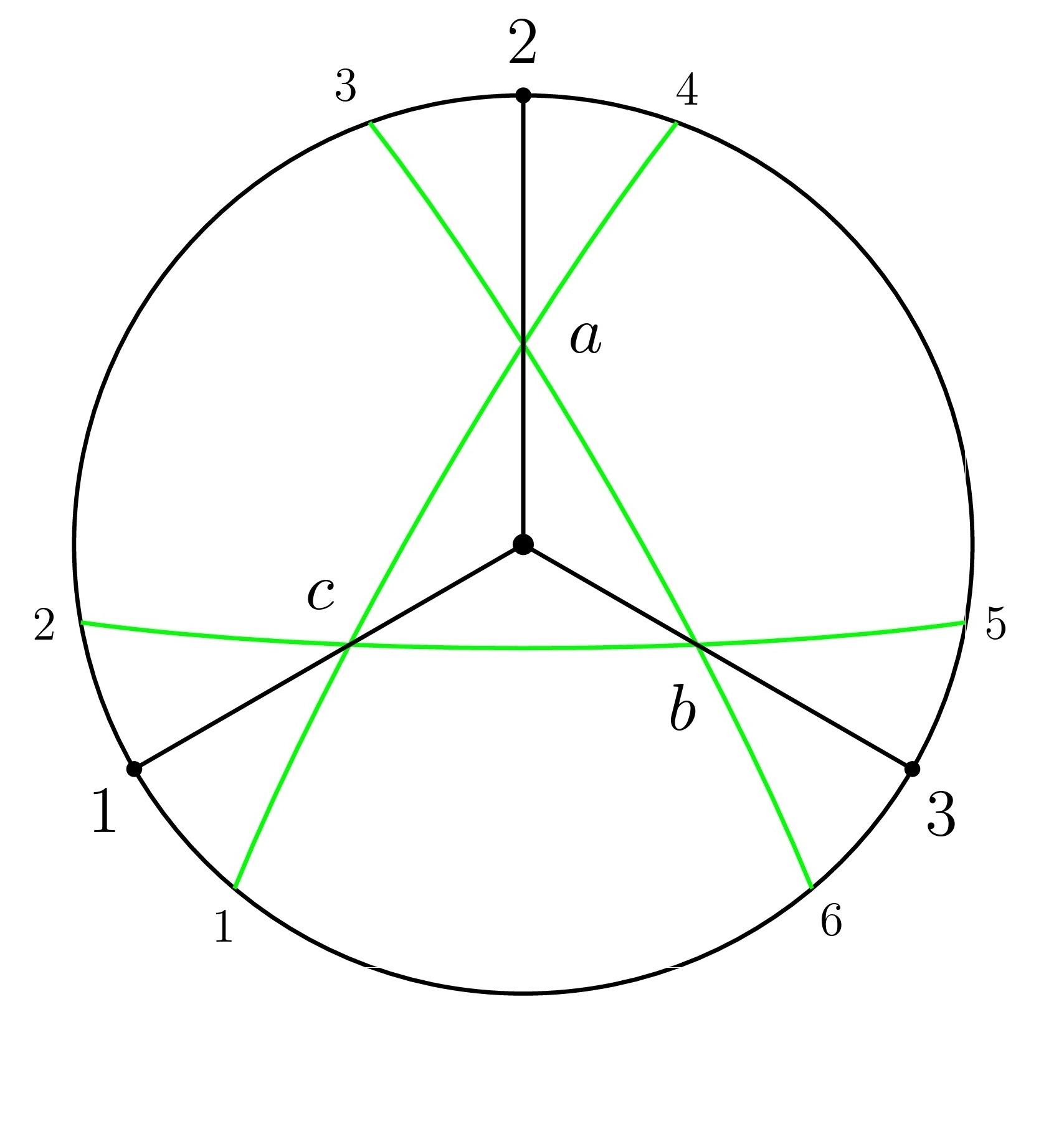}
    \caption{Star-shaped network, its medial graph and the strand permutation $T(e)=(14)(25)(36)$ }
    \label{fig:triangle-strands}
\end{figure}

\begin{definition}
    The {\itshape medial graph} of an electrical network $e(\Gamma, \omega) \in E_n $ is the graph $\Gamma_M$ whose internal vertices are the midpoints of the edges of $\Gamma$ and two internal vertices are connected by an edge if the edges of the original graph $\Gamma$ are adjacent. The boundary vertices of $\Gamma_M$ are defined as the intersections of the natural extensions of the edges of $\Gamma_M$  with the boundary circle. Since each interior vertex of the medial graph has degree four, the strands naturally define a permutation $T(e)$ on the set of points $\{1,\dots,2n\}$, see Fig. \ref{fig:triangle-strands}. 
\end{definition}

 \begin{definition}
An electrical network   $e(\Gamma, \omega) \in E_n $  is called {\itshape minimal} if each strand of its medial graph does not form a closed cycle; the strands  do not have self-intersections and any two strands intersect in at most one point, equivalently the medial graph has no loops or lenses, see  Fig. \ref{fig:forbid1}.

If an electrical network   $e $ is minimal and well-connected, it is called {\itshape critical}. Note that our notations are different from \cite{L} where minimal networks are called critical.
\end{definition}

\begin{figure}[H]
     \centering
     \includegraphics[scale=0.14]{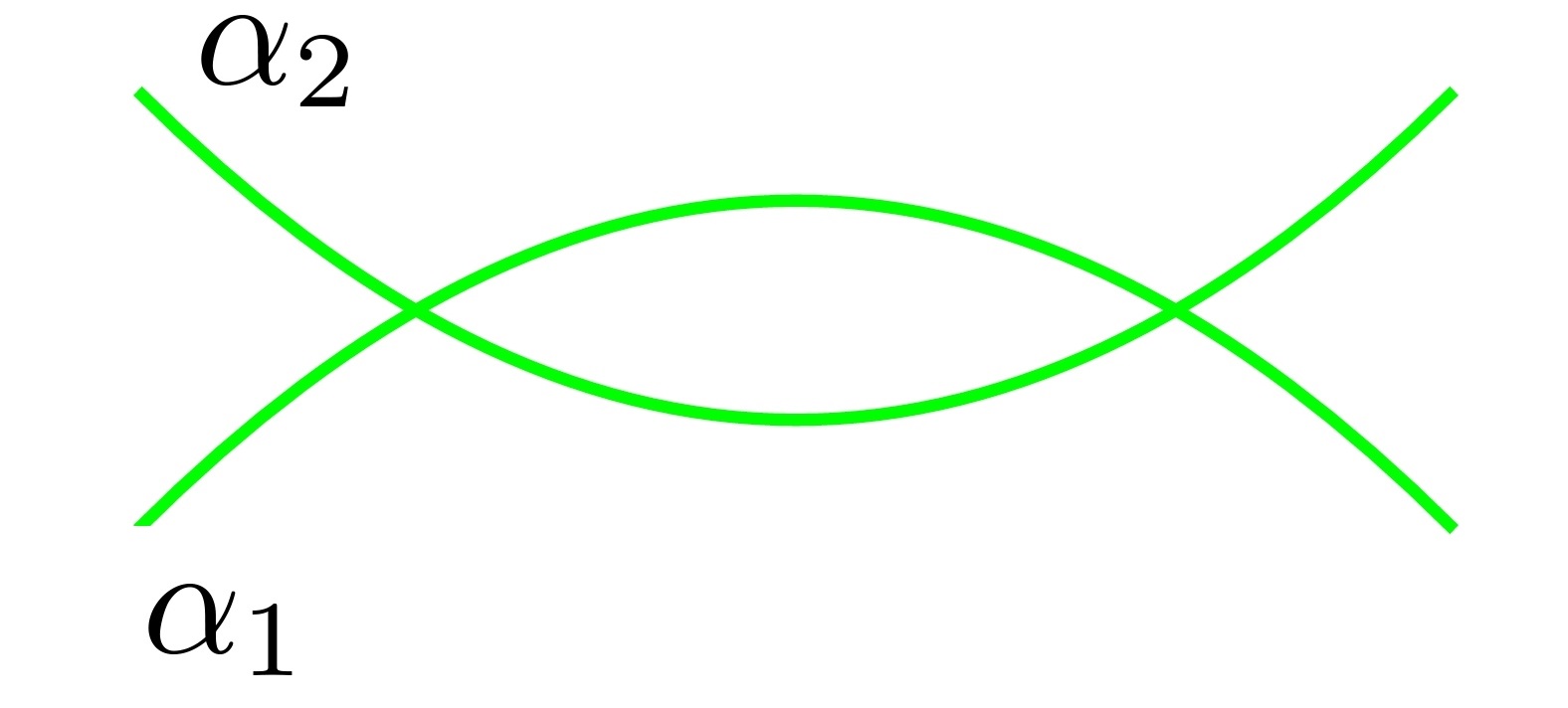}
     \caption{A forbidden  lens} \label{fig:forbid1}
     \end{figure}

\begin{lemma} \textup{\cite[Theorem 2.17]{Kaz}} \label{lemma: minimal}
    Consider a connected minimal electrical network   $e(\Gamma, \omega) \in E_n,$ then the associated  plabic graph 
    $N(e)$ is reduced.
\end{lemma}

\subsection{The special series of electrical networks} \label{sub:spclass}
We now define a series of critical networks, that were constructed in \cite[Section 2.5]{CM}. Later we will use the circular minors of their response matrices.

For each integer $n\ge 3$, the nodes of the graph $G_n$ are the points of intersection of $n$ radial lines and a finite collection of circles centered at the origin. The edges of $G_n$ lie on the $n$ rays $\rho_0,\rho_1,\ldots,\rho_{n-1}$ originating from the origin at angles $\theta_0, \theta_1,\ldots,\theta_{n-1}$ measured clockwise from the first ray $\rho_0$, where $0=\theta_0<\theta_1<\ldots<\theta_{n-1}<2\pi$. The circles have radii $0<r_1<\ldots<r_i<\ldots$. 

For convenience, $(i,j)$ will represent the point that is the intersection of the circle of radius $r_i$ with ray $\rho_j$. All points $(0,j)$ are identified with the single point $(0,0)$. All second indices are taken modulo $n$; thus the point $(i,j+n)$ is identified with the point $(i,j)$; in particular, $(i,n)$ is the same point as $(i,0)$.

The well-connected critical graph $G_n$ requires four separate descriptions depending on $n $ mod $4$. We give the construction for the two odd congruence classes.

Let $n=4m+1$. The boundary circle of radius $m+1$, centered at $(0,0)$. The nodes of $G_{4m+1}$ are the points $(i,j)$ for integers $i$ and $j$ with $0<i\le m+1$ and $1\le j\le 4m+1$. The radial edges are the radial line segments joining $(i,j)$ and $(i+1,j)$ for each $0<i\le m$ and each $1\le j\le 4m+1$. The circular edges are the circular arcs joining $(i,j)$ and $(i,j+1)$ for each $1\le i\le m$ and each $1\le j\le 4m+1$. The graph $G_{4m+1}$ has $2m(4m+1)$ edges and $(m+1)(4m+1)$ vertices. The boundary nodes of $G_{4m+1}$ are the points $v_j=(m+1,j)$, for $j=1,\ldots,4m+1$, with the convention that $v_0=v_{4m+1}$. The graph $G_5$ is shown in Fig. \ref{fig: G_5}.

Let $n=4m+3$. The boundary circle for $G_{4m+3}$ is the circle of radius $m+1$, centered at $(0,0)$. The nodes of $G_{4m+3}$ are the points $(h,j)$ for integers $h$ and $j$ with $0\le h\le m+1$ and $0\le j\le 4m+3$. The radial edges are the radial line segments joining $(h,j)$ to $(h+1,j)$ for each $0\le h\le m$ and each $0\le j<4m+3$. The circular edges are the circular arcs joining $(h,j)$ to $(h,j+1)$ for each $1\le h\le m$ and each $0\le j\le 4m+3$.

\begin{figure}[ht]
\centering
\begin{tikzpicture}
    \draw (0,0) circle (0.75);

    \draw (90:0.75) -- node [left] {}(90:3);
    \draw (162:0.75) -- node [left] {}(162:3);
    \draw (234:0.75) -- node [left] {}(234:3);
    \draw (306:0.75) -- node [left] {}(306:3);
    \draw (376:0.75) -- node [left] {}(376:3);
    \draw[white] (-90:0) -- node [left] {}(-90:3);

    \filldraw[fill=black] (90:0.75) node [right] {} circle (2pt);
    \filldraw[fill=black] (162:0.75) node [right] {} circle (2pt);
    \filldraw[fill=black] (234:0.75) node [right] {} circle (2pt);
    \filldraw[fill=black] (306:0.75) node [right] {} circle (2pt);
    \filldraw[fill=black] (376:0.75) node [right] {} circle (2pt);

    
\end{tikzpicture} \phantom{aaaa}
\begin{tikzpicture}
    \draw (0,0) circle (0.75);
    \draw (0,0) circle (1.5);

    \draw (90:0.75) -- node [left] {}(90:3);
    \draw (162:0.75) -- node [left] {}(162:3);
    \draw (234:0.75) -- node [left] {}(234:3);
    \draw (306:0.75) -- node [left] {}(306:3);
    \draw (376:0.75) -- node [left] {}(376:3);

    \draw (126:0) -- node [left] {}(126:3);
    \draw (198:0) -- node [left] {}(198:3);
    \draw (270:0) -- node [left] {}(270:3);
    \draw (342:0) -- node [left] {}(342:3);
    \draw (414:0) -- node [left] {}(414:3);

    \filldraw[fill=black] (90:0.75) node [right] {} circle (2pt);
    \filldraw[fill=black] (162:0.75) node [right] {} circle (2pt);
    \filldraw[fill=black] (234:0.75) node [right] {} circle (2pt);
    \filldraw[fill=black] (306:0.75) node [right] {} circle (2pt);
    \filldraw[fill=black] (376:0.75) node [right] {} circle (2pt);

    \filldraw[fill=black] (0:0) node [right] {} circle (2pt);
    \%filldraw[fill=black] (126:2) node [right] {} circle (2pt);
    \filldraw[fill=white] (126:0.75) node [right] {} circle (2pt);
    \filldraw[fill=white] (198:0.75) node [right] {} circle (2pt);
    \filldraw[fill=white] (270:0.75) node [right] {} circle (2pt);
    \filldraw[fill=white] (342:0.75) node [right] {} circle (2pt);
    \filldraw[fill=white] (414:0.75) node [right] {} circle (2pt);

    \filldraw[fill=white] (90:1.5) node [right] {} circle (2pt);
    \filldraw[fill=white] (162:1.5) node [right] {} circle (2pt);
    \filldraw[fill=white] (234:1.5) node [right] {} circle (2pt);
    \filldraw[fill=white] (306:1.5) node [right] {} circle (2pt);
    \filldraw[fill=white] (376:1.5) node [right] {} circle (2pt);

    \filldraw[fill=black] (90:2.25) node [right] {} circle (2pt);
    \filldraw[fill=black] (162:2.25) node [right] {} circle (2pt);
    \filldraw[fill=black] (234:2.25) node [right] {} circle (2pt);
    \filldraw[fill=black] (306:2.25) node [right] {} circle (2pt);
    \filldraw[fill=black] (376:2.25) node [right] {} circle (2pt);

    \filldraw[fill=black] (126:1.5) node [right] {} circle (2pt);
    \filldraw[fill=black] (198:1.5) node [right] {} circle (2pt);
    \filldraw[fill=black] (270:1.5) node [right] {} circle (2pt);
    \filldraw[fill=black] (342:1.5) node [right] {} circle (2pt);
    \filldraw[fill=black] (414:1.5) node [right] {} circle (2pt);

    
\end{tikzpicture} 
    \caption{Graph $G_5$ and its generalized Temperley's trick}
    \label{fig: G_5}
\end{figure}
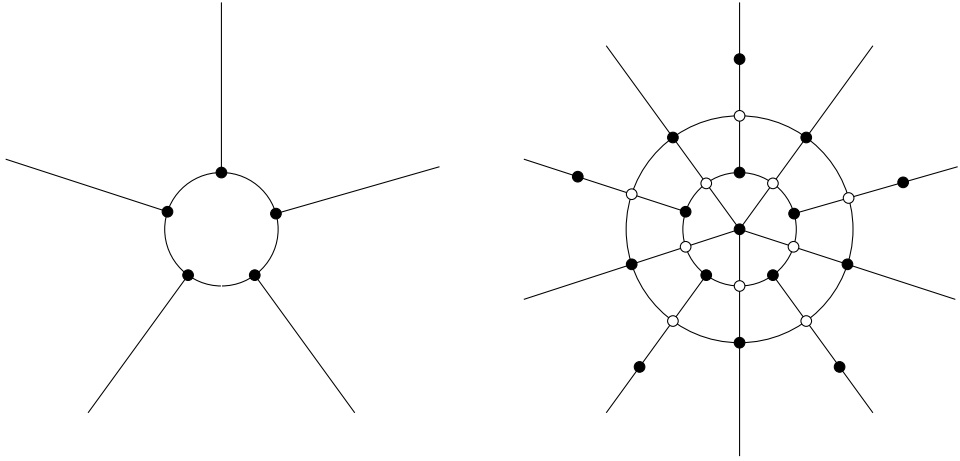

\begin{figure}[ht]
\centering

\begin{tikzpicture}
    \draw (0,0) circle (1.5);

    \draw (51.4:0) -- node [left] {}(51.4:3);
    \draw (102.8:0) -- node [left] {}(102.8:3);
    \draw (154.2:0) -- node [left] {}(154.2:3);
    \draw (205.6:0) -- node [left] {}(205.6:3);
    \draw (257:0) -- node [left] {}(257:3);
    \draw (308.4:0) -- node [left] {}(308.4:3);
    \draw (359.8:0) -- node [left] {}(359.8:3);

    
   \filldraw[fill=black] (0:0) node [right] {} circle (2pt);
   \filldraw[fill=black] (359.8:1.5) node [right] {} circle (2pt);
   \filldraw[fill=black] (308.4:1.5) node [right] {} circle (2pt);
     \filldraw[fill=black] (257:1.5) node [right] {} circle (2pt);
 \filldraw[fill=black] (205.6:1.5) node [right] {} circle (2pt);
 \filldraw[fill=black] (154.2:1.5) node [right] {} circle (2pt);
\filldraw[fill=black] (102.8:1.5) node [right] {} circle (2pt);
    \filldraw[fill=black] (51.4:1.5) node [right] {} circle (2pt);

    
\end{tikzpicture} \phantom{aaaaaaa}
\begin{tikzpicture}
    \draw (0,0) circle (1.5);
    \draw (0,0) circle (0.75);
    \draw (0,0) circle (2.25);

    \draw (51.4:0) -- node [left] {}(51.4:3);
    \draw (102.8:0) -- node [left] {}(102.8:3);
    \draw (154.2:0) -- node [left] {}(154.2:3);
    \draw (205.6:0) -- node [left] {}(205.6:3);
    \draw (257:0) -- node [left] {}(257:3);
    \draw (308.4:0) -- node [left] {}(308.4:3);
    \draw (359.8:0) -- node [left] {}(359.8:3);

   \filldraw[fill=black] (0:0) node [right] {} circle (2pt);
   \filldraw[fill=black] (359.8:1.5) node [right] {} circle (2pt);
   \filldraw[fill=black] (308.4:1.5) node [right] {} circle (2pt);
     \filldraw[fill=black] (257:1.5) node [right] {} circle (2pt);
 \filldraw[fill=black] (205.6:1.5) node [right] {} circle (2pt);
 \filldraw[fill=black] (154.2:1.5) node [right] {} circle (2pt);
\filldraw[fill=black] (102.8:1.5) node [right] {} circle (2pt);
    \filldraw[fill=black] (51.4:1.5) node [right] {} circle (2pt);

\filldraw[fill=black] (359.8:2.625) node [right] {} circle (2pt);
   \filldraw[fill=black] (308.4:2.625) node [right] {} circle (2pt);
     \filldraw[fill=black] (257:2.625) node [right] {} circle (2pt);
 \filldraw[fill=black] (205.6:2.625) node [right] {} circle (2pt);
 \filldraw[fill=black] (154.2:2.625) node [right] {} circle (2pt);
\filldraw[fill=black] (102.8:2.625) node [right] {} circle (2pt);
    \filldraw[fill=black] (51.4:2.625) node [right] {} circle (2pt);

    \filldraw[fill=white] (51.4:0.75) node [right] {} circle (2pt);
    \filldraw[fill=white] (102.8:0.75) node [right] {} circle (2pt);
    \filldraw[fill=white] (154.2:0.75) node [right] {} circle (2pt);
    \filldraw[fill=white] (205.6:0.75) node [right] {} circle (2pt);
    \filldraw[fill=white] (257:0.75) node [right] {} circle (2pt);
     \filldraw[fill=white] (308.4:0.75) node [right] {} circle (2pt);
     \filldraw[fill=white] (359.8:0.75) node [right] {} circle (2pt);

     \filldraw[fill=white] (51.4:2.25) node [right] {} circle (2pt);
    \filldraw[fill=white] (102.8:2.25) node [right] {} circle (2pt);
    \filldraw[fill=white] (154.2:2.25) node [right] {} circle (2pt);
    \filldraw[fill=white] (205.6:2.25) node [right] {} circle (2pt);
    \filldraw[fill=white] (257:2.25) node [right] {} circle (2pt);
     \filldraw[fill=white] (308.4:2.25) node [right] {} circle (2pt);
     \filldraw[fill=white] (359.8:2.25) node [right] {} circle (2pt);

\filldraw[fill=black] (77.1:0.75) node [right] {} circle (2pt);
    \filldraw[fill=black] (128.5:0.75) node [right] {} circle (2pt);
    \filldraw[fill=black] (179.9:0.75) node [right] {} circle (2pt);
    \filldraw[fill=black] (231.3:0.75) node [right] {} circle (2pt);
    \filldraw[fill=black] (282.7:0.75) node [right] {} circle (2pt);
     \filldraw[fill=black] (334.1:0.75) node [right] {} circle (2pt);
     \filldraw[fill=black] (385.5:0.75) node [right] {} circle (2pt);

\filldraw[fill=black] (77.1:2.25) node [right] {} circle (2pt);
    \filldraw[fill=black] (128.5:2.25) node [right] {} circle (2pt);
    \filldraw[fill=black] (179.9:2.25) node [right] {} circle (2pt);
    \filldraw[fill=black] (231.3:2.25) node [right] {} circle (2pt);
    \filldraw[fill=black] (282.7:2.25) node [right] {} circle (2pt);
     \filldraw[fill=black] (334.1:2.25) node [right] {} circle (2pt);
     \filldraw[fill=black] (385.5:2.25) node [right] {} circle (2pt);

\filldraw[fill=white] (77.1:1.5) node [right] {} circle (2pt);
    \filldraw[fill=white] (128.5:1.5) node [right] {} circle (2pt);
    \filldraw[fill=white] (179.9:1.5) node [right] {} circle (2pt);
    \filldraw[fill=white] (231.3:1.5) node [right] {} circle (2pt);
    \filldraw[fill=white] (282.7:1.5) node [right] {} circle (2pt);
     \filldraw[fill=white] (334.1:1.5) node [right] {} circle (2pt);
     \filldraw[fill=white] (385.5:1.5) node [right] {} circle (2pt);
\filldraw[fill=white] (77.1:1.5) node [right] {} circle (2pt);
    \filldraw[fill=white] (128.5:1.5) node [right] {} circle (2pt);
    \filldraw[fill=white] (179.9:1.5) node [right] {} circle (2pt);
    \filldraw[fill=white] (231.3:1.5) node [right] {} circle (2pt);
    \filldraw[fill=white] (282.7:1.5) node [right] {} circle (2pt);
     \filldraw[fill=white] (334.1:1.5) node [right] {} circle (2pt);
     \filldraw[fill=white] (385.5:1.5) node [right] {} circle (2pt);

\draw (77.1:0.75) -- node [left] {}(77.1:3);
    \draw (128.5:0.75) -- node [left] {}(128.5:3);
    \draw (179.9:0.75) -- node [left] {}(179.9:3);
    \draw (231.3:0.75) -- node [left] {}(231.3:3);
    \draw (282.7:0.75) -- node [left] {}(282.7:3);
    \draw (334.1:0.75) -- node [left] {}(334.1:3);
    \draw (385.5:0.75) -- node [left] {}(385.5:3);
\filldraw[fill=white] (77.1:1.5) node [right] {} circle (2pt);
    \filldraw[fill=white] (128.5:1.5) node [right] {} circle (2pt);
    \filldraw[fill=white] (179.9:1.5) node [right] {} circle (2pt);
    \filldraw[fill=white] (231.3:1.5) node [right] {} circle (2pt);
    \filldraw[fill=white] (282.7:1.5) node [right] {} circle (2pt);
     \filldraw[fill=white] (334.1:1.5) node [right] {} circle (2pt);
     \filldraw[fill=white] (385.5:1.5) node [right] {} circle (2pt);

\end{tikzpicture} 
    \caption{Graph $G_7$ and its generalized Temperley's trick}
    \label{fig: G_7}
\end{figure}
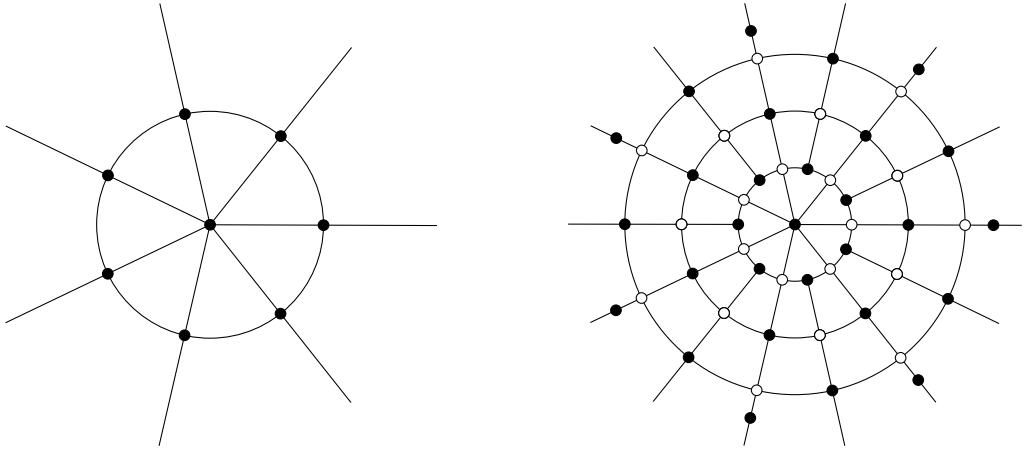

\begin{figure}[h!]
\centering
	\scalebox{0.14} 
	{
\includegraphics{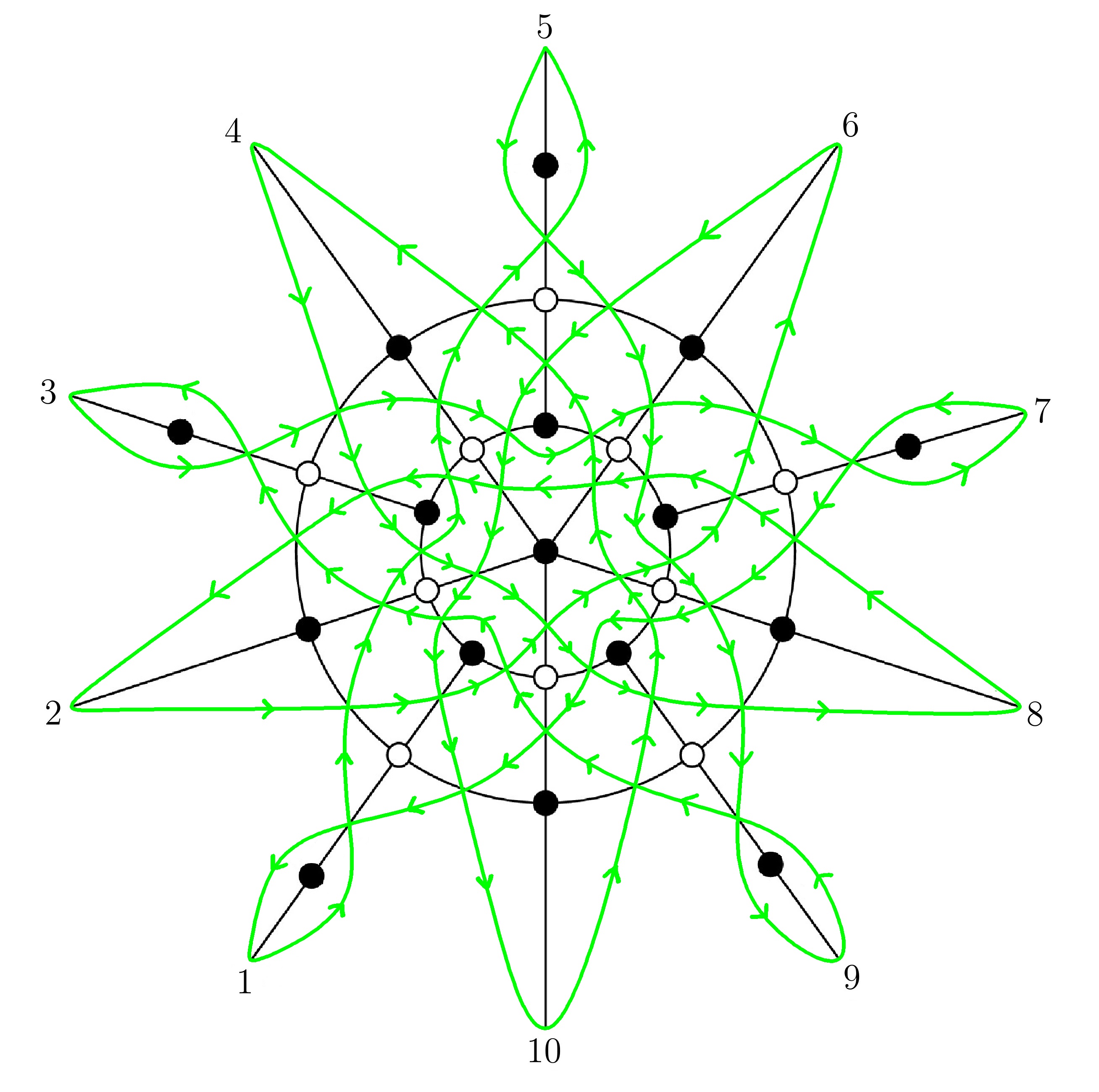}
}
\caption{Temperley's trick of $G_5$ and its oriented medial graph}
\label{fig: Temperley's trick of $G_5$ and its oriented medial graph}
\end{figure}

\begin{lemma} \label{lemma: permutation}
The strand permutation $\tau$ of the medial graph of Temperley's trick of an electrical network $G_n$ is equal to $\tau:i\mapsto i+(n-1)$.   
\end{lemma}

\begin{proof}
   Since $G_n$ is critical, by \cite[Corollary~4.9]{K} we obtain that the corresponding strand permutation  $T(G_n)$ sends $i$ to $i + n$. On the other hand, \cite[Proposition~5.17]{L} gives that $\tau(i) \equiv T(i) - 1 \pmod{2n}$, which completes the proof.
\end{proof}

\subsection{Grassmannian cluster algebra}\label{sec:Scottdef}
In this section, we briefly state a construction from \cite{Sc} of cluster algebra on the coordinate ring of the Grassmannian.

Consider a reduced plabic graph $G$ with $n$ boundary vertices. Include the index $i$ in the label of a face of $G$ if the zig-zag path (see Definition \ref{def: zig-zag path}) starting from the $i$-th node stays to the right of the face; i.e. if the circuit obtained by first traversing the $i$-th zig-zag path and then traversing the boundary of the circle clockwise from $\pi(i)$ to $i$ does not wind around the face. We will refer to this face-labeling rule as {\itshape Scott's rule}.

\begin{remark}
    The medial graph of a reduced plabic graph is called an \emph{alternating wiring arrangement}; see \cite[Section~14]{Pos} and Figure~\ref{fig: Temperley's trick of $G_5$ and its oriented medial graph}.
    In \cite{Sc}, the construction of the cluster algebra on $\mathbb C[Gr(k,n)]$ was described using this object, where it was referred to as {\itshape Postnikov arrangement}. We will use the equivalent language of plabic graphs. 
\end{remark}

We consider a special type of plabic graphs whose strand permutation is the Grassmann permutation, given by:
$$\pi_{k,n}=\left(\begin{matrix}
1 & \ldots & n-k & n-k+1 & \ldots & n \\
k+1 & \ldots & n & 1 & \ldots & k &
\end{matrix}\right).$$

A reduced plabic graph with \(n\) boundary vertices, strand permutation \(\pi_{k,n}\), and face labels obtained by Scott's rule will be called a \(\pi_{k,n}\)-\emph{diagram}.
\begin{proposition} \textup{\cite{Pos}}
Let $\pi_{k,n}$ be the Grassmann permutation, then
\begin{itemize}
    \item The number of faces in a $\pi_{k,n}$-diagram is $k(n-k)+1$.
    \item Each face is labeled by exactly $k$ distinct indices from $[n]$.
    \item Every $k$-subset in $[n]$ occurs as the labeling set of a face in some $\pi_{k,n}$-diagram.
\end{itemize}
%
%
\end{proposition}

Consider a $\pi_{k,n}$-diagram A. Denote by $\mathbf{x}$ the set of $k(n-k)+1-n$ indeterminates indexed by labeling sets of faces which are not adjacent to the boundary of the plabic graph. Namely,
$$\mathbf{x}=\{x_K|\ K \ \text{is the labeling set of an interior face in}\ A\}.$$
And by $\mathbf{c}$ the $n$ indeterminates corresponding to the faces adjacent to the boundary. Namely,
$$\mathbf{c}=\{c_K|\ K \ \text{is the labeling set of a boundary face in}\ A\}.$$
Consider the field of rational functions in variables $\tilde{\mathbf{x}}:=\mathbf{x}\cup\mathbf{c}$. 

 A $\pi_{k,n}$-diagram $A$ defines a quiver $Q(A)$: we place vertices in faces and connect those vertices located in faces sharing a common edge; the orientation of edges is clockwise around black vertices of $A$ and counter clockwise otherwise. Let $\mathcal{A}_{k,n}$ denote the cluster algebra generated by the seed $\left(\tilde{\mathbf{x}}, Q(A)\right)$ with mutable variables $\mathbf{x}$ and frozen variables $\mathbf{c}$, and with the quiver $Q(A)$.

\begin{theorem}\textup{\cite[Theorems 2 and 3]{Sc}}
    There is an isomorphism $\varphi:\mathcal{A}_{k,n}\longrightarrow \mathbb C[\mathrm{Gr}(k,n)]$ of $\mathbb C[\mathbf{c}]$-algebras with the property that the indeterminate of $A$ corresponding to the $k$-subset $K$ is mapped to $\Delta_K$. This gives the latter the structure of a cluster algebra.
\end{theorem}
    \begin{corollary} \textup{\cite[Corollary 4]{Sc}} \label{pos-test-gen-sc}
       Let \(A\) be a \(\pi_{k,n}\)-diagram, and let \(\mathcal F(A)\) be the set of \(k\)-subsets labeling the faces of \(A\). If a point \(X\in\mathrm{Gr}(k,n)\) satisfies
\[
    \Delta_K(X)>0 \qquad \text{for every } K\in\mathcal F(A),
\]
then all Plücker coordinates \(\Delta_I(X)\) are positive.
    \end{corollary}


\subsection{Alman et al. cluster algebra}\label{sec:LMCM}
The study of positivity tests for circular minors was initiated in \cite[Section 4]{KW space}, \cite[Section 4.5.3]{K}, and continued in \cite[Section 6]{ALT}, \cite[Section 4]{Jian}. 
The general question is to describe minimal subsets of circular minors of a symmetric $n\times n$ matrix with zero row sums, whose positivity implies positivity of all circular minors. This was done in \cite{ALT} in two different ways: as a Laurent Phenomenon algebra $\mathcal{LM}_n$ and as a cluster algebra $\mathcal{CM}_n$.

The term {\itshape Laurent Phenomenon algebra} appeared in \cite{LP} and is essentially a cluster algebra which is not necessarily constructed from a quiver and whose exchange polynomials might be more general than  Laurent polynomials in \eqref{eq: exchange relation}. Laurent Phenomenon algebra $\mathcal{LM}_n$ for circular minors was defined in \cite{ALT}
 and it was proved there that each cluster of $\mathcal{LM}_n$ consisting entirely of circular minors is a positivity test for circular total positivity of a fixed symmetric $n\times n$ matrix $M$ with with zero row sums. 

\begin{definition}
    Let $M$ be an $n\times n$ matrix, whose rows and columns are indexed by two sets $I,J$. We write $M^{i_1i_2\ldots i_m,j_1j_2\ldots j_m}$ for the complementary minor of the matrix $M$ obtained by deleting the rows corresponding to $\{i_1,i_2,\ldots,i_m\}\subset I$ and $\{j_1,j_2,\ldots,j_m\}\subset J$, provided $M$ is square.
\end{definition}

While the meaning of $M^{i_1i_2\ldots i_m,j_1j_2\ldots j_m}$ depends on the underlying sets $I,J$, these sets will always be implicit. Suppose that the row $a$ appears above row $b$ and column $c$ appears to the left of column $d$, then the following Grassmann-Pl\"ucker relation holds:
\begin{equation} \label{eq: Grassmann Plücker}
    M^{a,c}M^{b,d}=M^{a,d}M^{b,c}+M^{ab,cd}M^{\emptyset,\emptyset}.
\end{equation}

Recall that $(P;Q)$ stands for a circular pair, and $D_n$ stands for the set of central circular pairs, see Definition \ref{def:Dn}. Let $\widetilde{P}$ (resp. $\widetilde{Q}$) be the sets $P$ (resp. $Q)$ with the order of indices
reversed, we will call circular pairs $(P;Q)$ and $(\widetilde{Q},\widetilde{P})$ symmetric to each other. Denote by $D_n^{\dagger}$ the union of the set of central circular pairs and those symmetric to them: 
$$D_n^{\dagger} = \{(P;Q)|(P;Q)\in D_n\}\cup \{(\widetilde{Q},\widetilde{P})|(P;Q)\in D_n \}.$$
Denote by $s$ the map $s:D_n^\dagger\to D_n^\dagger$, such that $s(P;Q) = (\widetilde{Q};\widetilde{P})$.

 \subsubsection{Laurent Phenomenon algebra $\mathcal{LM}_n$} 

 Let $n$ be odd. We now describe an undirected graph $Q_{\mathcal{LM}_n}$ that encodes the desired mutation relations among the variables of the initial seed. In the context of $\mathcal{LM}_n$ we identify the central circular pair $(P;Q)\in D_n$ with its symmetric pair $(\widetilde{Q};\widetilde{P})$.
 \begin{remark}\textup{\cite[Section 6.2]{ALT}} \label{rem: Alman condition}
 For each pair $(P;Q)\in D_n$, which is non-maximal, there is a unique equation  of the form \eqref{eq: Grassmann Plücker} such that $(P;Q)$ is one of the indices appearing in the left-hand side, and all indices of the four terms in the right-hand side are in $D_n^\dagger\cup\{(\emptyset;\emptyset)\}$. This is an immediate consequence of the fact that for an odd $n$ there is a unique way to obtain a central circular pair by deleting one index from $P$ and one from $Q$ in a given central circular pair $(P;Q)$.
 \end{remark}

The vertex set of $Q_{\mathcal{LM}_n}$ is labeled by the set $D_n\cup\{\emptyset;\emptyset\}$. Each vertex labeled by an element $(P;Q)\in D_n$ can also be thought of as labeled by $(\widetilde{Q};\widetilde{P})$.
 
 We draw edges from each vertex to the four vertices in $Q_{\mathcal{LM}_n}$ corresponding to the central circular pairs or those symmetric to them appearing on the right-hand side of the relation in Remark \ref{rem: Alman condition}. Since in our conventions vertices are labeled in two different ways ($(P;Q)$ and $(\widetilde{Q};\widetilde{P})$) one can choose the equation corresponding to $(P;Q)$ or, alternatively, the equation corresponding to $(\widetilde{Q};\widetilde{P})$. Note that circular pairs entering into the corresponding equations of the form \eqref{eq: Grassmann Plücker} differ from each other by the map $c$. Thus, the construction is independent of this choice, and each vertex labeled by non-maximal non-empty $(P;Q)\in D_n$ has degree $4$. These constitute all edges in $Q_{\mathcal{LM}_n}$. The vertices corresponding to the maximal central circular pairs and $(\emptyset;\emptyset)$ are taken to be frozen.

 $Q_{\mathcal{LM}_n}$ can be embedded in the plane in a natural way with the central circular pairs of size $k$ lying on the circle of radius $k$ centered at $(\emptyset;\emptyset)$, and all edges either along these circles or radially outwards from $(\emptyset;\emptyset)$, see Fig. \ref{fig: for lm5}. 

Let $\mathcal{LM}_n$ be the Laurent Phenomenon algebra constructed as follows: the initial seed $(\tilde{\mathbf{y}}_{D_n},Q_{\mathcal{LM}_n})$ has variables 
$$\tilde{\mathbf{y}}_{D_n}=\{y_{(P;Q)}|\ (P;Q)\in D_n\}\cup\{y_{(\emptyset;\emptyset)}\}$$
labeled by $D_n\cup\{(\emptyset;\emptyset)\}$. Recall that in our conventions vertices are labeled in two different ways ($(P;Q)$ and $(\widetilde{Q};\widetilde{P})$), the same goes for cluster variables, i.e. we could refer to $y_{(P;Q)}$ as to $y_{(\widetilde{Q};\widetilde{P})}$. The maximal pairs and the pair $(\emptyset;\emptyset)$ are frozen; the set of mutable variables is denoted by $\mathbf{y}_{D_n}$. Let us fix a local orientation around each mutable vertex $(P;Q)$ in $Q_{\mathcal{LM}_n}$, such that the edges around the vertex incident to $(P;Q)$ alternate between in- and out-edges. The exchange polynomial $F_{(P;Q)}$ is given by equation \eqref{eq: exchange relation} w.r.t.  this orientation. In this way it is designed to coincide with \eqref{eq: Grassmann Plücker}.
Note also that there is no global orientation of edges of $Q_{\mathcal{LM}_n}$ for $n\ge5$ such that they alternate between in- and out-edges. Thus, the corresponding structure is indeed a Laurent Phenomenon algebra and not a cluster algebra.

In order to study positivity tests on circular minors of a symmetric $n\times n$ matrix $M$ whose row entries sums are equal to $0$, we interpret the variables $\tilde{\mathbf{y}}_{D_n}$ as central circular minors $\{M_P^Q|\ (P;Q)\in D_n\}$. Note that in this case $M_P^Q=M_{\widetilde{Q}}^{\widetilde{P}}$ as minors symmetric to each other w.r.t. the main diagonal. The main property of the algebra $\mathcal{LM}_n$ is the following result:
\begin{lemma}\textup{\cite[Lemma 6.2.6]{ALT}}
Any cluster of $\mathcal{LM}_n$ consisting entirely of circular pairs is a positivity test.    
\end{lemma}
Note that these tests are minimal, see \cite[Section 4.5.3]{K}.

\begin{figure}[ht]
  \centering
  \hspace*{0.6cm}
  \includegraphics[width=0.8\textwidth]{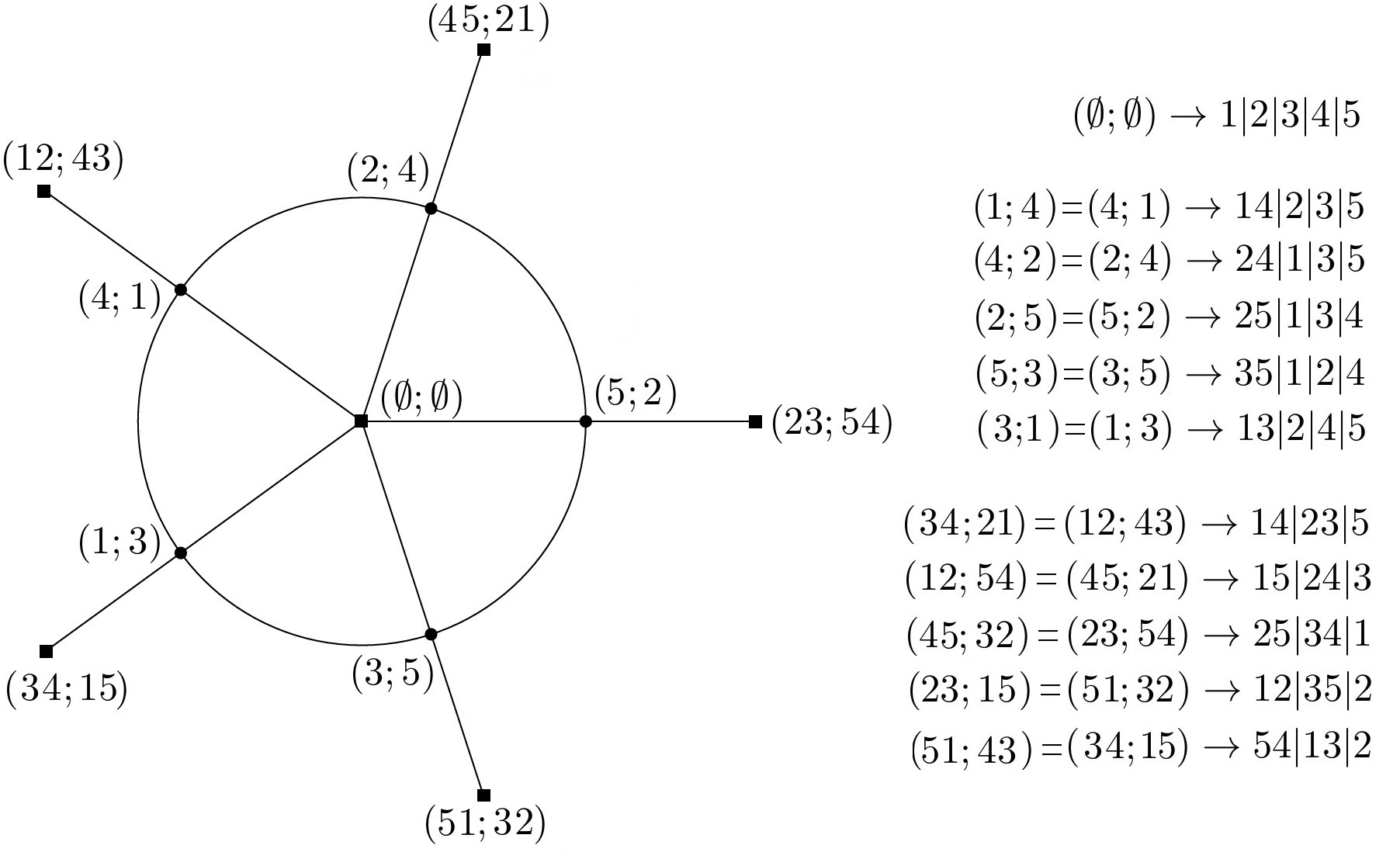}
 \caption{Initial seed of $\mathcal{LM}_5$, labeling of its variables and the non-crossing partitions corresponding to them by Lemma \ref{lemma: bijection}.}
\label{fig: for lm5}
\end{figure}

By abuse of notation, we do not use bars above the indices in the figures from now on.

\subsubsection{Cluster algebra $\mathcal{CM}_n$}  
Define a vertex set of $Q_{\mathcal{CM}_n}$ to be the set $D_n^\dagger\cup\{(\emptyset;\emptyset)\}$. Note that $|D_n^{\dagger}|=2\binom{n}{2}+1$.

Let us now describe the edges of $Q_{\mathcal{CM}_n}$. 
We draw edges from $(P;Q)\in D_n^\dagger$ to the four vertices in $Q_{\mathcal{CM}_n}$ corresponding to the central circular pairs or those symmetric to them from Remark \ref{rem: Alman condition}. 

$Q_{\mathcal{CM}_n}$ can be embedded in the plane in a natural way with the circular pairs of size $k$ lying on the circle of radius $k$ centered at $(\emptyset;\emptyset)$, and all edges either along those circles or radially outwards from $(\emptyset;\emptyset)$.

The variables of the extended cluster in the initial seed are 

$$\tilde{\mathbf{y}}_{D_n^{\dagger}}=\{y_{(P;Q)}|\ (P;Q)\in D_n^\dagger\}\cup\{y_{(\emptyset;\emptyset)}\}$$
labeled by $(P;Q)\in D_n^{\dagger}$ and $(\emptyset;\emptyset)$. The vertices labeled by pairs $(P;Q)$ such that $(P;Q)$ is maximal together with the empty-pair vertex, are frozen; the set of mutable variables is denoted by $\mathbf{y}_{D_n^\dagger}$. The edges of $Q_{\mathcal{CM}_n}$ can be oriented such that they alternate around each non-frozen vertex. Let us choose either of these orientations and let $\mathcal{CM}_n$ be the cluster algebra generated by the initial seed $(\tilde{\mathbf{y}}_{D_n^{\dagger}},Q_{\mathcal{CM}_n})$.

In order to use $\mathcal{CM}_n$ to study circular minors of a symmetric matrix $M$, we can restrict ourselves to the following: whenever we mutate at a cluster variable $v$, we then mutate at $s(v)$ immediately afterwards. We call these mutations {\itshape symmetric}. 

Let $p:D_n^\dagger \to D_n$ be the projection to the set of central circular pairs. Extend it to the $\mathbb C[\tilde{\mathbf{y}}_{D_n}\setminus \mathbf{y}_{D_n}]$-algebra homomorphism $\mathcal{CM}_n\to\mathcal{LM}_n$ by sending both $y_{(P;Q)}, y_{(\widetilde{Q};\widetilde{P})}\mapsto y_{(P;Q)}$ for each $(P;Q)\in D_n$ via the projection $p$. The next lemma provides a relation between $\mathcal{LM}_n$ and $\mathcal{CM}_n$, and thus explains how one can read circular positivity tests from $\mathcal{CM}_n$:

\begin{lemma}\textup{\cite[Lemma 6.2.5]{ALT}} \label{lemma: symmetric mutations}
    Let $L_1^{\dagger}$ be the cluster of $\mathcal{CM}_n$ that results from starting at the initial cluster and performing the sequence of mutations $\mu_{y_1},\mu_{s(y_1)},\mu_{y_2},\mu_{s(y_2)},\ldots,\mu_{y_r},\mu_{s(y_r)}$. Let $L_2$ be the cluster of $\mathcal{LM}_n$ that results from starting at the initial cluster and performing the sequence of mutations $\mu_{y_1},\mu_{y_2},\ldots,\mu_{y_r}$. Then, $p(L_1^{\dagger})=L_2$.
\end{lemma}

\begin{table}[h!]
\centering
\hspace*{-0.5cm}
\begin{tabular}{|M{1.5cm}|M{8.9cm}|M{6.5cm}|}
\hline
Algebra & Initial seed & Variable labels \\
\hline
$\mathcal{A}_{n-1,2n}$ & Dual graph to the generalized Temperley's trick applied to the electrical network $G_{n}$: $Q_{\mathcal{A}_{n-1,2n}}$, see Fig. \ref{fig: Labeling sets and central minors} and \ref{fig: quivers} & Indices of Plücker coordinates obtained by Scott's rule: $\tilde{\mathbf{x}}:=\mathbf{x}\cup\mathbf{c}$ \\
\hline
$\mathcal{A}'_{n-1,2n}$ &
Obtained from the graph above by freezing and trivializing $n$ central vertices: $Q_{\mathcal{A}_{n-1,2n}}'$ &
$\tilde{\mathbf{x}}\setminus\mathbf{x}_{\mathrm{even}}$\\
\hline
$\mathcal{CM}_{n}$ &
Obtained from $Q_{\mathcal{A}_{n-1,2n}}$ by merging all central vertices and subsequent freezing  the new merged vertex: $Q_{\mathcal{CM}_n},$ see  Fig.  \ref{fig: quivers} (Fig. \ref{fig:quiver C for even n} for even $n$) &
Central circular pairs and symmetric to them / corresponding grove partitions: $\tilde{\mathbf{y}}_{D_n^{\dagger}}$ \\
\hline
$\mathcal{CM}'_{n}$ &
Obtained from $Q_{\mathcal{CM}_n}$ by trivialization of the most central vertex: $Q_{\mathcal{CM}_n}'$&
$\tilde{\mathbf{y}}_{D_n'}\setminus\{y_{(\emptyset;\emptyset)}\}$ \\
\hline
$\mathcal{LM}_{n}$ &
Obtained from $Q_{\mathcal{A}_{n-1,2n}}$ by merging symmetric vertices: $Q_{\mathcal{LM}_n}$ see Fig.  \ref{fig: for lm5}&
Central circular pairs / corresponding grove measurements: $\tilde{\mathbf{y}}_{D_n}$ \\
\hline
\end{tabular}
\vspace{6pt}

\caption{Cluster (LP) algebras, quivers, variable labels and the
corresponding trivialized versions, which coincide due to Theorem \ref{thm: cluster algebras coinside}}
\label{tab:seed-graphs}
\end{table}

\begin{remark}
    We will define cluster algebra $\mathcal{CM}_n$ and Laurent phenomenon algebra $\mathcal{LM}_n$ for even $n$ in Section \ref{sec:even n}. The reason for doing so is that the case of even $n$ is more subtle, in particular there is no critical electrical network such that its generalized Templerey's trick corresponds to the quiver of $\mathcal{A}_{n-1,2n}$.
\end{remark}


\section{Main results} \label{sec: main results}
\subsection{Central circular minors as the cluster seed. The case of odd $n$} \label{Central circular as seed}

\begin{theorem} \label{thm: cluster algebras coinside}
Let $n$ be an odd natural number. There is a seed in the cluster algebra structure $\mathcal{A}_{n-1,2n}$ which, after freezing and subsequent trivialization of $n$ cluster variables, coincides with the initial seed of the cluster algebra structure $\mathcal{CM}_n$ in which one frozen variable is trivialized. The cluster structures generated by these seeds coincide.
\end{theorem}

\begin{proof}[Proof of the case $n=4m+1$.]

\

\subsubsection{Outline of the proof.} \label{sec: Outline of the proof}
We start with the network $G_{4m+1}$, apply generalized Temperley's trick, and denote the corresponding $\pi_{k,n}$-diagram by $S_{4m+1}$, by Lemma \ref{lemma: permutation} its strands permutation is $\pi_{n-1,2n}$. The plabic graph $S_{4m+1}$ is reduced by Lemma \ref{lemma: minimal}.

Thus, we can apply Scott's construction to make $S_{4m+1}$ an initial seed of the cluster algebra $\mathcal{A}_{n-1,2n}$. Denote by $Q_{\mathcal{A}_{n-1,2n}}$ the corresponding quiver, by $\mathbf{x}$ cluster variables, and by $\mathbf{c}$ frozen variables; $\tilde{\mathbf{x}}:=\mathbf{x}\cup\mathbf{c}$ is the extended cluster. We analyze labeling sets of faces and prove that they satisfy the conditions of Lemma \ref{lemma: bijection}. That is, a labeling set $I$ of a variable $x_I$ has the form \eqref{eq: I for lemma on Lam minors}. 

This implies that the elements of the extended cluster in the seed $(\tilde{\mathbf{x}};Q_{\mathcal{A}_{n-1,2n}})$, labeled by the labeling sets of faces, can be mapped to circular pairs according to the correspondence from Lemma \ref{lemma: bijection} and Remark \ref{rem: empty circular pair} (see Fig. \ref{fig: Labeling sets and central minors} for an illustration):

\begin{equation} \label{eq: isomorphism}
\varphi:x_I\stackrel{\text{Lemma \ref{lemma: bijection}}}{\longmapsto} y_{(P;Q)},\ \text{where } I \text{ is of the form \eqref{eq: I for lemma on Lam minors}.}
\end{equation}

\begin{figure}[ht]
  \centering
  \includegraphics[width=1.05\textwidth]{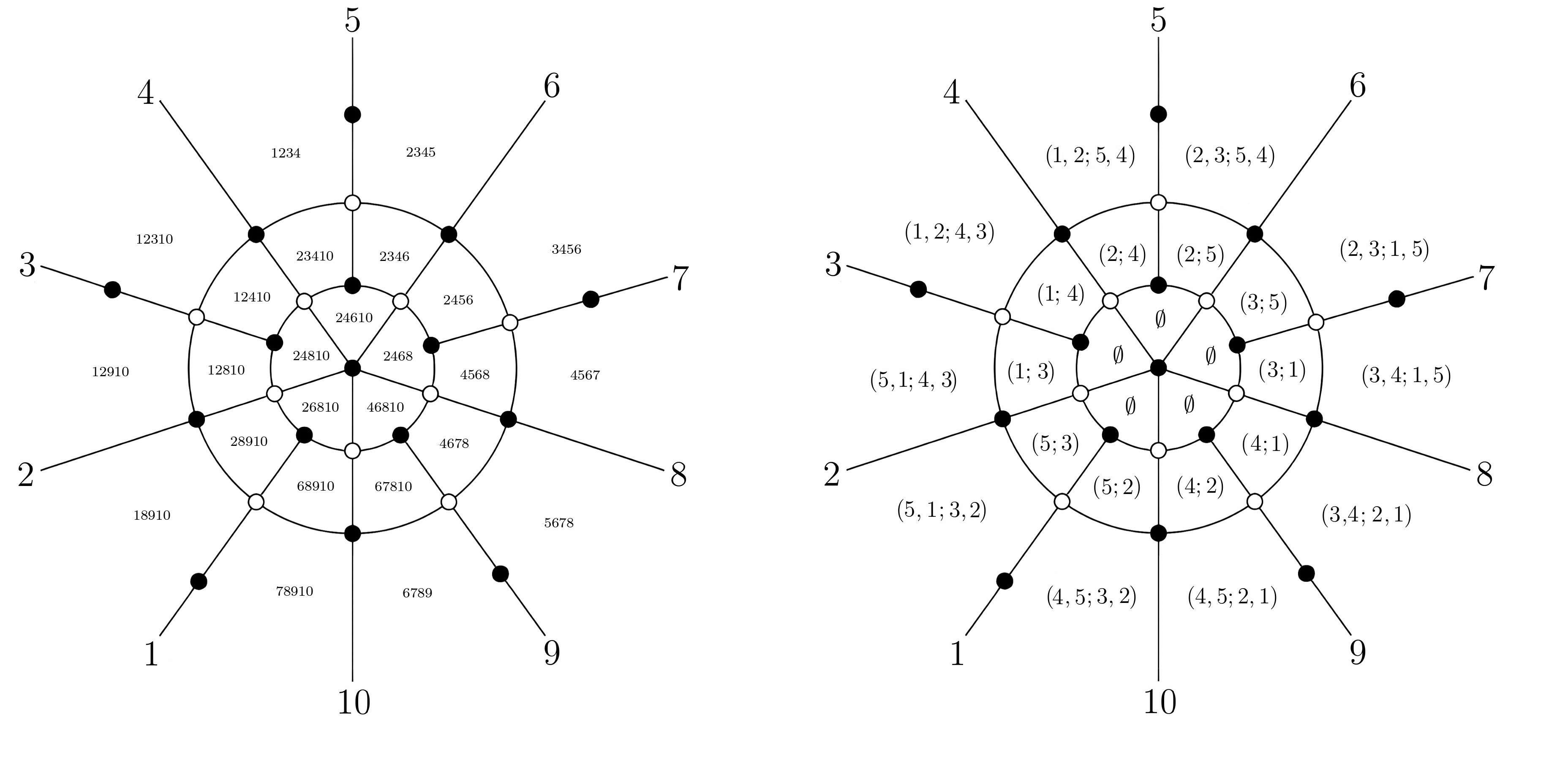}
  \caption{Labeling sets and corresponding central circular minors  }
\label{fig: Labeling sets and central minors}
\end{figure}

Moreover, we prove that these circular pairs are central and that their arrangement in all circles, except for the most central one (the latter corresponds to the case $k=0$, see Remark \ref{rem: empty circular pair}), of the seed $Q_{\mathcal{A}_{n-1,2n}}$ coincides with the arrangement in the initial seed $(\tilde{\mathbf{y}}_{D_n^{\dagger}};Q_{\mathcal{CM}_n})$ of the cluster algebra $\mathcal{CM}_n$. 

\subsubsection{Arrangement of labels.} \label{seq: Arrangement of labels} We now provide an equivalent description of Scott's rule for the labeling sets of faces of $S_{4m+1}$. Consider an angle $i,i-2$ for an even $i\in[2n]$. Throughout this proof all indices are understood modulo $2n$. We give this description by induction on the depth of the face. The base case is the face located in the most central face of the grid. Then we provide recursive formulas for the left and right parts of the angle separately, see Fig. \ref{fig: angle}. 

We denote by $I^{r,i}_{2j-1}$ (resp. $I^{r,i}_{2j}$) the labeling set of the face that belong to the right part of the angle $i,i-2$ and by $I^{l,i}_{2j-1}$ (resp. $I^{l,i}_{2j}$) those that belong to its left part, where $2j-1$ (resp. $2j$) is the distance from the most central face. Here $j$ runs over $\{1,2,\ldots,m\}.$

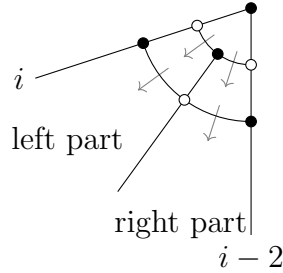
\begin{figure}[ht]
\centering
\begin{tikzpicture}
    \filldraw[fill=white] (-162:3) node [left] {$i$} ;
    \filldraw[fill=white] (-90:3) node [below] {$i-2$} ;
    
    \draw (-90:0) -- node [left] {}(-90:3);
    \draw (-162:0) -- node [left] {}(-162:3);
    \draw (-126:0.75) -- node [left] {}(-126:3);

    \draw[gray,->] (-144:0.6) -- node [left] {}(-144:1.1);
    \draw[gray,->] (-144:1.35) -- node [left] {}(-144:1.85);
    \draw[gray,->] (-108:0.6) -- node [left] {}(-108:1.1);
    \draw[gray,->] (-108:1.35) -- node [left] {}(-108:1.85);

    \draw (-162:0.75) arc (198:270:0.75);
    \draw (-162:1.5) arc (198:270:1.5);

    \filldraw[fill=black] (0:0) node [right] {} circle (2pt);
    \filldraw[fill=white] (-90:0.75) node [right] {} circle (2pt);
    \filldraw[fill=black] (-90:1.5) node [right] {} circle (2pt);
    \filldraw[fill=white] (-162:0.75) node [right] {} circle (2pt);
    \filldraw[fill=black] (-162:1.5) node [right] {} circle (2pt);
    \filldraw[fill=black] (-126:0.75) node [right] {} circle (2pt);
    \filldraw[fill=white] (-126:1.5) node [right] {} circle (2pt);

    \draw (-144:3) node {left part};
    \draw (-108:3) node {right part};

    
\end{tikzpicture} 
    \caption{Arrangement of labels}
    \label{fig: angle}
\end{figure}

    Define:
    $$e_0^i=e^{r,i}_0=e^{l,i}_0:=i+2m.$$ 
    
    The label of the central face is: 
    \begin{equation} \label{eq: label of central cell}
        I^{i,r}_0=I^{i,l}_0:=\{2j|\ j\in[n]\}\setminus\{e_0^i\}.
    \end{equation}

    With each step from the center toward the boundary we change the labeling set by the following rule: 

    For the right part: 
 \begin{equation} \label{eq: right part odd}
    e^{r,i}_{2j-1}=e^{r,i}_{0}-2j,\ \ \ I^{r,i}_{2j-1}=(I^{r,i}_{2j-2}\setminus \{e^{r,i}_{2j-1}\})\cup \{\pi(e^{r,i}_{0})+1+2(j-1)\},
    \end{equation}

    \begin{equation} \label{right part even}
    e^{r,i}_{2j}=e^{r,i}_{0}+2j,\ \ \ I^{r,i}_{2j}=(I^{r,i}_{2j-1}\setminus \{e^{r,i}_{2j}\})\cup \{\pi(e^{r,i}_{0})-1-  2(j-1)\}.
    \end{equation}
    
    For the left part:
    \begin{equation} \label{eq: left part odd}
    e^{l,i}_{2j-1}=e^{l,i}_{0}+2j,\ \ \ I^{l,i}_{2j-1}=(I^{l,i}_{2j-2}\setminus \{e^{l,i}_{2j-1}\})\cup \{\pi(e^{l,i}_{0})+2-1-2(j-1)\},
    \end{equation}
    
    \begin{equation} \label{eq: left part even}
    e^{l,i}_{2j}=e^{l,i}_{0}-2j,\ \ \ I^{l,i}_{2j}=(I^{l,i}_{2j-1}\setminus \{e^{l,i}_{2j}\})\cup \{\pi(e^{l,i}_{0})+2+1+  2(j-1)\}.
    \end{equation}
    
We now prove the above formulas for the right part. For the left part the argument is similar.

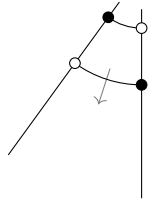
\begin{figure}[ht]
\centering
\begin{tikzpicture}

    
    \draw (-90:0.5) -- node [left] {}(-90:3);
    \draw (-126:0.5) -- node [left] {}(-126:3);

    \draw[gray,->] (-108:1.35) -- node [left] {}(-108:1.85);

    \draw (-126:0.75) arc (234:270:0.75);
    \draw (-126:1.5) arc (234:270:1.5);

    \filldraw[fill=white] (-90:0.75) node [right] {} circle (2pt);
    \filldraw[fill=black] (-90:1.5) node [right] {} circle (2pt);
    \filldraw[fill=black] (-126:0.75) node [right] {} circle (2pt);
    \filldraw[fill=white] (-126:1.5) node [right] {} circle (2pt);


    
\end{tikzpicture} 
    \caption{Step from the center toward boundary}
    \label{fig: Step from center towards boundary}
\end{figure}

During a step from one face to another (see Fig. \ref{fig: Step from center towards boundary}), which is closer to the boundary, we cross two zig-zag paths. Thus, the labeling set remains unchanged except for two elements, which are the sources of these zig-zag paths. These are the two zig-zag paths (in opposite directions) going along the arc edge, which separates these two faces. The labeling sets differ from each other in the sources of these two zig-zag paths. The task is to identify these sources. We will do it for one of the two paths. Namely, for the path which goes from the left to the right along the separating arc edge, see Fig. \ref{fig: stair shape}. 
For the other one the argument is similar.

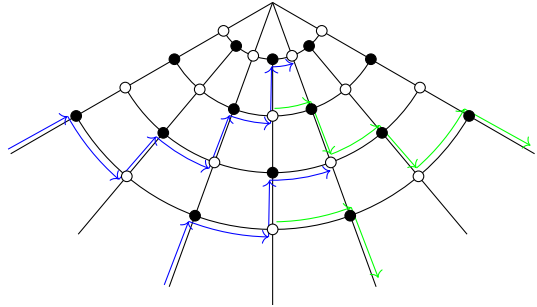
\begin{figure}[ht]
\centering
\begin{tikzpicture}

    
    \draw (-20:0) -- node [left] {}(-30:4);
    \draw (-60:0) -- node [left] {}(-70:4);
    \draw (-100:0) -- node [left] {}(-110:4);
    \draw (-140:0) -- node [left] {}(-150:4);

    \draw (-50:0.75) -- node [left] {}(-50:4);
    \draw (-90:0.75) -- node [left] {}(-90:4);
    \draw (-130:0.75) -- node [left] {}(-130:4);

    \draw[blue,->] (-151:4) -- node [left] {}(-151:3.1);
    \draw[blue,->] (209:3.1) arc (209:229:3.1);
    \draw[blue,->] (229:3.1) -- node [left] {}(229:2.35);
    \draw[blue,->] (229:2.35) arc (229:249:2.35);
    \draw[blue,->] (249:2.35) -- node [left] {}(249:1.6);
    \draw[blue,->] (249:1.6) arc (249:269:1.6);
    \draw[blue,->] (269:1.6) -- node [left] {}(269:0.85);
    \draw[blue,->] (269:0.85) arc (269:289:0.85);

    \draw[blue,->] (249:4) -- node [left] {}(249:3.1);
    \draw[blue,->] (249:3.1) arc (249:269:3.1);
    \draw[blue,->] (269:3.1) -- node [left] {}(269:2.35);
    \draw[blue,->] (269:2.35) arc (269:289:2.35);

    \draw[green,<-] (291:3.9) -- node [left] {}(291:2.9);
    \draw[green,<-] (291:2.9) arc (291:271:2.9);

    \draw[green,<-] (331:3.9) -- node [left] {}(331:2.9);
    \draw[green,<-] (331:2.9) arc (331:311:2.9);
    \draw[green,<-] (311:2.9) -- node [left] {}(311:2.15);
    \draw[green,<-] (311:2.15) arc (311:291:2.15);
    \draw[green,<-] (291:2.15) -- node [left] {}(291:1.4);
    \draw[green,<-] (291:1.4) arc (291:271:1.4);


    \draw (210:0.75) arc (210:330:0.75);
    \draw (210:1.5) arc (210:330:1.5);
    \draw (210:2.25) arc (210:330:2.25);
    \draw (210:3) arc (210:330:3);

    \filldraw[fill=white] (-150:0.75) node [right] {} circle (2pt);
    \filldraw[fill=black] (-150:1.5) node [right] {} circle (2pt);
    \filldraw[fill=white] (-150:2.25) node [right] {} circle (2pt);
    \filldraw[fill=black] (-150:3) node [right] {} circle (2pt);

    \filldraw[fill=white] (-110:0.75) node [right] {} circle (2pt);
    \filldraw[fill=black] (-110:1.5) node [right] {} circle (2pt);
    \filldraw[fill=white] (-110:2.25) node [right] {} circle (2pt);
    \filldraw[fill=black] (-110:3) node [right] {} circle (2pt);

    \filldraw[fill=white] (-70:0.75) node [right] {} circle (2pt);
    \filldraw[fill=black] (-70:1.5) node [right] {} circle (2pt);
    \filldraw[fill=white] (-70:2.25) node [right] {} circle (2pt);
    \filldraw[fill=black] (-70:3) node [right] {} circle (2pt);

    \filldraw[fill=white] (-30:0.75) node [right] {} circle (2pt);
    \filldraw[fill=black] (-30:1.5) node [right] {} circle (2pt);
    \filldraw[fill=white] (-30:2.25) node [right] {} circle (2pt);
    \filldraw[fill=black] (-30:3) node [right] {} circle (2pt);

    \filldraw[fill=black] (-130:0.75) node [right] {} circle (2pt);
    \filldraw[fill=white] (-130:1.5) node [right] {} circle (2pt);
    \filldraw[fill=black] (-130:2.25) node [right] {} circle (2pt);
    \filldraw[fill=white] (-130:3) node [right] {} circle (2pt);

    \filldraw[fill=black] (-90:0.75) node [right] {} circle (2pt);
    \filldraw[fill=white] (-90:1.5) node [right] {} circle (2pt);
    \filldraw[fill=black] (-90:2.25) node [right] {} circle (2pt);
    \filldraw[fill=white] (-90:3) node [right] {} circle (2pt);

    \filldraw[fill=black] (-50:0.75) node [right] {} circle (2pt);
    \filldraw[fill=white] (-50:1.5) node [right] {} circle (2pt);
    \filldraw[fill=black] (-50:2.25) node [right] {} circle (2pt);
    \filldraw[fill=white] (-50:3) node [right] {} circle (2pt);


    
\end{tikzpicture} 
    \caption{Arrangement of labels}
    \label{fig: stair shape}
\end{figure}

The arc edge splits a zig-zag path into two parts: one part consists of all edges that precede it along the orientation, and the other consists of all edges that follow it. Choose that which does not contain the center vertex of $S_{4m+1}$. These parts have a staircase shape (see Fig. \ref{fig: stair shape}). This is a key property that allows us to identify the sources. Namely, the angular grid distance from the boundary point of a chosen part to the angle equals   the number of arc segments in the stairs, which equals the number of radial segments in the stairs. 

The number of stairs in the chosen part of a zig-zag path on the first step from the center to the boundary (the longest blue path in Fig. \ref{fig: stair shape}), is equal to the number of radial edges in $S_{4m+1}$ going along any radius without the most central edge, which is equal to $2m$. Observing the chosen parts along the whole angle, we note that they interchange with respect to whether they contain the source of a zig-zag path or its sink (blue and green paths on Fig. \ref{fig: stair shape}). Taking this into account, we have the following. 


At step $2j-1$ the source of a segment of a zig-zag path (a blue one) between the faces is 
\begin{equation} \label{eq: odd number of source r}
(i-2)+2m-((2j-1)-1)=i-2j+2m,
\end{equation}
where the first term corresponds to the number on the right side of the angle, the second term corresponds to the number of stairs in the longest stair-shaped part (it is the chosen part at the first step), and the last term corresponds to the number of stairs passed on previous steps.

At step $2j$ the sink of a segment of a zig-zag path (a green one) between the faces is 
$$(i-1)-(2m-1)+((2j)-2),$$
where the first term corresponds to the number on the left side of the angle, the second term corresponds to the number of stairs in the longest stair-shaped part (it is the chosen part at the second step), and the last term corresponds to the number of stairs passed on previous steps.

Thus, the corresponding source is 
\begin{equation} \label{eq: even number of source r}
\pi^{-1}\left((i-1)-(2m-1)+((2j)-2)\right)=
\end{equation}
$$(i-1)-(2m-1)+((2j)-2)-(4m+1-1)=i+2j-6m-2.$$

Note that for $j=1$ the numbers given by \eqref{eq: odd number of source r} and \eqref{eq: even number of source r} are equal to $i+2m-2$ and $i-6m$ respectively, and thus are different from each other by 4 (mod $8m+2$). This proves the formulas for $e_1^{r,i}$ and $e_2^{r,i}$. Then the formulas for $e_{2j-1}^{r,i}$ and $e_{2j}^{r,i}$ follow by induction on $j$ from \eqref{eq: odd number of source r} and \eqref{eq: even number of source r}.

\subsubsection{From labels to pairs: circularity.} \label{sec: From labels to pairs: circularity} Note that, by \eqref{eq: right part odd}--\eqref{eq: left part even}, each face is labeled by an $(n-1)$-element subset of a specific form: it contains a consecutive interval of integers, together with the nearest even integers on either side, chosen so that the resulting set has cardinality $n-1$. Equivalently, the label is of the form \eqref{eq: I for lemma on Lam minors}. Consequently, Lemma~\ref{lemma: bijection} applies, and the face labels are mapped to circular pairs via \eqref{eq: isomorphism}.

\subsubsection{From labels to pairs: centrality.} \label{sec: From labels to pairs: centrality} We now prove that the labeling set located in every face mapped by \eqref{eq: isomorphism} to a central circular pair. We will do it for faces lying in the right parts of the angles; for others the proof is similar. The proof is by induction on the number $2j-1$ (or $2j$) of steps made from the center toward the boundary needed to reach a face. An illustration of this inductive proof is given in Fig. \ref{fig: centrality}.

The labeling set of the most central face is given by \eqref{eq: label of central cell}. The only non-crossing partition which is  concordant with this label $I_0^{r,i}$ is the uncrossing partition, see Example \ref{bulcon-ex}. Note also that the chord connecting $e_0^{r,i}$ and $\pi(e_0^{r,i})$ is the most central one in the sense of Section \ref{sec: Central circular minors}. Recall that $I_0^{r,i}$ contains all even numbers from $[2n]$ except $e_0^{r,i}$.

In the first step, by the formula \eqref{eq: right part odd}, we remove from the labeling set $I_0^{r,i}$ an even number and add an odd. Thus, the labeling set consists of all even numbers in $[2n]$, except for two of them, together with one odd number. This means that the non-crossing partition $(\sigma|\widetilde{\sigma})$ which is concordant with $I_1^{r,i}$ consists of two even parts and one odd chord separating them. In remains to determine the location of this chord. The closest odd number from the left to $\pi(e_0^{r,i})$ is equal to $\pi(e_0^{r,i})+1=\pi(e_0^{r,i})+1+2(1-1)$, and clearly belongs to this odd chord since it is the only odd element in the labeling set $I_1^{r,i}$ as follows from \eqref{eq: right part odd}. The other element of this odd chord should separate two even numbers ($e_0^{r,i}$ and $e_0^{r,i}+2$) which do not belong to $I_1^{r,i}$, so it is the closest from the left odd number to $e_0^{r,i}$. Taking all together, we conclude that the unique odd chord is the closest possible from the left side chord to the even chord connecting $e_0^{r,i}$ and $\pi(e_0^{r,i})$, and thus it is also the most central one. This odd chord represents the circular pair which is the image of $\Delta_{I_1^{r,i}}$ under the bijection from Lemma \ref{lemma: bijection}. Thus, we proved that the corresponding circular pair is central. This completes the base case.

For the induction step, the argument is analogous. Proceeding from the center toward the boundary, at each odd step we modify the corresponding non-crossing partition by adjoining the nearest available odd chord on the left to the collection of odd chords chosen in the previous steps; at each even step we instead adjoin the nearest available odd chord on the right. This alternation ensures that after every step the set of chosen odd chords remains as central as possible. Consequently, the associated circular pair is central.

\begin{figure}[ht]
\centering
\begin{tikzpicture}
    \draw (0,0) circle (2);

    \draw (144:2) -- node [left] {}(-144:2);
    \draw (144:2) -- node [left] {}(72:2);
    \draw (-144:2) -- node [left] {}(-72:2);
    \draw (72:2) -- node [left] {}(0:2);
    \draw (-72:2) -- node [left] {}(0:2);

    \filldraw[fill=black] (180:2) node [left=2.3pt] {$1$} circle (2pt);
    \filldraw[fill=white] (144:2) node [left=2.3pt] {$2$} circle (2pt);
    \filldraw[fill=black] (108:2) node [above=2.3pt] {$3$} circle (2pt);
    \filldraw[fill=white] (72:2) node [above=2.3pt] {$4$} circle (2pt);
    \filldraw[fill=black] (36:2) node [right=2.3pt] {$5$} circle (2pt);
    \filldraw[fill=white] (0:2) node [right=2.3pt] {$6$} circle (2pt);
    \filldraw[fill=black] (-36:2) node [right=2.3pt] {$7$} circle (2pt);
    \filldraw[fill=white] (-72:2) node [below=2.3pt] {$8$} circle (2pt);
    \filldraw[color=blue, fill=blue] (-108:2) node [below=2.3pt] {$9$} circle (2pt);
    \filldraw[fill=white] (-144:2) node [left=2.3pt] {$10$} circle (2pt);

    \draw (0:2) +(-0.15,-0.15) rectangle +(0.15,0.15);
    \draw (-72:2) +(-0.15,-0.15) rectangle +(0.15,0.15);
    \draw (144:2) +(-0.15,-0.15) rectangle +(0.15,0.15);
    \draw (-144:2) +(-0.15,-0.15) rectangle +(0.15,0.15);

    \draw (0,-2.8) node {$I^{r,2}_0$};
    \draw[color=white] (0,-3.4) node {Central $M(\bar2,\bar5)$};
\end{tikzpicture} \begin{tikzpicture}
    \draw (0,0) circle (2);

    \draw (-144:2) -- node [left] {}(144:2);
    \draw (-108:2) -- node [left] {}(108:2);
    \draw (72:2) -- node [left] {}(-72:2);
    \draw (72:2) -- node [left] {}(0:2);
    \draw (-72:2) -- node [left] {}(0:2);

    \filldraw[fill=black] (180:2) node [left=2.3pt] {$1$} circle (2pt);
    \filldraw[fill=white] (144:2) node [left=2.3pt] {$2$} circle (2pt);
    \filldraw[fill=black] (108:2) node [above=2.3pt] {$3$} circle (2pt);
    \filldraw[fill=white] (72:2) node [above=2.3pt] {$4$} circle (2pt);
    \filldraw[fill=black] (36:2) node [right=2.3pt] {$5$} circle (2pt);
    \filldraw[fill=white] (0:2) node [right=2.3pt] {$6$} circle (2pt);
    \filldraw[color=blue, fill=blue] (-36:2) node [right=2.3pt] {$7$} circle (2pt);
    \filldraw[fill=white] (-72:2) node [below=2.3pt] {$8$} circle (2pt);
    \filldraw[fill=black] (-108:2) node [below=2.3pt] {$9$} circle (2pt);
    \filldraw[fill=white] (-144:2) node [left=2.3pt] {$10$} circle (2pt);
    
    \draw (0:2) +(-0.15,-0.15) rectangle +(0.15,0.15);
    \draw (-72:2) +(-0.15,-0.15) rectangle +(0.15,0.15);
    \draw (-144:2) +(-0.15,-0.15) rectangle +(0.15,0.15);
    \draw (-108:2) +(-0.15,-0.15) rectangle +(0.15,0.15);
    \draw (0,-2.8) node {$I^{r,2}_1=I^{r,2}_0\setminus2\cup7$};
    \draw (0,-3.4) node {Central $M(\bar2,\bar5)$};
\end{tikzpicture} \begin{tikzpicture}
    \draw (0,0) circle (2);

    \draw (-144:2) -- node [left] {}(144:2);
    \draw (-108:2) -- node [left] {}(108:2);
    \draw (72:2) -- node [left] {}(-72:2);
    \draw (-36:2) -- node [left] {}(36:2);

    \filldraw[fill=black] (180:2) node [left=2.3pt] {$1$} circle (2pt);
    \filldraw[fill=white] (144:2) node [left=2.3pt] {$2$} circle (2pt);
    \filldraw[fill=black] (108:2) node [above=2.3pt] {$3$} circle (2pt);
    \filldraw[fill=white] (72:2) node [above=2.3pt] {$4$} circle (2pt);
    \filldraw[fill=black] (36:2) node [right=2.3pt] {$5$} circle (2pt);
    \filldraw[fill=white] (0:2) node [right=2.3pt] {$6$} circle (2pt);
    \filldraw[fill=black] (-36:2) node [right=2.3pt] {$7$} circle (2pt);
    \filldraw[fill=white] (-72:2) node [below=2.3pt] {$8$} circle (2pt);
    \filldraw[fill=black] (-108:2) node [below=2.3pt] {$9$} circle (2pt);
    \filldraw[fill=white] (-144:2) node [left=2.3pt] {$10$} circle (2pt);

    \draw (-144:2) +(-0.15,-0.15) rectangle +(0.15,0.15);
    \draw (-108:2) +(-0.15,-0.15) rectangle +(0.15,0.15);
    \draw (-72:2) +(-0.15,-0.15) rectangle +(0.15,0.15);
    \draw (-36:2) +(-0.15,-0.15) rectangle +(0.15,0.15);

    \draw (0,-2.8) node {$I^{r,2}_2=I^{r,2}_1\setminus6\cup7$};
    \draw (0,-3.4) node {Central $M(\bar2\bar3,\bar5\bar4)$};
\end{tikzpicture} 

    \caption{Step by step changes in the non-crossing partitions while moving from the center towards the boundary illustrates centrality.}
    \label{fig: centrality}
\end{figure}
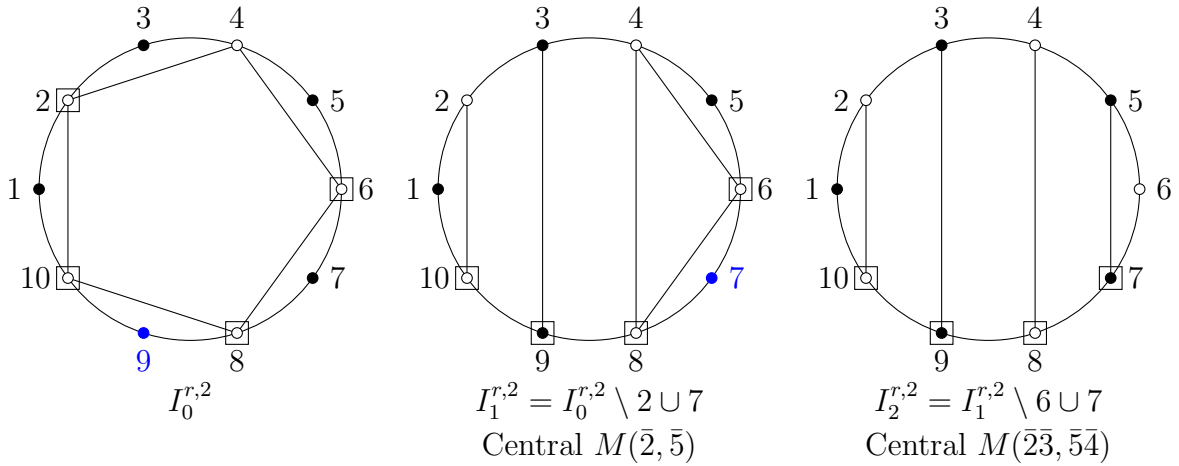

\subsubsection{Numerology.} \label{sec: Numerology} So far we have proved that all labeling sets in our seed are mapped by \eqref{eq: isomorphism} to central circular pairs. Note also that by a symmetry of $S_{4m+1}$ (or, equivalently, of the quiver $Q_{\mathcal{A}_{n-1,2n}}$) and by the fact that the formulas \eqref{eq: right part odd} -- \eqref{eq: left part even} depend linearly on the number of the angle, all central circular pairs will appear in such a way. By Lemma \ref{lemma: bijection} the number of odd elements in the labeling set is equal to the size of the corresponding circular minor. Thus, the sizes of central circular pairs in $S_{4m+1}$ (or, equivalently, of $Q_{\mathcal{A}_{n-1,2n}}$) will increase by one on each step from the center toward the boundary, as follows from \eqref{eq: right part odd} -- \eqref{eq: left part even}.

\begin{remark}
  The obtained rule of labeling sets  has important practical applications, as it can be used to solve the black box problem (the discrete Calderón problem) using the generalized chamber ansatz, see \cite[Example 3.7]{Kaz}.
\end{remark}

\begin{figure}[ht]
  \centering
  \hspace*{-1cm}
  \includegraphics[width=1.1\textwidth]{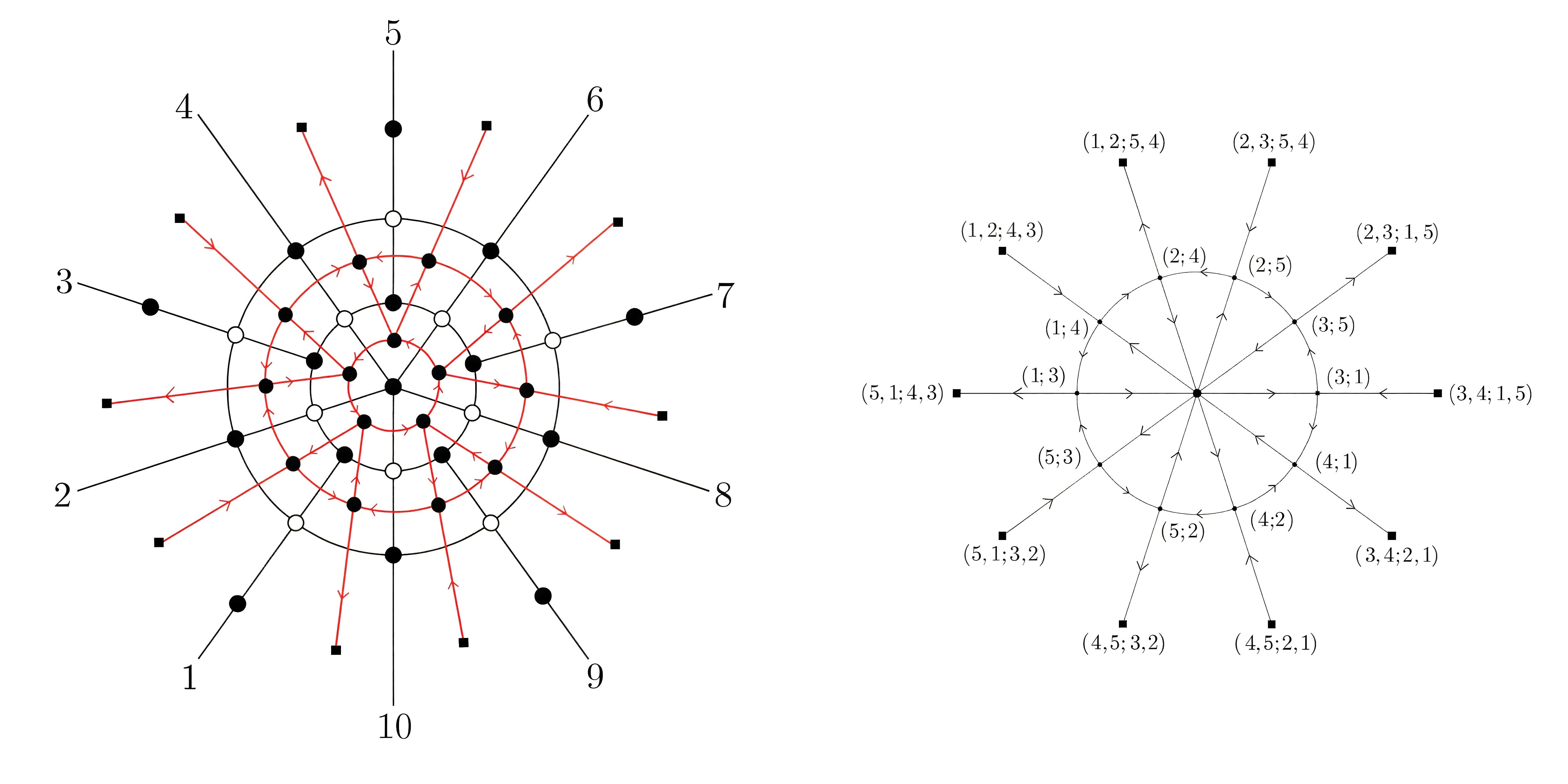}
 \caption{Quivers $Q'_{\mathcal{A}_{4,10}}$ and $Q'_{\mathcal{CM}_5}$}
\label{fig: quivers}
\end{figure}

\subsubsection{Relation to Alman et al. cluster algebra.} \label{sec: Relation to Alman et al cluster algebra} By \eqref{eq: label of central cell} the $n$ most  central variables have the form $x_I$, where $I\subset\{2,4,\ldots,2n\},\ |I|=n-1$. 
Denote by $Q_{\mathcal{A}_{n-1,2n}}'$ the quiver that is obtained by freezing and subsequently trivializing these $n$ variables.

We now prove that the quiver $Q_{\mathcal{A}_{n-1,2n}}'$ coincides with the quiver $Q_{\mathcal{CM}_{n}}'$, which is obtained from the quiver $Q_{\mathcal{CM}_{n}}$ of the initial seed of $\mathcal{CM}_n$ by trivializing the most central variable $(\emptyset;\emptyset)$. For an illustration, see Fig. \ref{fig: quivers}. Note first that they coincide as graphs. Next, by Section \ref{sec: Numerology} we know that central circular pairs located at any circle in $Q_{\mathcal{CM}_{n}}'$ coincide as a set with central circular minors located at the same circle in the quiver $Q_{\mathcal{A}_{n-1,2n}}'$. We now prove that their arrangements within each circle also coincide. In order to do so, we compare the rules for the arrangement of labels in all circles of $Q_{\mathcal{A}_{n-1,2n}}$ except the most central one with the rule for the arrangement of labels in all circles of $Q_{\mathcal{CM}_{n}}$.

We prove that in $Q_{\mathcal{A}_{n-1,2n}}$ for each except the most central vertex labeled by $(P;Q)$ for four vertices connected to this vertex the condition \eqref{eq: Grassmann Plücker} holds. Namely, it is satisfied with the pair $(P;Q)$ appearing on the left-hand side, and four other central circular pairs on the right-hand side. Since by Remark \ref{rem: Alman condition} such a substitution is unique, it follows that the arrangements of the pairs coincide. 

Denote the participating variable by $y_C$ and four its neighbors by $y_U$, $y_D$, $y_L$, $y_R$ as on Fig. \ref{fig: Arrangement of minors is in accordance with Grassmann-Plücker relation}. By Section \ref{sec: From labels to pairs: centrality} we have that $U\subset C\subset D$ and $|U|+2=|C|+1=|D|$. From these conditions we can determine $a,\ b,\ c,\ d$, which define the equation \eqref{eq: Grassmann Plücker}, uniquely. Recall that these five variables $y_C$, $y_U$, $y_D$, $y_L$, $y_R$ correspond to some variables in Scott's construction, namely, to their preimages under the map given by \eqref{eq: isomorphism}. These five variables satisfy the short Plücker relation, see \cite[Figure 7 and proof of Claim 1]{Sc}. Suppose that $y_C$, $y_U$, $y_D$, $y_L$, $y_R$ are not formal variables labeled by central circular pairs, but are central circular minors of some response matrix: $M_C$, $M_U$, $M_D$, $M_L$, $M_R$. Then we can rewrite the Plücker relation from above using Proposition \ref{circular minors via Plücker coordinates 0}, and obtain some relation on minors. This relation has the form binom $=$ binom $+$ binom and has the property that $M_C$ appears on the left-hand side and all terms on the right-hand side are central circular minors. However, by Remark \ref{rem: Alman condition} the substitution in \eqref{eq: Grassmann Plücker} is unique; thus, induced relation which we have obtained coincides with the one which defines pairs arrangement in the $Q_{\mathcal{CM}_n}$.

\begin{figure}[ht]
\centering\begin{tikzpicture}[scale=1.5]


    
    \draw (-72:0.7) -- node [left] {}(-72:3.3);
    \draw (-108:0.7) -- node [left] {}(-108:3.3);
    \draw (-140:1.8) arc (-140:-40:1.8);
    \draw (-130:2.7) arc (-130:-50:2.7);
     \draw[white] (0,-3.3) node {$n=9,\ i=2$};
     \draw[white] (0,-3.8) node {$a=6\ b=9\ c=5\ d=2$};
    \draw (0,-1.2) node {$M_U$};
    \draw (0,-2.1) node {$M_C$};
    \draw (0,-3.1) node {$M_D$};
    \draw (-1.2,-1.9) node {$M_L$};
    \draw (1.2,-1.9) node {$M_R$};
\end{tikzpicture} \phantom{aaaaaaaaaa}
\begin{tikzpicture}[scale=1.5] 

    
    \draw (-65:0.7) -- node [left] {}(-65:3);
    \draw (-115:0.7) -- node [left] {}(-115:3);
    \draw (-140:1.4) arc (-140:-40:1.4);
    \draw (-130:2.3) arc (-130:-50:2.3);
     \draw (0,-3.3) node {$n=9,\ i=2$};
     \draw (0,-3.8) node {$a=6\ b=9\ c=5\ d=2$};
    \draw (0,-1) node {\tiny{$78,34$}};
    \draw (0,-1.7) node  {\tiny{$ 789,432$}};
    \draw (0,-2.6) node {\tiny{$6789,5432$}};
    \draw (-1.6,-1.5) node {\tiny{$789,543$}};
    \draw (1.6,-1.5) node {\tiny{$678,543$}};
\end{tikzpicture} 
    \caption{Arrangement of minors is in accordance with Grassmann-Plücker relation}
    \label{fig: Arrangement of minors is in accordance with Grassmann-Plücker relation}
\end{figure}
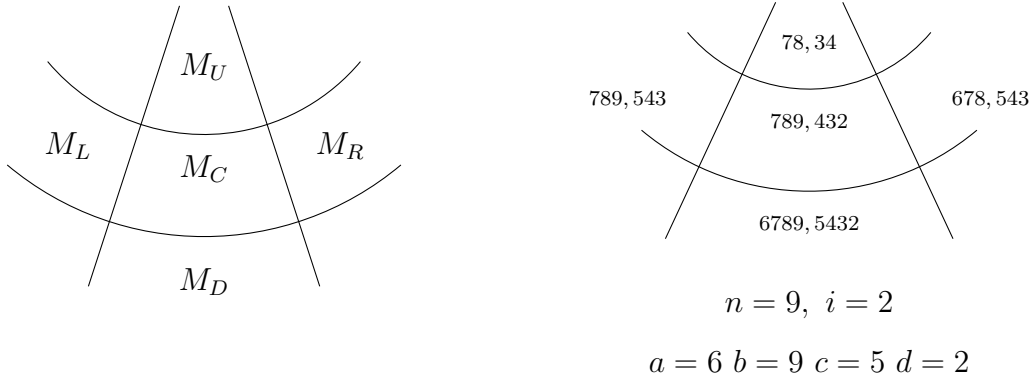

This completes the proof of the first assertion of the theorem. For the second assertion note that the map \eqref{eq: isomorphism} is an isomorphism since it maps the variables of the initial seed to the variables of the initial seed, and these seeds coincide as quivers. So we proved that the cluster algebra structures $\mathcal{A}_{n-1,2n}'$ and $\mathcal{CM}_n'$ coincide, where $\mathcal{A}_{n-1,2n}'$ and $\mathcal{CM}_n'$ are defined by the initial seeds $(\tilde{\mathbf{x}}\setminus\mathbf{x}_{\text{even}},Q_{\mathcal{A}_{n-1,2n}}')$ and $(\tilde{\mathbf{y}}_{D_n^\dagger}\setminus \{y_{(\emptyset;\emptyset)}\},Q_{\mathcal{CM}_n}')$ respectively, and $\mathbf{x}_{\mathrm{even}}:=\{x_I|\ I\subset\{2,4,\ldots,2n\},|I|=n-1\}$.
\end{proof}

\begin{proof}[Proof of the case $n=4m+3$.]
    The proof of this case is similar to the proof of the case $n=4m+1$ with the only difference in the formula for the missing even number in the most central faces:
    $$e_0^i=e^{r,i}_0=e^{l,i}_0:=i+2m+1.$$ 
\end{proof}

\subsection{Central circular minors as the cluster seed. The case of even $n$} \label{sec:even n}

In the case of even $n$, it is unknown whether there exists a well-connected electrical network such that its generalized Temperley's trick corresponds to a seed in $\mathcal{A}_{n-1,2n}$, consisting entirely of circular minors. Moreover, for $n=4$ there are only two well-connected electrical networks, namely the standard one and its dual (see Fig. \ref{fig:standart}), so it is easy to verify that neither is suitable. Nevertheless, Theorem \ref{thm: even case} states that there is a cluster in the cluster algebra $\mathcal{A}_{n-1,2n}$ which consists entirely of Plücker coordinates corresponding, by Lemma \ref{lemma: bijection}, to central circular pairs, including the empty one.

In the case of even $n$ we define the plabic graph $S_n$ by Fig. \ref{pic:quiver for even n} and the quiver $Q_{\mathcal{CM}_n}^{e}$ of the initial seed of the cluster algebra $\mathcal{CM}_n$ by Fig. \ref{fig:quiver C for even n}.

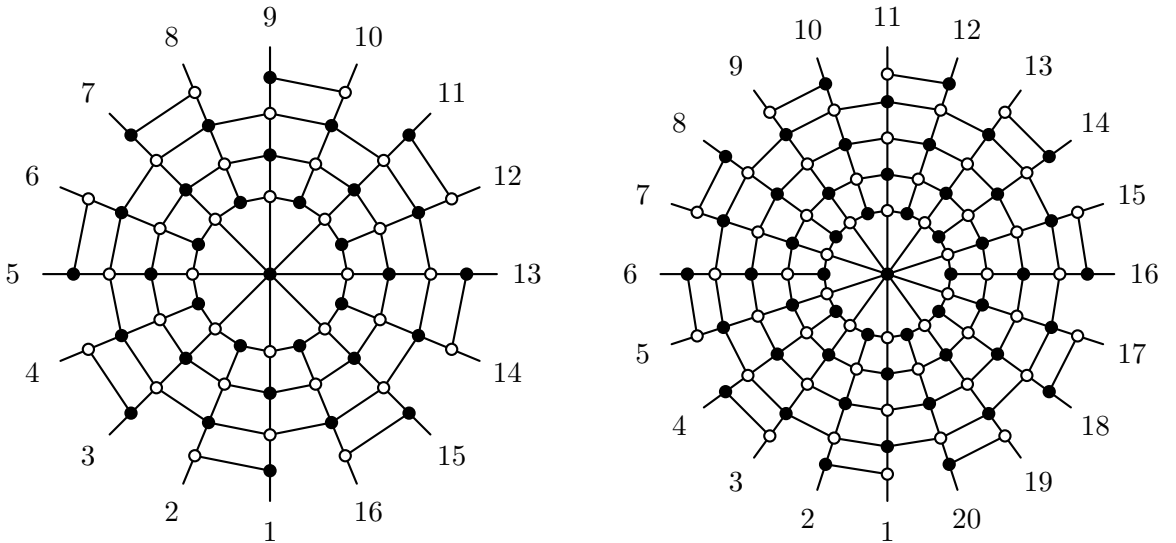
\begin{figure}[ht]

\begin{minipage}{0.48\textwidth}
\centering

\begin{tikzpicture}[
	line cap=round,
	line join=round,
	thick,
	blackv/.style={circle,draw=black,fill=black,inner sep=0pt,minimum size=4.2pt},
	whitev/.style={circle,draw=black,fill=white,inner sep=0pt,minimum size=4.2pt},
	scale=0.5
	]
	
	\def\R{6.00}   
	\def\Ra{5.20}
	\def\Rb{4.25}
	\def\Rc{3.15}
	\def\Rd{2.05}
	\def\Lab{6.8}
	
	\foreach \i in {1,...,16}{
		\pgfmathsetmacro{\ang}{-90 - 22.5*(\i-1)}
		\coordinate (B\i) at (\ang:\R);
		\coordinate (A\i) at (\ang:\Ra);
		\coordinate (C\i) at (\ang:\Rb);
		\coordinate (D\i) at (\ang:\Rc);
		\coordinate (E\i) at (\ang:\Rd);
	}
	
	
	\draw (A1)  -- (A2);
	\draw (A3)  -- (A4);
	\draw (A5)  -- (A6);
	\draw (A7)  -- (A8);
	\draw (A9)  -- (A10);
	\draw (A11) -- (A12);
	\draw (A13) -- (A14);
	\draw (A15) -- (A16);
	
	\foreach \i in {1,...,16}{
		\pgfmathtruncatemacro{\j}{mod(\i,16)+1}
		\draw (C\i) -- (C\j);
		\draw (D\i) -- (D\j);
		\draw (E\i) -- (E\j);
	}
	
	\foreach \i in {1,...,16}{
		\draw (B\i) -- (A\i);
		\draw (A\i) -- (C\i);
		\draw (C\i) -- (D\i);
		\draw (D\i) -- (E\i);
	}
	
	\node[blackv] (O) at (0,0) {};
	\foreach \i in {1,3,5,7,9,11,13,15}{
		\draw (O) -- (E\i);
	}
	
	\foreach \i in {1,...,16}{
		\ifodd\i
			\node[blackv] at (A\i) {};
			\node[whitev] at (C\i) {};
			\node[blackv] at (D\i) {};
			\node[whitev] at (E\i) {};
		\else
			\node[whitev] at (A\i) {};
			\node[blackv] at (C\i) {};
			\node[whitev] at (D\i) {};
			\node[blackv] at (E\i) {};
		\fi
	}
	
	\foreach \i in {1,...,16}{
		\node[font=\small] at ($(0,0)!{\Lab/\R}!(B\i)$) {\i};
	}
	
\end{tikzpicture}

    \end{minipage}
    \hfill
\begin{minipage}{0.48\textwidth}
\centering

    \begin{tikzpicture}[
		line cap=round,
		line join=round,
		thick,
		blackv/.style={circle,draw=black,fill=black,inner sep=0pt,minimum size=4.2pt},
		whitev/.style={circle,draw=black,fill=white,inner sep=0pt,minimum size=4.2pt},
        scale=0.5
		]
		
		\def\R{6.00}   
		\def\Ra{5.29}  
		\def\Rb{4.56}
		\def\Rc{3.60}
		\def\Rd{2.64}
		\def\Re{1.68}
		\def\Lab{6.8}
		
		\foreach \i in {1,...,20}{
			\pgfmathsetmacro{\ang}{-90 - 18*(\i-1)}
			\coordinate (B\i) at (\ang:\R);
			\coordinate (A\i) at (\ang:\Ra);
			\coordinate (C\i) at (\ang:\Rb);
			\coordinate (D\i) at (\ang:\Rc);
			\coordinate (E\i) at (\ang:\Rd);
			\coordinate (F\i) at (\ang:\Re);
		}
		
		
		\draw (A1)  -- (A2);
		\draw (A3)  -- (A4);
		\draw (A5)  -- (A6);
		\draw (A7)  -- (A8);
		\draw (A9)  -- (A10);
		\draw (A11) -- (A12);
		\draw (A13) -- (A14);
		\draw (A15) -- (A16);
		\draw (A17) -- (A18);
		\draw (A19) -- (A20);
		
		\foreach \i in {1,...,20}{
			\pgfmathtruncatemacro{\j}{mod(\i,20)+1}
			\draw (C\i) -- (C\j);
			\draw (D\i) -- (D\j);
			\draw (E\i) -- (E\j);
			\draw (F\i) -- (F\j);
		}
		
		\foreach \i in {1,...,20}{
			\draw (B\i) -- (A\i);
			\draw (A\i) -- (C\i);
			\draw (C\i) -- (D\i);
			\draw (D\i) -- (E\i);
			\draw (E\i) -- (F\i);
		}
		
		\node[blackv] (O) at (0,0) {};
		\foreach \i in {1,3,5,7,9,11,13,15,17,19}{
			\draw (O) -- (F\i);
		}
		
		\foreach \i in {1,...,20}{
			\ifodd\i
			\node[whitev] at (A\i) {};
			\node[blackv] at (C\i) {};
			\node[whitev] at (D\i) {};
			\node[blackv] at (E\i) {};
			\node[whitev] at (F\i) {};
			\else
			\node[blackv] at (A\i) {};
			\node[whitev] at (C\i) {};
			\node[blackv] at (D\i) {};
			\node[whitev] at (E\i) {};
			\node[blackv] at (F\i) {};
			\fi
		}
		
		\foreach \i in {1,...,20}{
			\node[font=\small] at ($(0,0)!{\Lab/\R}!(B\i)$) {\i};
		}
		
	\end{tikzpicture}
    \end{minipage}
    
    \caption{Quivers of $\mathcal{A}_{n-1,2n}$ for $n=8$ and $n=10$.}
\label{pic:quiver for even n}
\end{figure}

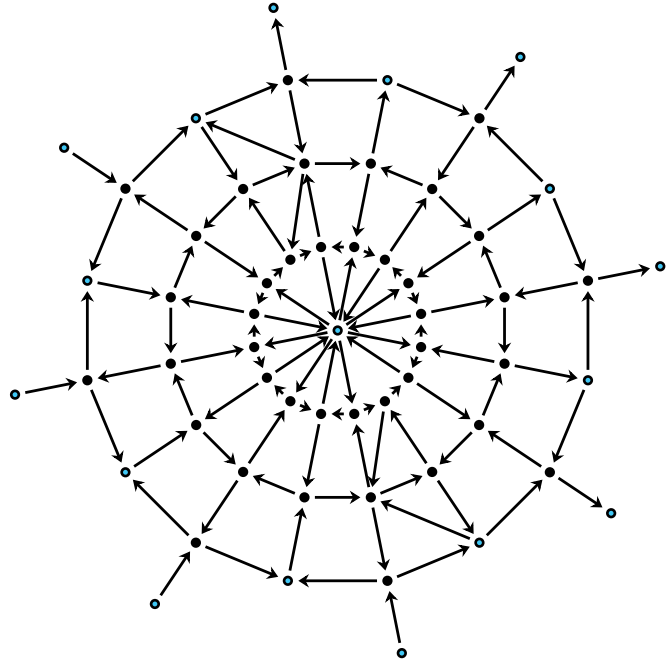
\begin{figure}[ht]

    \begin{tikzpicture}[
  >={Stealth[length=4pt,width=5pt]},
  edge/.style={draw=black,line width=1.05pt,-{Stealth[length=4pt,width=5pt]},shorten >=2pt,shorten <=2pt},
  redv/.style={circle,draw=black,line width=1pt,fill=black,inner sep=0pt,minimum size=1mm},
  bluev/.style={circle,draw=black,line width=1pt,fill=cyan!60,inner sep=0pt,minimum size=1mm},scale=0.5
]
\def\rA{2.25}
\def\rB{4.50}
\def\rC{6.75}
\def\rS{8.70}
\node[bluev] (v1) at (0,0) {};
\node[redv] (v2) at ({101.25-22.5*0}:\rA) {};
\node[redv] (v32) at ({101.25-22.5*1}:\rA) {};
\node[redv] (v34) at ({101.25-22.5*2}:\rA) {};
\node[redv] (v3) at ({101.25-22.5*3}:\rA) {};
\node[redv] (v4) at ({101.25-22.5*4}:\rA) {};
\node[redv] (v5) at ({101.25-22.5*5}:\rA) {};
\node[redv] (v6) at ({101.25-22.5*6}:\rA) {};
\node[redv] (v35) at ({101.25-22.5*7}:\rA) {};
\node[redv] (v33) at ({101.25-22.5*8}:\rA) {};
\node[redv] (v7) at ({101.25-22.5*9}:\rA) {};
\node[redv] (v8) at ({101.25-22.5*10}:\rA) {};
\node[redv] (v9) at ({101.25-22.5*11}:\rA) {};
\node[redv] (v10) at ({101.25-22.5*12}:\rA) {};
\node[redv] (v11) at ({101.25-22.5*13}:\rA) {};
\node[redv] (v12) at ({101.25-22.5*14}:\rA) {};
\node[redv] (v13) at ({101.25-22.5*15}:\rA) {};
\node[redv] (v14) at ({101.25-22.5*0}:\rB) {};
\node[redv] (v36) at ({101.25-22.5*1}:\rB) {};
\node[redv] (v37) at ({101.25-22.5*2}:\rB) {};
\node[redv] (v15) at ({101.25-22.5*3}:\rB) {};
\node[redv] (v16) at ({101.25-22.5*4}:\rB) {};
\node[redv] (v17) at ({101.25-22.5*5}:\rB) {};
\node[redv] (v18) at ({101.25-22.5*6}:\rB) {};
\node[redv] (v38) at ({101.25-22.5*7}:\rB) {};
\node[redv] (v39) at ({101.25-22.5*8}:\rB) {};
\node[redv] (v19) at ({101.25-22.5*9}:\rB) {};
\node[redv] (v20) at ({101.25-22.5*10}:\rB) {};
\node[redv] (v21) at ({101.25-22.5*11}:\rB) {};
\node[redv] (v22) at ({101.25-22.5*12}:\rB) {};
\node[redv] (v23) at ({101.25-22.5*13}:\rB) {};
\node[redv] (v24) at ({101.25-22.5*14}:\rB) {};
\node[redv] (v25) at ({101.25-22.5*15}:\rB) {};
\node[redv] (v40) at ({101.25-22.5*0}:\rC) {};
\node[bluev] (v41) at ({101.25-22.5*1}:\rC) {};
\node[redv] (v42) at ({101.25-22.5*2}:\rC) {};
\node[bluev] (v26) at ({101.25-22.5*3}:\rC) {};
\node[redv] (v43) at ({101.25-22.5*4}:\rC) {};
\node[bluev] (v31) at ({101.25-22.5*5}:\rC) {};
\node[redv] (v44) at ({101.25-22.5*6}:\rC) {};
\node[bluev] (v45) at ({101.25-22.5*7}:\rC) {};
\node[redv] (v46) at ({101.25-22.5*8}:\rC) {};
\node[bluev] (v30) at ({101.25-22.5*9}:\rC) {};
\node[redv] (v47) at ({101.25-22.5*10}:\rC) {};
\node[bluev] (v29) at ({101.25-22.5*11}:\rC) {};
\node[redv] (v48) at ({101.25-22.5*12}:\rC) {};
\node[bluev] (v28) at ({101.25-22.5*13}:\rC) {};
\node[redv] (v49) at ({101.25-22.5*14}:\rC) {};
\node[bluev] (v27) at ({101.25-22.5*15}:\rC) {};
\node[bluev] (v57) at ({101.25-22.5*0}:\rS) {};
\node[bluev] (v56) at ({101.25-22.5*2}:\rS) {};
\node[bluev] (v55) at ({101.25-22.5*4}:\rS) {};
\node[bluev] (v54) at ({101.25-22.5*6}:\rS) {};
\node[bluev] (v53) at ({101.25-22.5*8}:\rS) {};
\node[bluev] (v52) at ({101.25-22.5*10}:\rS) {};
\node[bluev] (v51) at ({101.25-22.5*12}:\rS) {};
\node[bluev] (v50) at ({101.25-22.5*14}:\rS) {};

\draw[edge] (v2) -- (v1);
\draw[edge] (v1) -- (v3);
\draw[edge] (v4) -- (v1);
\draw[edge] (v1) -- (v5);
\draw[edge] (v6) -- (v1);
\draw[edge] (v7) -- (v1);
\draw[edge] (v1) -- (v8);
\draw[edge] (v9) -- (v1);
\draw[edge] (v1) -- (v10);
\draw[edge] (v11) -- (v1);
\draw[edge] (v1) -- (v12);
\draw[edge] (v1) -- (v32);
\draw[edge] (v1) -- (v33);
\draw[edge] (v34) -- (v1);
\draw[edge] (v13) -- (v2);
\draw[edge] (v2) -- (v14);
\draw[edge] (v32) -- (v2);
\draw[edge] (v3) -- (v4);
\draw[edge] (v15) -- (v3);
\draw[edge] (v3) -- (v34);
\draw[edge] (v5) -- (v4);
\draw[edge] (v4) -- (v16);
\draw[edge] (v5) -- (v6);
\draw[edge] (v17) -- (v5);
\draw[edge] (v6) -- (v18);
\draw[edge] (v35) -- (v6);
\draw[edge] (v8) -- (v7);
\draw[edge] (v7) -- (v19);
\draw[edge] (v33) -- (v7);
\draw[edge] (v8) -- (v9);
\draw[edge] (v20) -- (v8);
\draw[edge] (v10) -- (v9);
\draw[edge] (v9) -- (v21);
\draw[edge] (v10) -- (v11);
\draw[edge] (v22) -- (v10);
\draw[edge] (v12) -- (v11);
\draw[edge] (v11) -- (v23);
\draw[edge] (v12) -- (v13);
\draw[edge] (v24) -- (v12);
\draw[edge] (v14) -- (v13);
\draw[edge] (v13) -- (v25);
\draw[edge] (v25) -- (v14);
\draw[edge] (v14) -- (v27);
\draw[edge] (v14) -- (v36);
\draw[edge] (v40) -- (v14);
\draw[edge] (v16) -- (v15);
\draw[edge] (v15) -- (v26);
\draw[edge] (v37) -- (v15);
\draw[edge] (v16) -- (v17);
\draw[edge] (v43) -- (v16);
\draw[edge] (v18) -- (v17);
\draw[edge] (v17) -- (v31);
\draw[edge] (v18) -- (v38);
\draw[edge] (v44) -- (v18);
\draw[edge] (v19) -- (v20);
\draw[edge] (v30) -- (v19);
\draw[edge] (v19) -- (v39);
\draw[edge] (v21) -- (v20);
\draw[edge] (v20) -- (v47);
\draw[edge] (v21) -- (v22);
\draw[edge] (v29) -- (v21);
\draw[edge] (v23) -- (v22);
\draw[edge] (v22) -- (v48);
\draw[edge] (v23) -- (v24);
\draw[edge] (v28) -- (v23);
\draw[edge] (v25) -- (v24);
\draw[edge] (v24) -- (v49);
\draw[edge] (v27) -- (v25);
\draw[edge] (v26) -- (v42);
\draw[edge] (v26) -- (v43);
\draw[edge] (v27) -- (v40);
\draw[edge] (v49) -- (v27);
\draw[edge] (v48) -- (v28);
\draw[edge] (v49) -- (v28);
\draw[edge] (v47) -- (v29);
\draw[edge] (v48) -- (v29);
\draw[edge] (v46) -- (v30);
\draw[edge] (v47) -- (v30);
\draw[edge] (v31) -- (v43);
\draw[edge] (v31) -- (v44);
\draw[edge] (v32) -- (v34);
\draw[edge] (v36) -- (v32);
\draw[edge] (v33) -- (v35);
\draw[edge] (v39) -- (v33);
\draw[edge] (v34) -- (v37);
\draw[edge] (v38) -- (v35);
\draw[edge] (v35) -- (v39);
\draw[edge] (v37) -- (v36);
\draw[edge] (v36) -- (v41);
\draw[edge] (v42) -- (v37);
\draw[edge] (v39) -- (v38);
\draw[edge] (v38) -- (v45);
\draw[edge] (v45) -- (v39);
\draw[edge] (v39) -- (v46);
\draw[edge] (v41) -- (v40);
\draw[edge] (v40) -- (v57);
\draw[edge] (v41) -- (v42);
\draw[edge] (v42) -- (v56);
\draw[edge] (v43) -- (v55);
\draw[edge] (v45) -- (v44);
\draw[edge] (v44) -- (v54);
\draw[edge] (v46) -- (v45);
\draw[edge] (v53) -- (v46);
\draw[edge] (v52) -- (v47);
\draw[edge] (v51) -- (v48);
\draw[edge] (v50) -- (v49);
\end{tikzpicture}

    \caption{Quiver $Q_{\mathcal{CM}_n}^{e}$ of $\mathcal{CM}_n$ for even $n$, consisting entirely of central circular pairs.}
    \label{fig:quiver C for even n}
\end{figure}

Denote
$$\mathbf{x}_{Q_n^\dagger}:=\{x_I, \mathrm{ where }\ I\in\binom{[2n]}{n-1}, I=\varphi^{-1}(P;Q)\ \mathrm{(see}\ \eqref{eq: isomorphism})|\ (P;Q)\in D_n^\dagger\}.$$
Recall
$$\mathbf{x}_{\mathrm{even}}:=\{x_I|\ I\subset\{2,4,\ldots, 2n\}, |I|=n-1\}.$$

\begin{theorem} \label{thm: even case}
The pair $(\mathbf{x}_{Q_n^\dagger}\cup\mathbf{x}_{\mathrm{even}}, Q_{\mathcal{CM}_n}^{e})$ is a seed in $\mathcal{A}_{n-1,2n}$ for even $n$.
\end{theorem}

\begin{proof}[Proof of the case $n=4m+2$.]

We first check that the strand permutation of $S_n$ is indeed $\pi_{n-1,2n}$ and that $S_n$ is reduced. There are two types of zig-zag paths on $S_n$: starting at odd and even boundary vertices. Those starting at odd vertices go clockwise and have a middle vertex on the most central circle. The part before the middle vertex consists of the first stair and then as many stairs as there are circles in $S_n$. So the total number of stairs in the first part is $2m+1$. Each stair contributes one to $\pi(i)-i$. The second part consists of as many stairs as there are circles in $S_n$. So the total number of stairs in the second part is $2m$. Thus, we get that $\pi:i\mapsto i+(2m+(2m+1))=i+(n-1)$ for an odd $i$. The case of zig-zag paths which start at even boundary vertex is similar with the only difference that they go counter clockwise and their middle vertex is the center of $S_n$. We now prove that there are no oriented lenses. Zig-zag paths of the same type cannot intersect by the same (first or second) parts more than once since they are differ by an even shift; intersections by different parts do not create oriented lenses since if two zig-zag paths of the same type intersect by different parts then these parts go in opposite directions. Zig-zag paths of different types go in opposite directions, so they also do not create oriented lenses. Thus, $S_n$ is reduced.

Similarly to Section \ref{seq: Arrangement of labels} the labeling sets of the faces of $S_{4m+2}$ are given by the formulas \eqref{eq: label of central cell-2}--\eqref{eq: left part even 4m+2} below. Consider an angle $i,i-2$ for an odd $i\in[2n]$. Throughout this proof all indices are understood modulo $2n$, with representatives in $[2n]$.

Define:
    $$e_0^i=e^{r,i}_0=e^{l,i}_0:=i+2m+1.$$ 
    
    The labeling set corresponding to the most central face is the following: 
    \begin{equation} \label{eq: label of central cell-2}
        I^{i,r}_0=I^{i,l}_0:=\{2j|\ j\in[n]\}\setminus\{e_0^i\}.
    \end{equation}

    With each step from the center to the boundary we change the labeling set by the following rule: 

    For the right part:  
 \begin{equation} \label{eq: right part odd 4m+2}
    e^{r,i}_{2j-1}=e^{r,i}_{0}-2j,\ \ \ I^{r,i}_{2j-1}=(I^{r,i}_{2j-2}\setminus \{e^{r,i}_{2j-1}\})\cup \{\pi(e^{r,i}_{0})+2(j-1)\},
    \end{equation}

    \begin{equation} \label{eq: right part even 4m+2}
    e^{r,i}_{2j}=e^{r,i}_{0}+2j,\ \ \ I^{r,i}_{2j}=(I^{r,i}_{2j-1}\setminus \{e^{r,i}_{2j}\})\cup \{\pi(e^{r,i}_{0})-2j\}.
    \end{equation}
    
    For the left part:
    \begin{equation} \label{eq: left part odd 4m+2}
    e^{l,i}_{2j-1}=e^{l,i}_{0}+2j,\ \ \ I^{l,i}_{2j-1}=(I^{l,i}_{2j-2}\setminus \{e^{l,i}_{2j-1}\})\cup \{\pi(e^{l,i}_{0})-2(j-1)\},
    \end{equation}
    
    \begin{equation} \label{eq: left part even 4m+2}
    e^{l,i}_{2j}=e^{l,i}_{0}-2j,\ \ \ I^{l,i}_{2j}=(I^{l,i}_{2j-1}\setminus \{e^{l,i}_{2j}\})\cup \{\pi(e^{l,i}_{0})+  2j\}.
    \end{equation}


    Similarly to Sections \ref{sec: From labels to pairs: circularity} and \ref{sec: From labels to pairs: centrality} it can be proven that all labeling sets of the faces of $S_{4m+2}$ correspond by \eqref{eq: isomorphism} to the circular pairs $(P;Q)\in CM_n^\dagger$. The argument from Section \ref{sec: Relation to Alman et al cluster algebra} proves that circular pairs assigned to a face and four its neighbors are related by the equation \eqref{eq: Grassmann Plücker}.

    We now prove that along the second central circle labeling sets assigned to faces interlace among those whose $D$-statistic equals to $0$ and to $2$. Consider such a face lying on the left part of the angle $i,i-2$ for an odd $i$, and the two its neighbors; one lies on the right part of the same angle and the other on the right part of the angle $i+2,i$, see Fig. \ref{fig: D interlace}.

    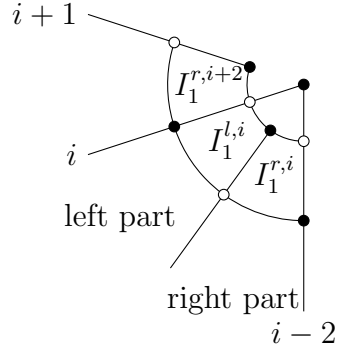
\begin{figure}[ht]
\centering
\begin{tikzpicture}
    \filldraw[fill=white] (-162:3) node [left] {$i$} ;
    \filldraw[fill=white] (-90:3) node [below] {$i-2$};
    \filldraw[fill=white] (-198:3) node [left] {$i+1$};
    
    \draw (-90:0) -- node [left] {}(-90:3);
    \draw (-162:0) -- node [left] {}(-162:3);
    \draw (-126:0.75) -- node [left] {}(-126:3);
    \draw (-198:0.75) -- node [left] {}(-198:3);


    \draw (162:0.75) arc (162:270:0.75);
    \draw (162:1.8) arc (162:270:1.8);

     \filldraw[fill=white] (162:1.8) node [right] {} circle (2pt);
      \filldraw[fill=black] (162:0.75) node [right] {} circle (2pt);
    \filldraw[fill=black] (0:0) node [right] {} circle (2pt);
    \filldraw[fill=white] (-90:0.75) node [right] {} circle (2pt);
    \filldraw[fill=black] (-90:1.8) node [right] {} circle (2pt);
    \filldraw[fill=white] (-162:0.75) node [right] {} circle (2pt);
    \filldraw[fill=black] (-162:1.8) node [right] {} circle (2pt);
    \filldraw[fill=black] (-126:0.75) node [right] {} circle (2pt);
    \filldraw[fill=white] (-126:1.8) node [right] {} circle (2pt);

    \draw (-144:3) node {left part};
    \draw (-108:3) node {right part};

    \filldraw[fill=white] (-144:1.25) node [] {$I_1^{l,i}$};
    \filldraw[fill=white] (-108:1.25) node [] {$I_1^{r,i}$};
    \filldraw[fill=white] (-180:1.25) node [] {$I_1^{r,i+2}$};

    
\end{tikzpicture} 
    \caption{$D(P;Q)=0$ and $D(P;Q)=2$ interlace among the 2'nd most central circle}
    \label{fig: D interlace}
\end{figure}

Using formulas \eqref{eq: right part odd 4m+2}--\eqref{eq: left part even 4m+2} we write down the corresponding labeling sets explicitly:
$$I_1^{l,i}=([2n]\setminus\{i+2m+1,i+2m+3\})\cup\{\pi(i+2m+1)\},$$
$$I_1^{r,i}=([2n]\setminus\{i+2m+1,i+2m-1\})\cup\{\pi(i+2m+1)\},$$
$$I_1^{r,i+2}=([2n]\setminus\{i+2m+3,i+2m+1\})\cup\{\pi(i+2m+3)\}.$$

By  Lemma \ref{lemma: bijection} we have three corresponding circular pairs.
It is easy to check that $D$-statistic changes while moving from one circular pair to another.

Next, note that the argument from Section \ref{sec: From labels to pairs: centrality} actually shows that $D$-statistic changes as we move along any fixed left (or right) part of any angle. This finishes the proof of the fact that $D$ - statistic changes in a chessboard type manner among the faces of $S_{4m+2}$.

Let $(P;Q)$ be a circular pair assigned to some face such that $(P;Q)\in(CM_n\setminus D_n)^\dagger$. Since $(P;Q)\in(CM_n\setminus D_n)^\dagger$ we have $D(P;Q)=2$. Thus, by the chessboard type manner of $D$-statistic, we conclude that its four neighbors have $D$-statistic equals $0$. If a circular pair $(P';Q')\in CM_n^\dagger$ is such that $D(P';Q')=0$, then $(P';Q')\in D_n^\dagger$. Taking all together we conclude that all four neighbors of any face with $(P;Q)\in(CM_n\setminus D_n)^\dagger$ belong to $D_n^\dagger$. 

We next prove that a chessboard structure is not only local but also global. Namely, that there is a semicircle in $S_{4m+2}$ such that all appearing circular pairs from $(CM_n\setminus D_n)^\dagger$ are concentrated there, and they are arranged there in a chessboard type manner. First, note that from a cyclic symmetry of the formulas \eqref{eq: central minors} and \eqref{eq: right part odd 4m+2}--\eqref{eq: left part even 4m+2} we conclude that inside the $k$-th circle of $S_{4m+2}$ we have a semicircle that contains all appearing circular pairs of size $k-1$ from $(CM_n\setminus D_n)^\dagger$. Consider this semicircle and the complement one. In one of them, circular pairs from $D_n^\dagger$ with $D$-statistic $0$ interlace with circular pairs from $(CM_n\setminus D_n)^\dagger$; and in the other one, circular pairs from $D_n^\dagger$ with $D$-statistic $0$ interlace with circular pairs from $D_n^\dagger$ with $D$-statistic $2$. It remains to prove why these semicircles are attached to each other for different $k$. 

First, note that from Section \ref{sec: From labels to pairs: centrality} it follows that all circular pairs appearing along any part (left or right) of any angle correspond to the same central diagonal. Moreover, due to cyclicity of the formulas \eqref{eq: right part odd 4m+2}--\eqref{eq: right part even 4m+2} for any part of any angle the centrally symmetric part of the angle in $S_{4m+2}$ contains circular pairs with the same central diagonal. It then follows that semicircles corresponding to $(CM_n\setminus D_n)^\dagger$ are attached either to each other or to the complements, i.e. they can not be attached by nontrivial intersection which is less than a half of the circle. We next prove that actually the first alternative holds. In order to do that we prove that the part of the angle containing $CM_{n+2,1}$ consists of circular pairs symmetric to the circular pairs $\{CM_{2,y}\}_{1\le y\le n/2}$. Let us compare \eqref{eq: central minors} written for $x=2$ and some $y, y+1$.

For $y$ we get:
$$a_y:=\left\lfloor\frac{2-y}{2}\right\rfloor\ \ \ \ \ b_y:=\left\lfloor\frac{-y+n+1}{2}\right\rfloor,$$

for $y+1$ we get:
$$a_{y+1}:=\left\lfloor\frac{-y+1}{2}\right\rfloor\ \ \ \ \ b_{y+1}:=\left\lfloor\frac{-y+n}{2}\right\rfloor,$$
that is $a_{y+1}=a_y-1,\ b_{y+1}=b_y$ for an even $y$; and $a_{y+1}=a_y,\ b_{y+1}=b_y-1$ for an odd $y$. This rule precisely corresponds to the rule for the change of labeling sets appearing along the angle which is given by formulas \eqref{eq: right part odd 4m+2}--\eqref{eq: right part even 4m+2}, see Section \eqref{sec: From labels to pairs: centrality} for further details. Thus the part of the angle containing $CM_{n+2,1}$ does not contain any circular pairs from $(CM_n\setminus D_n)^\dagger$, that proves the assertion on how the semicircles attachment looks like.


Finally, we explain how to construct an initial seed whose variables are all labeled by $D_n^\dagger\cup(\emptyset;\emptyset)$. The starting point is a seed in $\mathcal{A}_{n-1,2n}$ given by the plabic graph $S_{4m+2}$ whose faces are labeled by elements of $CM_n^\dagger$ using \eqref{eq: isomorphism}. We then mutate in all faces which are labeled by elements of $(CM_n\setminus D_n)^\dagger$. It is straightforward from \eqref{eq: Grassmann Plücker} that these mutations replace elements of $(CM_n\setminus D_n)^\dagger$, which are off-center by one from a central diagonal, 
by a central pair off-center by one in the other direction (recall Remark \ref{remark: central minors} and that one of those is necessarily central). The resulting graph $Q_{\mathcal{CM}_n}^{e}$ differs from $S_{4m+2}$ as follows: two symmetric edges adjacent to the center disappear, and a pair of symmetric parts of symmetric angles becomes filled with diagonals, see Fig. \ref{fig:quiver C for even n} for an illustration. Moreover, in the resulting graph the central pairs $(P;Q)$ and $(\widetilde{Q};\widetilde{P})$ are assigned to centrally symmetric faces. This follows from the cyclic symmetry of formulas \eqref{eq: central minors} and the fact that a mutation of the form \eqref{eq: Grassmann Plücker} flips two pairs parallel to the same central direction but does not flip rows with columns.
\end{proof}

\begin{proof}[Proof of the case $n=4m$.]
    The proof of this case is similar to the proof of the case $n=4m+2$ with the only difference in the formulas for labeling sets of faces of $S_{4m}$. The labeling sets of the faces of $S_{4m+2}$ are given by the formulas \eqref{eq: label of central cell-2 4m}--\eqref{eq: left part even 4m} below. Consider an angle $i,i-2$ for an odd $i\in[2n]$.

Define:
    $$e_0^i=e^{r,i}_0=e^{l,i}_0:=i+2m-1.$$ 
    
    The labeling set corresponding to the most central face is the following: 
    \begin{equation} \label{eq: label of central cell-2 4m}
        I^{i,r}_0=I^{i,l}_0:=\{2j|\ j\in[n]\}\setminus\{e_0^i\}.
    \end{equation}

    With each step from the center to the boundary we change the labeling set by the following rule: 

    For the right part:  
 \begin{equation} \label{eq: right part odd 4m}
    e^{r,i}_{2j-1}=e^{r,i}_{0}-2j,\ \ \ I^{r,i}_{2j-1}=(I^{r,i}_{2j-2}\setminus \{e^{r,i}_{2j-1}\})\cup \{\pi(e^{r,i}_{0})+2j\},
    \end{equation}

    \begin{equation} \label{eq: right part even 4m}
    e^{r,i}_{2j}=e^{r,i}_{0}+2j,\ \ \ I^{r,i}_{2j}=(I^{r,i}_{2j-1}\setminus \{e^{r,i}_{2j}\})\cup \{\pi(e^{r,i}_{0})-2(j-1)\}.
    \end{equation}
    
    For the left part:
    \begin{equation} \label{eq: left part odd 4m}
    e^{l,i}_{2j-1}=e^{l,i}_{0}+2j,\ \ \ I^{l,i}_{2j-1}=(I^{l,i}_{2j-2}\setminus \{e^{l,i}_{2j-1}\})\cup \{\pi(e^{l,i}_{0})+2-2(j-1)\},
    \end{equation}
    
    \begin{equation} \label{eq: left part even 4m}
    e^{l,i}_{2j}=e^{l,i}_{0}-2j,\ \ \ I^{l,i}_{2j}=(I^{l,i}_{2j-1}\setminus \{e^{l,i}_{2j}\})\cup \{\pi(e^{l,i}_{0})+2+  2j\}.
    \end{equation}
\end{proof}

In \cite[Definition 6.1.4]{ALT} a different set of basis circular pairs was chosen, those are called there {\itshape diametric solid minors}, and we  denote them $D_n^\bullet$. A circular pair whose $D$-statistic is $0$ is a diametric solid pair iff it is central. However, for circular pairs with $D$-statistic equal to $2$ these notions do not coincide. The difference arises since along some central diagonals there are two pairs in $CM_n$. However, a mutation in such a pair on the quiver given by $S_{n}$ gives the other pair. Thus, we get the corollary below.

Denote
$$\mathbf{x}_{(Q^\bullet_n)^\dagger}:=\{x_I, \mathrm{ where }\ I\in\binom{[2n]}{n-1}, I=\varphi^{-1}(P;Q)\ \mathrm{(see}\ \eqref{eq: isomorphism})|\ (P;Q)\in (D_n^\bullet)^\dagger\}.$$

\begin{corollary} \label{cor: Alman}
    Let $n$ be even. The set $\mathbf{x}_{(Q_n^\bullet)^\dagger}\cup\mathbf{x}_{\mathrm{even}}$ is a cluster in $\mathcal{A}_{n-1,2n}$. 
\end{corollary}

The natural question arises whether the quiver corresponding to the cluster in Corollary \ref{cor: Alman} coincides with the one given in \cite[Section 6.2]{ALT}. However, the construction of the cluster algebra $\mathcal{CM}_n$ in \cite[Section 6.2]{ALT} contains a typo in the case of even $n$. Namely, \cite[Definition 6.2.1]{ALT} produces vertices of odd degree. However, \cite[Fig. 5]{ALT} corresponding to $n=8$ is drawn correctly (i.e., in particular, it does not agree with \cite[Definition 6.2.1]{ALT}), and we have verified that in this case the two quivers: one obtained from $S_{n}$ by mutations in $(CM_n\setminus D_n^\bullet)^\dagger$ and the other given by \cite[Fig. 5]{ALT} coincide.

\begin{remark}
The proof of Theorem \ref{thm: even case} provides a construction of the initial seed of the cluster structure $\mathcal{CM}_n$ on central circular minors for even $n$, see Fig. \ref{fig:quiver C for even n}.

Following the lines of Section \ref{sec:LMCM} we define initial seed of a Laurent phenomenon algebra $\mathcal{LM}_n$ for even $n$ as a plabic graph obtained by gluing antipodal vertices of $S_n$, or equivalently, as a plabic graph obtained by holding $S_n$ along a diameter (see Fig. \ref{fig:quiver C for even n}).
\end{remark}

\subsection{Positivity tests for circular minors} \label{sec: Positivity tests for circular minors}

\begin{definition} 
Let $M=(x_{ij})_{i,j\in[n]}$ be a symmetric matrix whose row sums are zero. Define the following $(n\times 2n)$-matrix:
\begin{equation} \label{phi_n}
\Phi_n(M)=\left(\begin{matrix}
x_{11} & 1 & -x_{12} & 0 & x_{13} & 0 & \ldots & (-1)^n \\
-x_{21} & 1 & x_{22} & 1 & -x_{23} & 0 & \ldots & 0 \\
x_{31} & 0 & -x_{32} & 1 & x_{33} & 1 & \ldots & 0 \\
\vdots & \vdots & \vdots & \vdots & \vdots & \vdots & \ddots & \vdots &  
\end{matrix}\right).
\end{equation}
Denote by $\Phi'_n$ the matrix obtained from   $\Phi_n$ by deleting the first row.
\end{definition}

\begin{lemma} \label{lem:cir-min-gen}
Consider a point $\Phi_n\in\mathrm{Gr}(n-1,2n) $ of the form \eqref{phi_n}, then for any contiguous circular  pair $(P; Q), \ |P|=k$ the following holds:
    $$(-1)^k\det M(P;Q)=\Delta_{I}(\Phi'_n),$$
    here $I$ has the form \eqref{eq: I for lemma on Lam minors}.
\end{lemma}
\begin{proof}
  Since $\det M(P;Q)$ and $\Delta_{I}(\Phi'_n)$ are polynomial functions in variables $\{ x_{ij} \}$, it is sufficient to prove this identity only when  $\{ x_{ij} \}$ belong to some open set. Note that, according to Theorem \ref{Image of the Lam's embedding surjectivity}, the set of response matrices of well-connected networks defines an open set in   $\{ x_{ij} \}$, and due to Proposition \ref{circular minors via  Plücker coordinates 0} the identity holds. That ends the proof. 
\end{proof} 

The following corollary is the circular analogue of the relation between total positivity of matrices and positivity in the Grassmannian, see Section \ref{sec: Motivation}:

\begin{corollary} \label{cor: positivity tests}
The seed $(\tilde{\mathbf{y}}_{D_n^{\dagger}};Q_{\mathcal{CM}_n})$ of the cluster algebra $\mathcal{CM}_n$ gives a test for circular total positivity of $M$ which coincides with the test for positivity of $\Phi_n(M)$ given by the initial seed $(\tilde{\mathbf{x}};Q_{\mathcal{A}_{n-1,2n}})$ in the cluster algebra $\mathcal{A}_{n-1,2n}$.
\end{corollary}

\begin{proof}

    By Theorem \ref{thm: cluster algebras coinside} all labeling sets $I$ in the seed $(\tilde{\mathbf{x}};Q_{\mathcal{A}_{n-1,2n}})$ have either the form \eqref{eq: I for lemma on Lam minors} or $I\subset\{2,4,\ldots,2n-2\}$. For $I$ of the first type by Lemma \ref{lem:cir-min-gen} we have $\Delta_I(\Phi_n')=(-1)^k\det M(P;Q)$. For $I$ of the second type we have $\Delta_I(\Phi_n')=1$, so information about their positivity is excessive. Thus, all Pl\"ucker coordinates 
    $$\{\Delta_I(\Phi_n')|\ I \text{ runs over labeling sets of variables } \tilde{\mathbf{x}}\}$$
    are positive if and only if all $\{(-1)^k\det M(P;Q)|\ (P;Q)\in D_n^\dagger\}$ are positive. The latter set forms a positivity test for all circular minors of $M$ by Theorem \ref{thm: central form test}. Moreover, the set $D_n^\dagger\cup\{(\emptyset;\emptyset)\}$ indexes the extended cluster $\tilde{\mathbf{y}}_{D_n^\dagger}$. The variable $(\emptyset;\emptyset)$ has to be evaluated at $1$.  

\end{proof}

\begin{remark}
    Note also that we obtain that any seed of the cluster algebra $\mathcal{A}'_{n-1,2n}$ provides a test for circular total positivity for $M$. And vice versa, any seed of the cluster algebra $\mathcal{CM}'_{n}$ provides a test for positivity for $\Phi_n(M)$. It is an immediate corollary of the fact that the exchange relations in the cluster algebras are subtraction free and the fact that the trivialization of a variable can be interpreted as the evaluation of it at $1$. 
\end{remark}
\begin{remark}
    Let us restrict ourselves only to the Pl\"ucker mutations \cite[Fugure 7]{Sc} in the algebra $\mathcal{A}_{n-1,2n}$ and to the mutations of the form \eqref{eq: Grassmann Plücker} in the algebra $\mathcal{CM}_{n}$. Then, the seeds obtained by a sequence of mutations from the seed $(\tilde{\mathbf{x}};Q_{\mathcal{A}_{n-1,2n}})$ not mutating in the most $n$ center vertices whose cluster consists entirely of variables labeled by the sets of the form \eqref{eq: I for lemma on Lam minors} correspond to the {\itshape solid seeds} in $\mathcal{CM}_n$ (see \cite[Definition 6.2.12]{ALT}). This correspondence is given by repeating the same mutations starting from the seed $(\tilde{\mathbf{y}}_{D_n^\dagger};Q_{\mathcal{CM}_n})$. This follows from the fact that under the isomorphism from Theorem \ref{thm: cluster algebras coinside} the Pl\"ucker mutations are mapped to the mutations of the form \eqref{eq: Grassmann Plücker} and \cite[Lemma 6.2.14]{ALT}.
\end{remark}


\begin{remark} \label{rem: Phi(M) is positive}
    For a circular totally positive symmetric matrix $M$ whose row entries sums equal to $0$, we can regard it as a response matrix of the graph $G_n$ with some appropriate choice of weights $\omega$ (see \cite{CM}), by means of which each initial seed $(\tilde{\mathbf{x}};Q_{\mathcal{A}_{n-1,2n}})$ was constructed; and in this case $\Phi_n(M)=\Omega_n(e), \ e=e(\omega,G_n)$. 
\end{remark}

\begin{remark}
 The statement that an  $n\times n$ symmetric matrix $M$ with zero row  sums  is circular totally positive if and only if $\Phi_n(M)\in \mathrm{Gr}_{> 0}(n-1,2n)$ is well-known and is an immediate corollary of Theorem \ref{Set of response matrices all network well conected}, Theorem \ref{Image of the Lam's embedding surjectivity} and Theorem \ref{th: main_gr}. We suggest here a different proof of this fact, which is based only on the results on cluster algebra $\mathcal{A}_{n-1,2n}$ and Lam's embedding. In particular, it does not involve the characterization of response matrices (Theorem \ref{Set of response matrices all network well conected}).

 Suppose $M$ is a circular totally positive matrix. It means that all variables in the seed $(\tilde{\mathbf{y}}_{D_n^\dagger};Q_{\mathcal{CM}_n})$ are positive and so are all variables in the seed $(\tilde{\mathbf{x}}; Q_{\mathcal{A}_{n-1, 2n}})$ by Lemma \ref{lem:cir-min-gen}. Applying Corollary \ref{pos-test-gen-sc} we obtain that all other Pl\"ucker coordinates are positive.

Assume that $\Phi_n(M) \in \mathrm{Gr}_{> 0}(n-1, 2n)$. By \cite[Section 4]{BGKT}, \cite[Lemma 5.8]{BGKT} we have that  $\Phi_n(M) \in \mathrm{Gr}_{>0}(n-1,2n)\cap \mathbb PH$. Then by Theorem  \ref{th: the_first_ph}:
$$\Phi_n(M)=\sum _{\sigma \in \mathcal{NC}_n}L_{\sigma}w_\sigma,$$
where all  $L_{\sigma}>0$.  Therefore, formula \eqref{form:gen-min-grove} implies that the matrix $M$ is circular totally positive.   
\end{remark}

\subsection{Laurent phenomenon for circular minors} \label{sec: Laurent phenomenon for circular minors}
Kenyon and Wilson (see \cite[Corollary 8]{KW space}) have proved the following theorem:
\begin{theorem} \label{loran-ph-circ}
    Consider a well-connected electrical network, then each semicontiguous $($see Proposition \ref{circ-con-gro1}$)$ circular minor of its response matrix multiplied by $(-1)^k$, where $k$ is the size of the minor, can be expressed as a Laurent polynomial in circular minors with positive coefficients.
\end{theorem}
  This result was derived   using matrix condensation identities.  In \cite[Section 4.5.3]{K} it  has been suggested that Theorem \ref{loran-ph-circ} might  be related to the general Laurent Phenomenon for cluster algebras. This relation was established in \cite{ALT} by constructing a new cluster algebra $\mathcal{CM}_n$. Our Theorem \ref{thm: cluster algebras coinside} and Theorem \ref{thm: even case} allow us to relate Theorem \ref{loran-ph-circ} to the Laurent Phenomenon in the classical cluster algebra $\mathcal{A}_{n-1,2n}$.

\begin{proof}[Proof of Theorem \ref{loran-ph-circ}]

\ 

Let us consider a well-connected electrical network $e$ and a  semicontiguous circular minor $M^Q_P$ of its response matrix, then by Proposition \ref{circ-con-gro1} there is a  Plücker coordinate $\Delta_I$ of the point $\mathcal{L}(e)$ such that 

$$(-1)^k\det M_P^Q=\Delta_I\bigr(\Omega'_n(e)\bigl), \ \Delta_{246\dots }\bigr(\Omega'_n(e)\bigl)=\Delta_{468\dots }\bigr(\Omega'_n(e)\bigl)=\dots=1. $$ 

On the other hand, by the Laurent Phenomenon, Theorem \ref{thm: laurent phenomenon} applied to the case of $\mathcal{A}_{n-1,2n}$, we have that $\Delta_I$ can be expressed as a homogeneous  Laurent polynomial with positive coefficients in the cluster variables of the initial seed. These variables by Proposition \ref{circ-con-gro1} and Example \ref{bulcon-ex} are equal either to the central circular minors or to the Plücker coordinates, which are equal to $1$. 
\end{proof}

\subsection{Laurent phenomenon structure on the grove algebra} \label{sec: Laurent Phenomenon structure on the grove algebra}
The homogeneous  coordinate ring $\mathbb C[\mathrm{LG}(n-1,2n-2)]$ due to Theorem \ref{theorem:to def of x} is isomorphic to the homogeneous coordinate ring $\mathbb C[\mathrm{Gr}(n-1,2n)\cap\mathbb PH]$, which in turn   is naturally isomorphic to the {\itshape grove algebra} $\mathcal{G}$ introduced and studied in \cite{GLX}. Here is the main structural result.
\begin{definition}
    The grove algebra
    $\mathcal{G}=\mathbb{C}[L_{\sigma} ] / \mathcal{I}_n,$
   is a polynomial algebra over $\mathbb{C}$ in variables $\{L_{\sigma}|\ \sigma\in\mathcal{NC}_n\}$ modulo the ideal $\mathcal{I}_n$ generated by the electrical  Plücker relations:
     \begin{equation} \label{pl-el}
		\sum \limits_{\sigma \in E(I), \kappa  \in E(J) }L_{\sigma}L_{\kappa}=\sum_{I', J'}(-1)^{a(I', J')}\sum \limits_{\sigma \in E(I'), \kappa  \in E(J') }L_{\sigma}L_{\kappa},
	\end{equation}
 where $E(T)$ is the set of all non-crossing partitions that are concordant with $T$; and on the right-hand side, the summation goes over all $I'$ and  $J'$ which are obtained from $I$ and  $J$ by swapping the first index of $J$ with any index of $I;$ 
   $a(I', J')$ is the number of swaps needed to put $I'$ and  $J'$ in the correct order. 
\end{definition}

Equation \eqref{pl-el} was obtained in \cite[Proposition 5.35]{L} from equation \eqref{eq:Grovesum} and Plücker relations. 
\begin{theorem} \textup{\cite[Theorem 1.2]{GLX}}\label{th:LamGrove}
    The  homogeneous  coordinate ring $\mathbb C[\mathrm{Gr}(n-1,2n)\cap\mathbb PH]$ is isomorphic to the grove algebra
    $\mathcal{G}.$
\end{theorem}

 The variables of the initial seed of $\mathcal{LM}_n$ are labeled by central circular pairs and by $(\emptyset;\emptyset)$. Thus, using the bijection from Lemma \ref{lemma: bijection} and Remark \ref{rem: empty circular pair}, we get that the variables of the initial seed of $\mathcal{LM}_n$ are indexed by  non-crossing partitions $\sigma,$ which correspond either  to a   central circular minor or  to the partition $(\overline{1}|\overline{2}|\ldots|\overline{n})$. It follows that there is a natural map 
 \begin{equation} \label{eq: LP grove isom}
 \phi: y_{(P;Q)} \mapsto L_{\sigma} 
  \end{equation}
 indexing the variables of the initial seed $\mathcal{LM}_n$ by groves, see Fig. \ref{fig: for lm5}. 
 We extend $\phi$ to the whole $\mathcal{LM}_n$ via mutations. In order to prove that $\phi$ induces a ring isomorphism, it suffices to verify that:
 \begin{itemize}
     \item Any function   within the image of $\phi$  is regular;
     \item Any regular function lies in the image of $\phi$.  
 \end{itemize}
The first statement can be established by adapting the  proof of \cite[Proposition 7]{Sc}:
 \begin{lemma} \label{lemma: phi is well-defined}
 Consider any cluster variable from an arbitrary seed of $\mathcal{LM}_n$. Then, a function corresponding to this variable via the map $\phi$ is regular in  $\mathbb C[\mathrm{Gr}(n-1,2n)\cap\mathbb P H]$.
 \end{lemma}
 \begin{proof}
  Consider the open subset consisting of all points in $\mathrm{Gr}(n-1,2n)\cap\mathbb PH$ whose grove coordinates, corresponding to the cluster variables in the seed $(\tilde{\mathbf{y}}_{D_n},Q_{\mathcal{LM}_n})$ via $\phi$, are nonzero:
 $$U=\{X\in \text{Gr}(n-1,2n)|\ L_\sigma(X)\ne0 \text{ for every } \sigma\in\phi(\mathbf{y}_{D_n})\}.$$ 
 Make a mutation of the form \eqref{eq: Grassmann Plücker} in $(\tilde{\mathbf{y}}_{D_n},Q_{\mathcal{LM}_n})$ to obtain a new seed $\mu(\tilde{\mathbf{y}}_{D_n},Q_{\mathcal{LM}_n})$ and consider an open subset $U'$ associated with $\mu(\tilde{\mathbf{y}}_{D_n},Q_{\mathcal{LM}_n})$ in a way similar to what was done above:
 $$U'=\{X\in \text{Gr}(n-1,2n)|\ L_\sigma(X)\ne0 \text{ for every } \sigma\in\phi(\mu(\mathbf{y}_{D_n}))\}.$$ 
 
 Any cluster variable $y$ we can by Theorem \ref{thm: laurent phenomenon} represent as a Laurent polynomial in the variables $\tilde{\mathbf{y}}_{D_n}$ with a denominator in $\mathbf{y}_{D_n}$, so $\phi(y)$ is regular on $U$. Similarly, we can represent $\phi(y)$ as a Laurent polynomial in the variables $\mu(\tilde{\mathbf{y}}_{D_n})$ with a denominator in $\mu(\mathbf{y}_{D_n})$, so $\phi(y)$ is regular on $U'$. Thus, $\phi(y)$ is regular on $U\cup U'$. Since the codimension of the complement of $U\cup U'$ is $2$ or greater, $\phi(y)$ is regular on the whole $\mathbb C[\mathrm{Gr}(n-1,2n)\cap \mathbb PH]$.   
 \end{proof}

Now, we prove the surjectivity of $\phi$ on a  localization of the grove algebra. 
Consider the open projective subvariety 
$$\mathcal{X}_{\mathcal{E}}:=\{X\in\mathrm{Gr}(n-1,2n)\cap\mathbb P H| \ L_{unc}\neq 0\}.$$ 

Note that by Remark \ref{rem:Lunc}, $L_{unc} = \Delta_I$ for  $I\in\binom{[2n]}{n-1}$ such that $I\subset\{2,4,\ldots,2n\}$. 
 The variety $X_\mathcal{E}$ is the noncompactified space of electrical networks, see Remark \ref{rem:Lunc2}. Thus, the ring $\mathbb C[\mathcal X_{\mathcal{E}}]$ is the coordinate ring of {\itshape the noncompactified space of electrical networks}. Note that by Theorem \ref{th:LamGrove} the homogeneous  coordinate ring $\mathbb C[\mathcal{X}_{\mathcal{E}}]$ is isomorphic to the localization $\mathcal{G}[L_{unc}]$.
 
 \begin{lemma} \label{Parametrization of the top-cell}
Any point $X\in \mathcal{X}_{\mathcal{E}}$ can be represented in the form \eqref{phi_n}.
Moreover, every matrix $\Phi_n$ of the form \eqref{phi_n} belongs to $\mathcal{X}_{\mathcal{E}}$.

In the representation above, each $x_{ij}, i\neq j$ is related to the grove coordinates of $X$ as follows:
\begin{equation*}
x_{ij}=-\dfrac{L_{ij|*|*|\dots}}{L_{unc}}.
\end{equation*}
 \end{lemma}
    The first part of the statement is proved in \hyperref[gen-form-ap]{Appendix A}.
    The second part follows from \cite[Section 4]{BGKT} and \cite[Lemma 5.8]{BGKT}. The third part follows from  direct computations and \cite[Lemma $2.27$]{BGKT}.
\begin{corollary} \label{col-gr-ij}
    Consider any point $X\in \mathcal{X}_{\mathcal{E}}$, then
    each of its grove coordinates $L_{\sigma}$ can be represented as a polynomial in $\{L_{ij|*|*|\dots}|\ 1\le i,j\le n,\ i\ne j\}\cup\{L^{\pm}_{unc}\}.$
\end{corollary}
 \begin{proof}
     By  \cite[Proposition 5.19]{L},  for an arbitrary $\sigma$ we have:
$$L_{\sigma}=\sum \limits_{J } a_{J, \sigma}\Delta_{J}(X).$$
By Lemma \ref{Parametrization of the top-cell} there exists a representative of $X$ defined by the matrix $\Phi_n$ in the form \eqref{phi_n}. Since $\Delta_{J}(X)=\Delta_J(\Phi_n')L_{unc}$ and by Lemma \ref{Parametrization of the top-cell} $\Delta_{J}(\Phi_n')$ is a polynomial in $\frac{L_{ij|*|*|\dots}}{L_{unc}}$, we obtain the statement.
 \end{proof}
 Following the correspondence from Remark \ref{rem: empty circular pair}, we denote $y_{(\emptyset;\emptyset)}:=y_{unc}$.
\begin{theorem} \label{thm: LP grove algebra}
     The localization
    $\mathcal{G}[L_{unc}]$ of the grove algebra $\mathcal{G}$ is isomorphic to the Laurent Phenomenon algebra $\mathcal{LM}_n$ considered as the $\mathbb{C}[y^{\pm1}_{unc}]$-algebra. 
\end{theorem}
\begin{proof}
     By Lemma \ref{lemma: phi is well-defined} $\phi:\mathcal{LM}_n\to\mathcal{G}[L_{unc}]$ is a well-defined homomorphism of $\mathbb C[y_{unc}^{\pm 1}]$-algebras.
     
     It is claimed in \cite[Theorem 6.2.16]{ALT} 
     that any cluster variable of the form $y_{(\{\bar i\};\{\bar j\})}$ can be reached in $\mathcal{LM}_n$ by mutations of the form \eqref{eq: Grassmann Plücker} from the initial seed (in \hyperref[sec:appendix B]{Appendix B}\ we provide a proof of this statement). Consider the homomorphism $\psi:\mathcal{LM}_n\to\mathbb C[\{x_{ij}|\ i\ne j\}]$ given by $y_{(P;Q)}\mapsto (-1)^{|P|}M(P;Q)$ and $y_{(\emptyset;\emptyset)}\mapsto 1$ \footnote{$\psi$ is an isomorphism of $\mathbb C$-algebras as was proved in \cite[Theorem 6.2.16]{ALT}.}, here $M =(x_{ij})_{i,j\in[n]}$. By the Laurent Phenomenon \cite[Theorem 5.1]{LP} applied to $\mathcal{LM}_n$, we can express $y_{(\{\bar i\};\{\bar j\})}$ as a homogeneous Laurent polynomial $f$ in $\tilde{\mathbf y}_{D_n^\dagger}$. Then, $\psi(f)=x_{ij}$.
     
     By homogeneity of $f$ and Lemma \ref{Parametrization of the top-cell} we get that $\phi(f)=L_{ij}$. In other words, all $L_{ij}|\ 1\le i,j\le n,\ i\ne j$ lie in the image of $\phi$. On the other hand, due to Corollary \ref{col-gr-ij}, each $L_{\sigma},\sigma\in\mathcal{NC}_n$ can be expressed as a  polynomial in $\{L_{ij|*|*|\dots}|\ 1\le i,j\le n,\ i\ne j\}\cup\{L^{\pm}_{unc}\}$, and therefore also lies in the image of $\phi$. Thus, it is true for any regular function in $\{L_{\sigma}|\ \sigma\in \mathcal{NC}_n\}$, which completes the proof.
\end{proof}

By analogy with the constructions of the  cluster algebras  on the ring of regular functions on the  Grassmannian and on its top positroid cell, we expect that the ring of regular functions on the  Lagrangian Grassmannian also admits a structure of a Laurent Phenomenon algebra.
\begin{problem}
Extend the result of Theorem \ref{thm: LP grove algebra} to the whole    grove algebra $\mathcal{G}$.  
\end{problem}

 \section{Discussion on cells of lower dimensions}
\label{sec:cells}

The construction of the poset $EP_n$ on electrical networks was originated in \cite{CM, CGV} and was systematically examined   in \cite[Section 3]{ALT}. In our work we  study positivity tests for the circular total positivity of symmetric $n\times n$ matrices with row entry sums equal to $0$. Such circular totally positive matrices are the response matrices of well-connected electrical networks, which correspond to the top cell of the poset $EP_n$. 
In this section, we discuss the direction for obtaining  tests for circular total positivity of response matrices of electrical networks that belong to lower-dimensional cells of the poset $EP_n$. This question was addressed previously in the following conjecture:

\begin{conjecture} \textup{\cite[Conjecture 2]{KW}} \label{conj:positroidpositivity}
   A matrix $M \in \mathrm{Mat}_{n \times n}(\mathbb{R})$ is the response matrix of a circular electrical network belonging to   a cell of co-dimension $r$ in the poset $EP_n$ if and only if:
   \begin{itemize}
       \item The first two conditions from  Theorem \ref{gen-theorem} are satisfied;
       \item There exist two sets $(S_1, S_2)$ such that $|S_1|=\frac{n(n-1)}{2}-r,  |S_2|=r$ and the circular positivity of all circular minors from $S_1$ and the vanishing of all circular minors from $S_2$ imply circular non-negativity of all other circular minors.
   \end{itemize}
\end{conjecture}

A cluster algebra structure on positroid varieties, whose initial seeds are also given by plabic graphs, was constructed in \cite{GL}. We suggest an analogous approach to what was done in the previous sections of this paper for the top positroid variety: to translate positivity tests on positroid varieties to positivity tests on cells of $E_n$. As for the top cell, consider an electrical network $e$ (for simplicity, we assume that $e$ contains no isolated boundary nodes).
Then we 
construct a reduced plabic graph associated with $e$ using generalized Temperley's trick and label its faces by Plücker coordinates following Scott's rules.
    The obtained object is an initial seed of the cluster algebra on an appropriate positroid cell of $\mathrm{Gr}_{\geq 0}(n-1, 2n)$, see \cite{GL, PaS}. The cluster of this seed defines a non-negativity test for a point $\Omega_n(e) \in \mathrm{Gr}_{\geq 0}(n-1, 2n)$ defined by $e$. That is, the positivity of all Plücker coordinates in the set guaranties the non-negativity of all other Plücker coordinates. 

A natural candidate for an electrical network to start with is one given by $G_n$ with some edges contracted, so that the obtained electrical network belongs to the relevant cell. However, as the following example shows, the main difference from the top-cell case is that we do not have a cluster in the sleeve for which appearing Plücker coordinates correspond to a single circular pair.
\begin{example}
    Consider an electrical network on the graph $G_5$ with one contracted spike, see Fig. \ref{lowdem}. In this case, any extended cluster, corresponding to the codimension $1$  cell  of $EP_5$, contains the frozen variable labeled by $5789$, which does not correspond to any circular pair.
\end{example}

\begin{figure}[ht]
  \centering
  \includegraphics[width=1.0\textwidth]{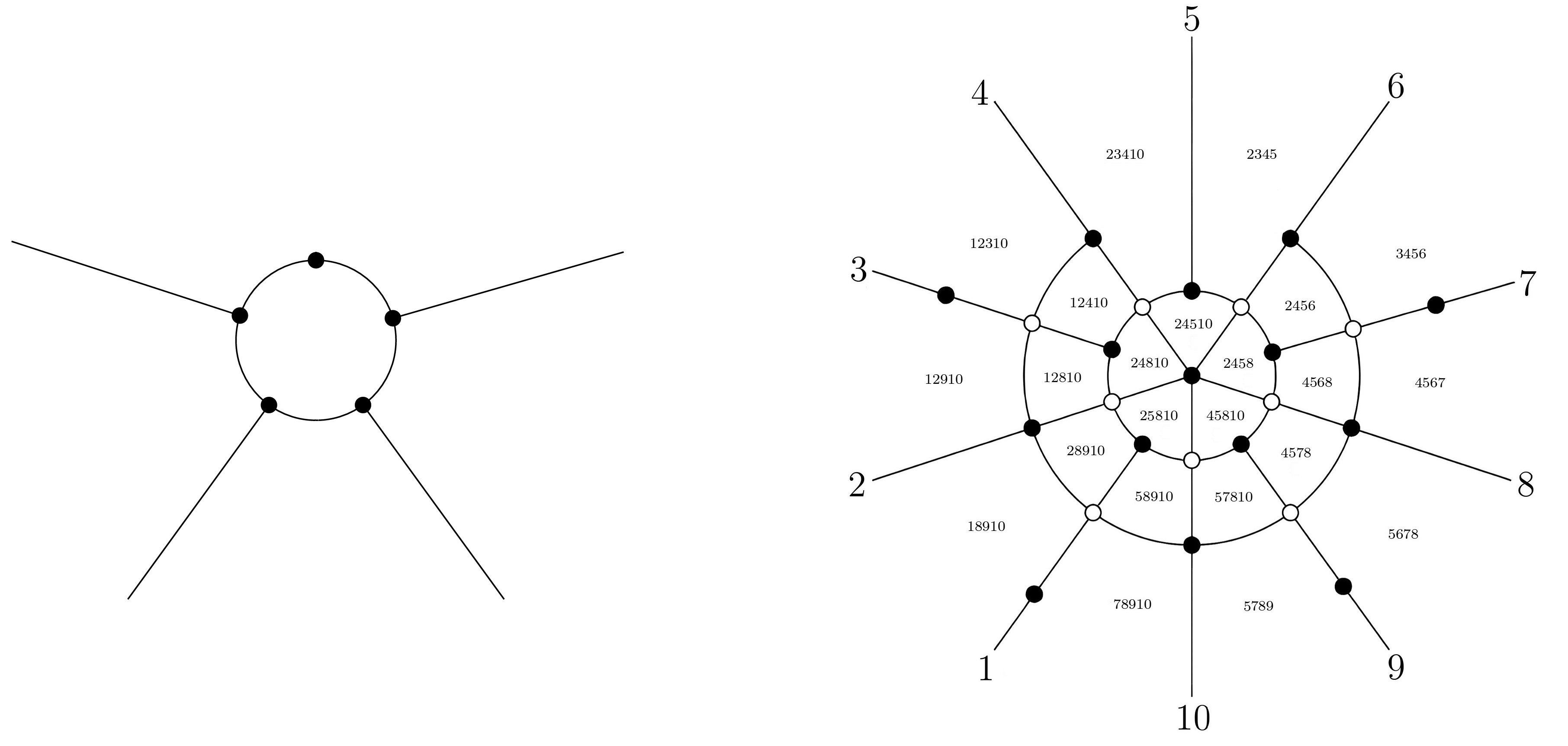}
 \caption{The codimension $1$ cell of $EP_5$.}
\label{lowdem}
\end{figure}

 However, for any electrical network $e$, every Plücker coordinate of $\Omega_n(e)$ can be expressed as a linear combination of minors of the response matrix $M_R(e)$ (due to the Binet--Cauchy formula applied to the decomposition \cite[Lemma 4.7]{BGKT}). Thus, we can obtain (non-minimal) tests in terms of the positivity of these linear combinations. 

Other positivity tests for the cells of codimension 1 and 2 were found in \cite[Section 4]{Jian}, it is also worth to compare these constructions.



\appendix

\section{Parametrization of \texorpdfstring{$\mathrm{Gr}(n-1,2n) \cap\mathbb P H$}{Gr}} 
\label{gen-form-ap}
In this appendix, we prove Lemma \ref{Parametrization of the top-cell}. For this purpose, let us define the subspace $V$ and  its fixed  basis: 
\begin{gather} \label{Basis of V}
 V=\{v\in \mathbb R^{2n}|\sum\limits_{i=1}^{n}(-1)^iv_{2i}=0,\ \sum\limits_{i=1}^n(-1)^iv_{2i-1}=0\}, \\\nonumber
 v_1=(1,0,1,0,\ldots,0,0,0),\ v_2=(0,1,0,1,\ldots,0,0,0),\ldots,v_{2n-2}=(0, 0,0,0,\ldots,1,0,1).
\end{gather}

By \cite[Lemma 4.2]{BGKT}, there is a natural map $\mathcal G: \mathrm{Gr}(n-1,2n)\cap\mathbb P H \to \mathrm{Gr}(n-1, V).$ Let $M$ be a matrix representing a point  $X \in \mathrm{Gr}(n-1,2n)\cap\mathbb P H $. Expanding its rows in the basis $v_1, \dots, v_{2n-2}$, we obtain the matrix $\widetilde{M}=MB_n^{-1}$ representing the point $\widetilde{X}=\mathcal G(X) \in \mathrm{Gr}(n-1, V) $. Here $B_n$ is the matrix whose rows are vectors $v_i$ and  $B_n^{-1}$ is the following matrix:

\begin{equation*}
\label{eq:Binverse}
B_n^{-1} = \left(
\begin{array}{cccccc}
1& 0  & -1  & 0 & \dots&0  \\
0& 1  & 0  & -1 & \dots&(-1)^{n+1}\\
0& 0  & 1  & 0 & \dots&0\\
0&0&0&1&\dots&(-1)^n\\
\vdots&\vdots&\vdots&\vdots&\ddots&\vdots\\
0& 0  & 0  & 0 & \dots&1\\
0& 0  & 0  & 0 & \dots&0\\
0& 0  & 0  & 0 & \dots&0
\end{array}
\right).
\end{equation*} 
\begin{theorem}  \textup{\cite[Theorem 5.9]{BGKT},\cite[Theorem 1.5]{CGS}} \label{lem: map_bij} 
The map $\mathcal G$ provides a bijection between $\mathrm{Gr}(n-1,2n)\cap\mathbb P H$ and \textit{Lagrangian Grassmannian}  $\mathrm{LG}(n-1, V)$ with respect to the non-degenerate  skew-symmetric form $\Lambda_{2n-2}:$
\begin{equation*} \label{8}
\Lambda_{2n-2} = \left(\begin{array}{cccccc}
\phantom{\vdots}0 & 1 & 0 & \cdots &  \cdots & 0   \\
\phantom{\vdots}-1 & 0 & -1 & 0 & & \vdots   \\
\phantom{\vdots}0 & 1 & 0 & 1 & \ddots & \vdots  \\
\vdots & \ddots  & \ddots  & \ddots & \ddots & 0   \\
 \vdots &    & \ddots &  1 & 0 & 1  \\
\phantom{\vdots}0 &  \cdots & \cdots  & 0 &  -1 & 0
\end{array}\right).  
\end{equation*}
In particular, we have:  
$$\widetilde{M}\Lambda_{2n-2}\widetilde{M}^T=0.$$
\end{theorem}

By $\Delta_I$ denote the set of  Plücker coordinates of a point $X\in \mathrm{Gr}(n-1,2n)\cap\mathbb P H$ and by $\Delta_I^{\bullet}$ the set of  Plücker coordinates of the point $\widetilde{X}:=XB_n^{-1}\in \mathrm{LG}(n-1,V)$. The lemma below provides a formula to express  Plücker coordinates of $\mathrm{Gr}(n-1,2n)\cap\mathbb P H$ in terms of  Plücker coordinates of $\mathrm{LG}(n-1,V)$. 
\begin{lemma} \label{Connection between  Plücker coordinates of Lam's embedding and of LGr}
Let $I=\{i_1,i_2,\ldots,i_{n-1}\}$, then
$$\Delta_I=\sum\limits_{i_1',i_2',\ldots,i_{n-1}'}\Delta_{i_1',i_2',\ldots,i_{n-1}'}^{\bullet},$$
where the sum is over all $i_k'\in [1,\ldots,2n-2]\cap\{i_k-2,i_k\}$.
\end{lemma}
\begin{proof}
     Represent $\widetilde{X}$ as a point of the Lagrangian Grassmannian $\mathrm{LG}(n-1,V)$ identified with its image under the  Plücker embedding: $$X=\sum \Delta_{j_1, \dots, j_{n-1}}^{\bullet}v_{j_1}\wedge v_{j_2}\dots \wedge v_{j_{n-1}}\in \bigwedge^{n-1}V.$$ Using the fact that $v_{j}=e_{j-2}+e_{j}$, we obtain:
$$\sum \Delta_{j_1, \dots, j_{n-1}}^{\bullet}(e_{j_1-2}+e_{j_1})\wedge (e_{j_2-2}+e_{j_2}) \wedge \dots \wedge (e_{j_{n-1}-2}+e_{j_{n-1}})=$$
$$=\sum \Delta_{i_1, i_2, \dots, i_{n-1}} e_{i_1}\wedge e_{i_2} \wedge  \dots \wedge e_{i_{n-1}}.$$
\end{proof}

\begin{example}
Consider a  Plücker coordinate $\Delta_{235},$ then
$i_1' \in \{0, 2\} \cap  [1,\ldots, 6],$ $i_2' \in \{1, 3\} \cap  [1,\ldots, 6],$ $i_3' \in \{3, 5\} \cap [1,\dots, 6]  $ and
$$\Delta_{235}=\Delta_{235}^{\bullet}+\Delta_{215}^{\bullet}+\Delta_{233}^{\bullet}+\Delta_{213}^{\bullet}=\Delta_{235}^{\bullet}-\Delta_{125}^{\bullet}-\Delta_{123}^{\bullet}.$$
\end{example}

\begin{corollary}\label{Even coordinate}
    The following holds:
    $$\Delta_{246\dots 2n-2}=\Delta_{246\dots 2n-2}^{\bullet}.$$
\end{corollary}
\begin{proof}
    This is immediate from Lemma \ref{Connection between  Plücker coordinates of Lam's embedding and of LGr}.
\end{proof}

\begin{proof}[Proof of Lemma \ref{Parametrization of the top-cell}]
Consider an arbitrary $X\in\mathcal{X}_\varepsilon$. In the following we assume that $X$ is identified with its representative given by a matrix. By \cite[Lemma 4.2]{BGKT} $X\in \bigwedge^{n-1}V$, thus $X\subset V$. Expand the rows of the matrix $X$ in the basis $v_1, \dots, v_{2n-2}$ and obtain the matrix $\widetilde{X}:=XB_n^{-1}\in \mathrm{LG}(n-1,V)$.
Since $X\in\mathcal{X}_\varepsilon$ we have, by Corollary \ref{Even coordinate}, that $\Delta_{246\ldots 2n-2}^{\bullet}\ne 0$. Denote by $D$ the matrix composed of even columns of $\widetilde{X}$. Since $\Delta_{246\ldots 2n-2}^{\bullet}\ne 0$ we have that $\det(D)\ne 0$. Thus we can multiply $\widetilde{X}$ by $D^{-1}$ on the left:
\begin{equation}
\label{eq:Binverse1}
D^{-1}\widetilde{X}= \left(
\begin{array}{ccccccc}
*& 1  & *  & 0 & \dots& * & 0  \\
*& 0  & *  & 1 & \dots& * & 0\\
*& 0  & *  & 0 & \dots& * & 0\\
*&0&*&0&\dots& * & 0\\
\vdots&\vdots&\vdots&\vdots&\cdots&\vdots&\vdots\\
*& 0  & *  & 0 & \dots& * &1
\end{array}
\right).
\end{equation}

Note that $D^{-1}\widetilde{X}$ and $\widetilde{X}$ represent the same point in $\mathrm{LG}(n-1,V)$ and 
denote the $i$-th row of the matrix $D^{-1}\widetilde{X}$ by $(D^{-1}\widetilde{X})_{i}$. Denote by $\widetilde{\Phi}_n$ the matrix obtained from $D^{-1}\widetilde{X}$ by adding the row $(D^{-1}\widetilde{X})_{1}-(D^{-1}\widetilde{X})_{2}+\ldots(-1)^n (D^{-1}\widetilde{X})_{n-1}$
as the first row. 
Again, since the row span of $\widetilde{\Phi}_n$ coincides with the row span of $D^{-1}\widetilde{X}$, $\Phi_n$ represents the same point in $\mathrm{LG}(n-1,V)$ as $D^{-1}\widetilde{X}$. In fact we obtain that
\begin{equation}
\label{Matrix Omega_n}
\widetilde{\Phi}_n= \left(
\begin{array}{ccccccc}
* & 1 & * & -1 & \ldots & * & (-1)^n \\
*& 1  & *  & 0 & \dots& * & 0  \\
*& 0  & *  & 1 & \dots& * & 0\\
*& 0  & *  & 0 & \dots& * & 0\\
*&0&*&0&\dots& * & 0\\
\vdots&\vdots&\vdots&\vdots&\cdots&\vdots&\vdots\\
*& 0  & *  & 0 & \dots& * &1
\end{array}
\right).
\end{equation}

By the straightforward matrix multiplication  we obtain:
\begin{equation}
\label{eq:Binverse2}
\widetilde{\Phi}_nB_n= \left(
\begin{array}{cccccc}
z_{11}& 1  & z_{12}  & 0 & \dots& (-1)^n  \\
z_{21}& 1  & z_{22}  & 1 & \dots& 0  \\
z_{23}& 0  & z_{24}  & 1 & \dots& 0\\
z_{31}& 0  & z_{32}  & 0 & \dots& 0\\

\vdots&\vdots&\vdots&\vdots&\ddots&\vdots\\
\end{array}
\right).
\end{equation}

The matrix $\widetilde{\Phi}_nB_n$ represents the point $X$,
and also, as explained above, $X\subset V$. Thus the rows of the matrix $\widetilde{\Phi}_nB_n$ lie in the subspace $V$. Thus the alternated sum of odd elements in each row is equal to $0$.

Thus the matrix $\Phi_n$ has the form:
$$\Phi_n=\left(\begin{matrix}
x_{11} & 1 & -x_{12} & 0 & x_{13} & 0 & \ldots & (-1)^n \\
-x_{21} & 1 & x_{22} & 1 & -x_{23} & 0 & \ldots & 0 \\
x_{31} & 0 & -x_{32} & 1 & x_{33} & 1 & \ldots & 0 \\
\vdots & \vdots & \vdots & \vdots & \vdots & \vdots & \ddots & \vdots &  
\end{matrix}\right),$$
where $\sum\limits_{k\in \{1,3,\ldots,2n-1\}}x_{ik}=0$. It remains to prove that $(x_{ij})_{i,j\in[n]}$ is symmetric.

Applying Theorem \ref{lem: map_bij} we have that $\Phi_nB_n^{-1}$ is Lagrangian with respect to the symplectic form $\Lambda_{2n-2}$, i.e. 
$(\Phi_nB_n^{-1})\Lambda_{2n-2}(\Phi_nB_n^{-1})^T=0.$
Finally, using \cite[Lemma 4.6]{BGKT} we obtain that $x_{ij}=x_{ji}$. 
\end{proof}
\begin{remark}
Let us provide a sketch of an alternative proof of Lemma \ref{Parametrization of the top-cell}.   In \cite[Theorem 1.6]{CGS} the authors obtain  an explicit  matrix embedding  of electrical points (denoted $U_{\textit{}{not\ shorted}}$ in their terminology) into the Isotropic Grassmannian $\mathrm{IG}(n+1, 2n)$. Combining this result with duality \cite[Section 2.7]{Galashin} (described in this context precisely in \cite[Proposition 3.1]{BGGK}), one recovers   Lemma \ref{Parametrization of the top-cell}.
\end{remark}

\section{Reachability of contiguous  minors} \label{sec:appendix B}
The following was claimed in \cite[Theorem 6.2.16]{ALT}. We provide here a proof for completeness.
\begin{lemma}
Let $y_{(P;Q)}$ be a variable where $(P;Q)$ is a non-symmetric contiguous circular pair. Then $y_{(P;Q)}\in \mathcal{CM}_n$, and can be reached by mutations of the form \eqref{eq: Grassmann Plücker} from the initial seed $(\tilde{\mathbf{y}}_{D_n^\dagger};Q_{\mathcal{CM}_n})$.
\end{lemma}

\begin{proof}
We first prove the statement for odd $n$ and for pairs of size $1$, i.e. $|P|=|Q|=1$. 

Let $v$ be a vertex of the quiver $Q_{\mathcal{CM}_n}$. Define the {\itshape highest $D$-statistic transformation} $\mu_{n,v}$ at $v$ by the following inductive process. For $n=5$ define $\mu_{n,v}=\mu_v$, see Fig. \ref{fig:mu5}. For $n>5$ define $\mu_{n,v}$ to be the sequence of mutations: $\mu_{n-2,u}$, where $u$ runs over all mutable neighbors of $v$ (in any order but such that the transformations in the most central neighbors are performed first), and finally the mutation $\mu_v$ itself. Note that the transformations $\mu_{n-2,u}$ require mutations in mutual vertices for different neighbors $u$, but we perform each mutation exactly once. Namely, if a certain mutation was performed as a part of $\mu_{n-2,u_1}$ at a neighbor $u_1$ we do not repeat it if it is required later for $\mu_{n-2,u_2}$ for another neighbor $u_2$.  

There are observations ensuring that the process is well-defined. First one is that the leftmost and the rightmost vertices involved in $\mu_{n,v}$ are in clockwise order. This follows from the fact that there are $\left\lfloor \frac{n}{2}\right\rfloor-1$ circles and $2n$ radii in the quiver $Q_{\mathcal{CM}_n}$. The second observation is that for $n_1< n_2$, $\mu_{n_1,v}$ can be performed in $Q_{\mathcal{CM}_{n_2}}$. This follows from the previous observation. 
Note also that in the conventions made above, at each step there are actually only three neighbors at which the transformations have to be applied. 

Now we prove that $\mu_{n,v}(v)$ is contiguous.
In order to obtain it we prove (by induction on $n$) that $\mu_{n,v}(Q_{\mathcal{CM}_n})$ differs from $Q_{\mathcal{CM}_n}$ by a pyramid $P_{\mu_{n,v}}$ centered at $v$ and filled with diagonals, see Fig. \ref{Fig:pyramid-7} for a definition.
For the base case $n=5$ the statement holds trivially: a single mutation at $v$ of the form \eqref{eq: Grassmann Plücker}, which constitutes $\mu_{5,v}$ entirely, clearly produces the simplest possible pyramid: two squares with diagonals not adjacent to $v$ (Fig. \ref{fig:mu5}). This is a general phenomenon for mutations of the form \eqref{eq: Grassmann Plücker} on a square grid with an alternating orientation around each vertex: after a mutation in vertex $v$ a diagonal which is not adjacent to $v$ appears in every square adjacent to $v$ if there was no diagonal in this square, and disappears otherwise. Moreover, no other edges, except for the edges of $Q_{\mathcal{CM}_n}$ and the diagonals in the squares, can appear, otherwise, the quiver will not be planar, however, mutations of the form \eqref{eq: Grassmann Plücker} preserve planarity.

We now prove the induction step. Let $u_1,u_2,u_3$ be an ordering of the three neighbors of $v$ at which $\mu_{n-2,u_i}$ are to be performed. By the induction assumption, $\mu_{n-2,u_1}$ gives a pyramid centered at $u_1$. We next perform $\mu_{n-2,u_2}$, and then $\mu_{n-2,u_3}$. The induction hypothesis could be applied to each of them individually performed on the initial quiver. However, we perform them one after another, and by our inductive process, we do not repeat mutual mutations for different neighbors. Thus, diagonals in certain squares of a pyramid appeared during $\mu_{n-2,u_1}$ (resp. $\mu_{n-2,u_2}$) can potentially disappear during $\mu_{n-2,u_2}$ or $\mu_{n-2,u_3}$ (resp. $\mu_{n-2,u_3}$). We prove that actually this is not the case. In a fixed square there is a diagonal after $\mu_{n-2,u_i}$ if at the corners of this square an odd number of mutations was performed in total (later we will say for simplicity: an odd number of mutations in a square). By the induction hypothesis we know that after $\mu_{n-2,u_1}$ the pyramid centered at $u_1$ appears. It means that $\mu_{n-1,u_1}$ requires an odd number of mutations at each of the squares of the pyramid. Attaching the pyramids $P_{\mu_{n-2,u_2}}$ and $P_{\mu_{n-2,u_3}}$ to the pyramid $P_{\mu_{n-2,u_1}}$, we do not perform certain mutations. These mutations correspond to the smaller pyramids (it follows from the property that pyramids intersect by smaller pyramids). By the induction hypothesis each of these smaller pyramids requires an odd number of mutations in the corners of each square. Since performing $\mu_{n-2,u_2}$ and $\mu_{n-2,u_3}$ we skip mutations corresponding to these smaller pyramids, we skip an odd number of mutations. Thus, in each square of an intersection we add an even number of mutations to an odd number of mutations that were performed on previous steps. It follows that we have performed an odd number of mutations in each square, so we get a pyramid filled with diagonals as was asserted. 

We next prove that the diagonals in the two squares of the pyramid $P_{\mu_{n,v}}$ whose corner is $v$ are not adjacent to $v$. After applying $\mu_{n-2,u_1}$ and $\mu_{n-2,u_2}$, the first row of the pyramid of $\mu_{n,v}$ is ready and thus filled with diagonals. The last mutation in $\mu_{n-2,u_3}$ is $\mu(u_3)$, and it deletes the diagonals in the two squares of the first row adjacent to $v$. This means that in $\mu_{n-2,u_3}\circ\mu_{n-2,u_2}\circ\mu_{n-2,u_1}(Q_{\mathcal{CM}_n})$ the vertex $v$ is of degree $4$, thus, by \cite[Remark 6.2.13]{ALT}, we can perform a mutation of the form \eqref{eq: Grassmann Plücker} at $v$. Hence, by \cite[Lemma 6.2.14]{ALT}, $\mu_{n,v}(v)$ is a contiguous minor.

Now we explain how to reach all $1\times1$ contiguous pairs from the initial quiver $Q_{\mathcal{CM}_n}$. Denote the set of the $2n$ vertices lying on the most central circle in $Q_{\mathcal{CM}_n}$ by $v_1,\ldots,v_{2n}$. We perform $\mu_{i,v_1}(v_1),\ldots,\mu_{i,v_{2n}}(v_{2n})$ for every odd $i$ where $5\le i\le n$. Together with initial $v_1,\ldots,v_{2n}$ this gives $2\cdot\frac{n(n-1)}{2}$ variables since for every $i$ we get $2n$ variables and $i$ runs over the set of size $\frac{n-3}{2}$. It remains to explain why these variables are different. 

In order to do that, we need the statistics $D(P;Q),\ T(P;Q)$ and $k(P;Q)$, see Section \ref{sec: Central circular minors}. As remarked in \cite[Remark 6.2.8]{ALT} in terms of these statistics a mutation of the form \eqref{eq: Grassmann Plücker} can be rewritten as 
\begin{equation} \label{eq: Grassmann Plucker in statistics}
    (D-2,T,k)(D+2,T,k)=(D,T-1/2,k)(D,T+1/2,k)+(D,T,k+1)(D,T,k-1).
\end{equation}

From \eqref{eq: Grassmann Plucker in statistics} it is clear that $\mu_{i,v_l}(v_l)\ne\mu_{j,v_m}(v_m)$ for $l\ne m$ since they have different $T$-statistics. It remains to explain why $\mu_{i,v}(v)\ne\mu_{j,v}(v)$ for $i\ne j$. An expression of $\mu_{i,v}(v)$ in terms of the initial cluster contains a unique monomial coming from the top vertex of the pyramid $P_{\mu_{i,v}}$. For different $i$ these top vertices are different, and we are done.

For the case $|P|=|Q|>1$ the proof is similar with the only difference that we apply $\mu_{n,v}$ not at the vertex of the most central circle.

The proof in the case of even $n$ needs only a small adjustment: by performing symmetric mutations in $Q_{\mathcal{CM}_n}$ along one side of the part of the angle containing the diagonals, we can shift the location of the diagonals by one. Hence all quivers containing diagonals in any part of any angle lie in our algebra. Moreover, the base of the pyramid, appearing in the proof is less than the half of the circle, therefore additional diagonals (see Fig. \ref{fig:quiver C for even n}) do not obstruct the construction.
\end{proof}

\begin{remark}
Note that the pyramids $P_{\mu_{n,v}}$ closely resemble the Aztec diamonds defined in \cite[Section 4]{KW space}. Moreover, our reachability algorithm is reminiscent of the Kuo's graphical condensations \cite{Kuo}. 

Also, notice that there are  generalized Aztec diamonds corresponding to  semicontiguous circular minors (see \cite{Lai}), which may be related to their reachability by cluster mutations.
\end{remark}

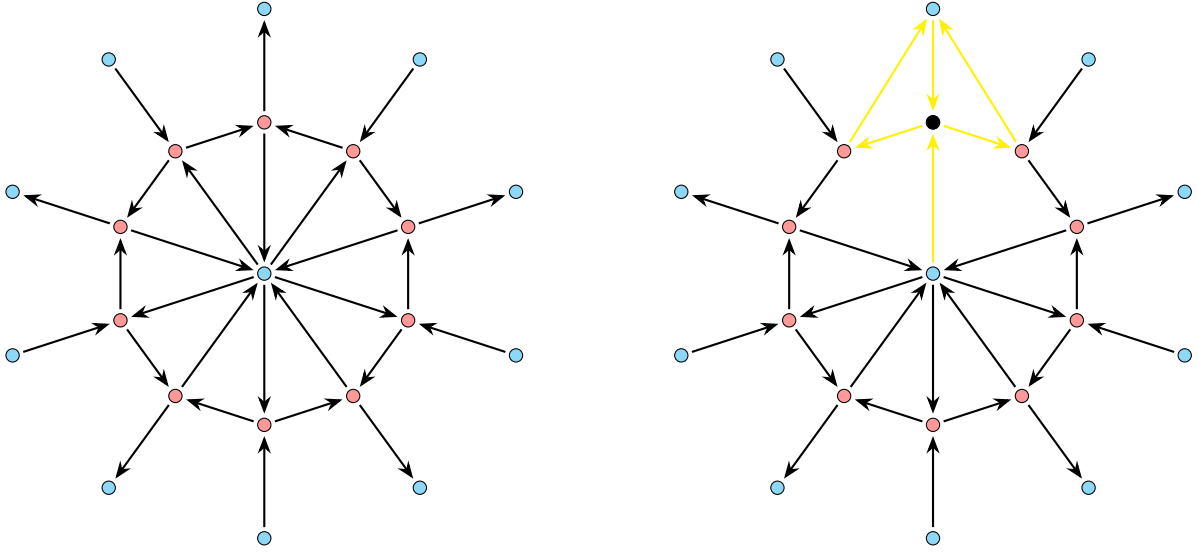
\begin{figure}[ht]

\begin{tikzpicture}[
    blueNode/.style={
        circle, draw=black, fill=cyan!40, inner sep=0pt, minimum size=5pt, line width=0.4pt
    },
    redNode/.style={
        circle, draw=black, fill=red!40, inner sep=0pt, minimum size=5pt, line width=0.4pt
    },
    arrow/.style={
        -{Stealth[scale=1.0]}, thick, shorten >=1.5pt, shorten <=1.5pt
    }
]

    \node[blueNode] (Center) at (0,0) {};

    \foreach \i in {0,1,...,9} {
        \pgfmathsetmacro{\angle}{90 + \i*36} 
        
        \node[redNode] (R\i) at (\angle:2) {};
        
        \node[blueNode] (B\i) at (\angle:3.5) {};
    }

        
        

\draw[arrow] (Center) -- (R1);
\draw[arrow] (Center) -- (R3);
\draw[arrow] (Center) -- (R5);
\draw[arrow] (Center) -- (R7);
\draw[arrow] (Center) -- (R9);

\draw[arrow] (R0) -- (Center);
\draw[arrow] (R2) -- (Center);
\draw[arrow] (R4) -- (Center);
\draw[arrow] (R6) -- (Center);
\draw[arrow] (R8) -- (Center);

\draw[arrow] (B1) -- (R1);
\draw[arrow] (B3) -- (R3);
\draw[arrow] (B5) -- (R5);
\draw[arrow] (B7) -- (R7);
\draw[arrow] (B9) -- (R9);

\draw[arrow] (R0) -- (B0);
\draw[arrow] (R2) -- (B2);
\draw[arrow] (R4) -- (B4);
\draw[arrow] (R6) -- (B6);
\draw[arrow] (R8) -- (B8);

\draw[arrow] (R1) -- (R0);
\draw[arrow] (R1) -- (R2);
\draw[arrow] (R3) -- (R2);
\draw[arrow] (R3) -- (R4);
\draw[arrow] (R5) -- (R4);
\draw[arrow] (R5) -- (R6);
\draw[arrow] (R7) -- (R6);
\draw[arrow] (R7) -- (R8);
\draw[arrow] (R9) -- (R8);
\draw[arrow] (R9) -- (R0);
\end{tikzpicture}\phantom{aaaaaaaaa}
\begin{tikzpicture}[
    blueNode/.style={
        circle, draw=black, fill=cyan!40, inner sep=0pt, minimum size=5pt, line width=0.4pt
    },
    redNode/.style={
        circle, draw=black, fill=red!40, inner sep=0pt, minimum size=5pt, line width=0.4pt
    },
     blackNode/.style={
        circle, draw=black, fill=black!100, inner sep=0pt, minimum size=5pt, line width=0.4pt
    },
    arrow/.style={
        -{Stealth[scale=1.0]}, thick, shorten >=1.5pt, shorten <=1.5pt
    },
    arrow1/.style={
        -{Stealth[scale=1.0]}, yellow, thick, shorten >=1.5pt, shorten <=1.5pt
    }
]

    \node[blueNode] (Center) at (0,0) {};

    \foreach \i in {0,1,...,9} {
        \pgfmathsetmacro{\angle}{90 + \i*36} 
        
        \node[redNode] (R\i) at (\angle:2) {};
        
        \node[blueNode] (B\i) at (\angle:3.5) {};
    }

        
        

\draw[arrow1] (R1) -- (B0);
\draw[arrow] (Center) -- (R3);
\draw[arrow] (Center) -- (R5);
\draw[arrow] (Center) -- (R7);
\draw[arrow1] (R9) -- (B0);

\draw[arrow1] (Center) -- (R0);
\draw[arrow] (R2) -- (Center);
\draw[arrow] (R4) -- (Center);
\draw[arrow] (R6) -- (Center);
\draw[arrow] (R8) -- (Center);

\draw[arrow] (B1) -- (R1);
\draw[arrow] (B3) -- (R3);
\draw[arrow] (B5) -- (R5);
\draw[arrow] (B7) -- (R7);
\draw[arrow] (B9) -- (R9);

\draw[arrow1] (B0) -- (R0);
\draw[arrow] (R2) -- (B2);
\draw[arrow] (R4) -- (B4);
\draw[arrow] (R6) -- (B6);
\draw[arrow] (R8) -- (B8);

\draw[arrow1] (R0) -- (R1);
\draw[arrow] (R1) -- (R2);
\draw[arrow] (R3) -- (R2);
\draw[arrow] (R3) -- (R4);
\draw[arrow] (R5) -- (R4);
\draw[arrow] (R5) -- (R6);
\draw[arrow] (R7) -- (R6);
\draw[arrow] (R7) -- (R8);
\draw[arrow] (R9) -- (R8);
\draw[arrow1] (R0) -- (R9);
\node[blackNode]  at (90:2) {};
\end{tikzpicture}
\caption{Quivers $Q_{\mathcal{CM}_5}$ and $\mu_{5,v}(Q_{\mathcal{CM}_5})$}
\label{fig:mu5}
\end{figure}

\begin{figure}
\centering
\scalebox{0.6} 
	{
\begin{tikzpicture}[
    baseNode/.style={
        circle, 
        draw=black, 
        fill=cyan!40, 
        inner sep=0pt, 
        minimum size=5pt, 
        line width=0.6pt
    },
    redNode/.style={
        fill=red!40 
    },
    arrow/.style={
        -{Stealth[scale=1.0]}, 
        thick,
        shorten >=1.5pt, 
        shorten <=1.5pt
    },
     arrow1/.style={
        -{Stealth[scale=1.0]}, yellow, thick, shorten >=1.5pt, shorten <=1.5pt
    }
]

    \node[baseNode] (N1) at (0,0) {};

    \def\n{14}          
    \def\rI{2.0}        
    \def\rII{3.8}       
    \def\rB{5.5}        

    \foreach \i in {0,1,...,13} {
        \pgfmathsetmacro{\angle}{90 - \i*(360/\n)} 
        
        \node[baseNode, redNode] (R1-\i) at (\angle:\rI) {};
        \node[baseNode, redNode] (R2-\i) at (\angle:\rII) {}; 
        \node[baseNode] (B-\i) at (\angle:\rB) {};
    }

    \foreach \i in {0,1,...,13} {
        \pgfmathsetmacro{\next}{int(mod(\i+1,\n))}
        
        \ifodd\i
            \draw[arrow] (R2-\i) -- (B-\i);
            \draw[arrow] (R2-\i) -- (R1-\i);
            \draw[arrow] (N1) -- (R1-\i);
            
            \draw[arrow] (R1-\i) -- (R1-\next);
            \draw[arrow] (R2-\next) -- (R2-\i);
        \else
            \draw[arrow] (R1-\i) -- (N1);
            \draw[arrow] (R1-\i) -- (R2-\i);
            \draw[arrow] (B-\i) -- (R2-\i);
            
            \draw[arrow] (R1-\next) -- (R1-\i);
            \draw[arrow] (R2-\i) -- (R2-\next);
        \fi
    }

\end{tikzpicture}
\hspace{6mm}
\begin{tikzpicture}[
    baseNode/.style={
        circle, 
        draw=black, 
        fill=cyan!40, 
        inner sep=0pt, 
        minimum size=5pt, 
        line width=0.6pt
    },
    redNode/.style={
        fill=red!40 
    },
    arrow/.style={
        -{Stealth[scale=1.0]}, 
        thick,
        shorten >=1.5pt, 
        shorten <=1.5pt
    },
     blackNode/.style={
        circle, draw=black, fill=black!100, inner sep=0pt, minimum size=5pt, line width=0.4pt
    },
    arrow1/.style={
        -{Stealth[scale=1.0]}, yellow, thick, shorten >=1.5pt, shorten <=1.5pt
    }
]

    \node[baseNode] (N1) at (0,0) {};

    \def\n{14}          
    \def\rI{2.0}        
    \def\rII{3.8}       
    \def\rB{5.5}        

    \foreach \i in {0,1,...,13} {
        \pgfmathsetmacro{\angle}{90 - \i*(360/\n)} 
        
        \node[baseNode, redNode] (R1-\i) at (\angle:\rI) {};
        \node[baseNode, redNode] (R2-\i) at (\angle:\rII) {}; 
        \node[baseNode] (B-\i) at (\angle:\rB) {};
    }

    \foreach \i in {0,1,...,13} {
        \pgfmathsetmacro{\next}{int(mod(\i+1,\n))}
        
        \ifodd\i
            \draw[arrow] (R2-\i) -- (B-\i);
            \draw[arrow] (R2-\i) -- (R1-\i);
            \draw[arrow] (N1) -- (R1-\i);
            
            \draw[arrow] (R1-\i) -- (R1-\next);
            \draw[arrow] (R2-\next) -- (R2-\i);
        \else
            \draw[arrow] (R1-\i) -- (N1);
            \draw[arrow] (R1-\i) -- (R2-\i);
            \draw[arrow] (B-\i) -- (R2-\i);
            
            \draw[arrow] (R1-\next) -- (R1-\i);
            \draw[arrow] (R2-\i) -- (R2-\next);
        \fi
    }

\draw[white, ->, line width=5pt] (Center) -- (R1-0);
\draw[white, ->, line width=5pt] (Center) -- (R1-12);
\draw[white, ->, line width=5pt] (R1-13) -- (Center);
\draw[white, ->, line width=5pt] (R1-13) -- (R1-0);
\draw[white, ->, line width=5pt] (R1-13) -- (R1-12);
\draw[white, ->, line width=5pt] (R2-13) -- (R1-13);
\draw[arrow1] (R1-13) -- (Center);
\draw[arrow1] (R1-13) -- (R2-13);
\draw[arrow1] (R1-0) -- (R1-13);
\draw[arrow1] (R1-12) -- (R1-13);
\draw[arrow1] (R2-13) -- (R1-0);
\draw[arrow1] (R2-13) -- (R1-12);
 \node[blackNode]  at (90 - 13*(360/\n:\rI) {};
\end{tikzpicture}
}
\end{figure}

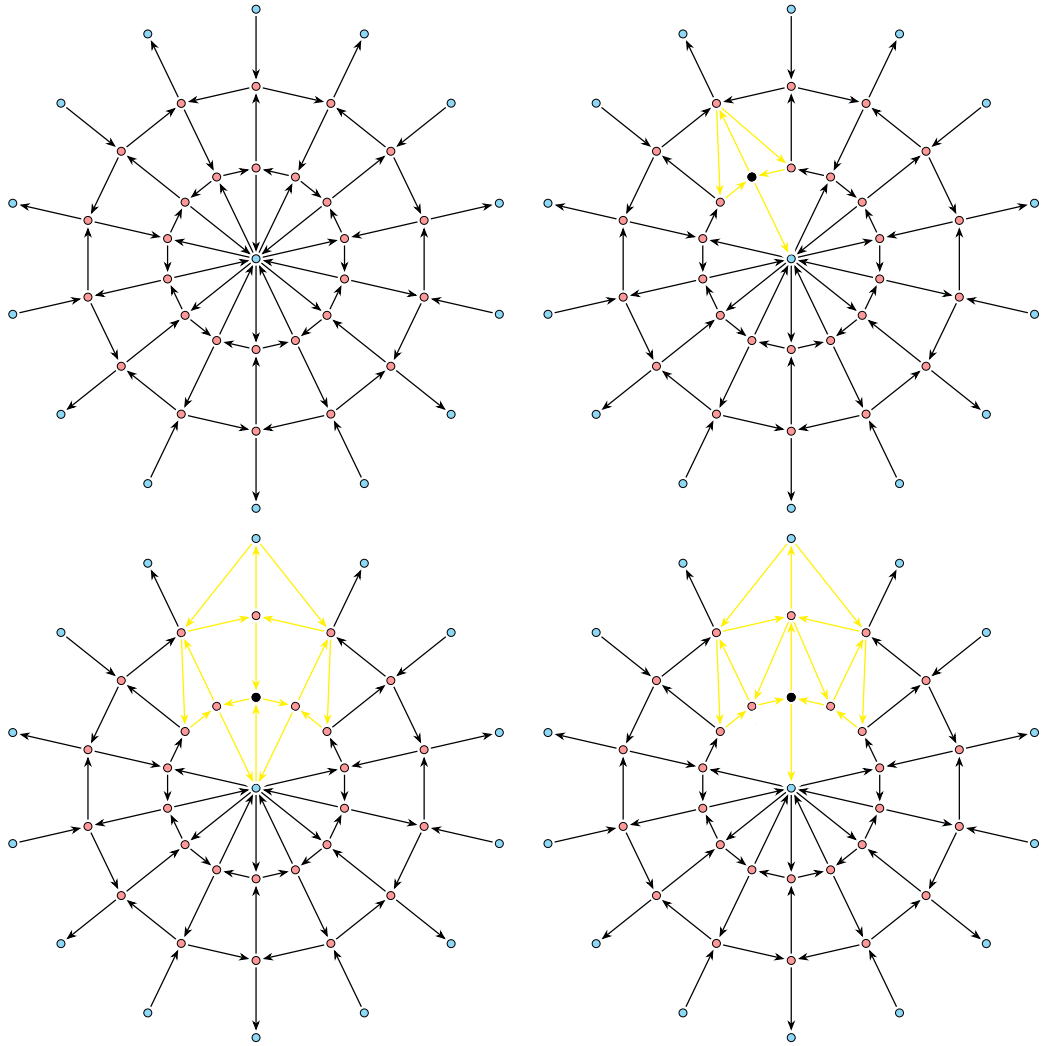
\begin{figure}

\centering
\scalebox{0.6} 
	{
\begin{tikzpicture}[
    baseNode/.style={
        circle, 
        draw=black, 
        fill=cyan!40, 
        inner sep=0pt, 
        minimum size=5pt, 
        line width=0.6pt
    },
    redNode/.style={
        fill=red!40 
    },
    arrow/.style={
        -{Stealth[scale=1.0]}, 
        thick,
        shorten >=1.5pt, 
        shorten <=1.5pt
    },
      blackNode/.style={
        circle, draw=black, fill=black!100, inner sep=0pt, minimum size=5pt, line width=0.4pt
    },
    arrow1/.style={
        -{Stealth[scale=1.0]}, yellow, thick, shorten >=1.5pt, shorten <=1.5pt
    }
]

    \node[baseNode] (N1) at (0,0) {};

    \def\n{14}          
    \def\rI{2.0}        
    \def\rII{3.8}       
    \def\rB{5.5}        

    \foreach \i in {0,1,...,13} {
        \pgfmathsetmacro{\angle}{90 - \i*(360/\n)} 
        
        \node[baseNode, redNode] (R1-\i) at (\angle:\rI) {};
        \node[baseNode, redNode] (R2-\i) at (\angle:\rII) {}; 
        \node[baseNode] (B-\i) at (\angle:\rB) {};
    }

    \foreach \i in {0,1,...,13} {
        \pgfmathsetmacro{\next}{int(mod(\i+1,\n))}
        
        \ifodd\i
            \draw[arrow] (R2-\i) -- (B-\i);
            \draw[arrow] (R2-\i) -- (R1-\i);
            \draw[arrow] (N1) -- (R1-\i);
            
            \draw[arrow] (R1-\i) -- (R1-\next);
            \draw[arrow] (R2-\next) -- (R2-\i);
        \else
            \draw[arrow] (R1-\i) -- (N1);
            \draw[arrow] (R1-\i) -- (R2-\i);
            \draw[arrow] (B-\i) -- (R2-\i);
            
            \draw[arrow] (R1-\next) -- (R1-\i);
            \draw[arrow] (R2-\i) -- (R2-\next);
        \fi
    }

\draw[white, ->, line width=5pt] (Center) -- (R1-0);
\draw[white, ->, line width=5pt] (Center) -- (R1-12);
\draw[white, ->, line width=5pt] (R1-13) -- (Center);
\draw[white, ->, line width=5pt] (R1-13) -- (R1-0);
\draw[white, ->, line width=5pt] (R1-13) -- (R1-12);
\draw[white, ->, line width=5pt] (R2-13) -- (R1-13);
\draw[white, ->, line width=5pt] (R1-1) -- (R1-0);
\draw[white, ->, line width=5pt] (R1-0) -- (R2-0);
\draw[white, ->, line width=5pt] (R1-1) -- (R1-2);
\draw[white, ->, line width=5pt] (Center) -- (R1-2);
\draw[white, ->, line width=5pt] (R2-1) -- (R1-1);
\draw[white, ->, line width=5pt] (B-0) -- (R2-0);
\draw[white, ->, line width=5pt] (R2-0) -- (R2-1);
\draw[white, ->, line width=5pt] (R2-0) -- (R2-13);
\draw[white, ->, line width=5pt] (R1-1) -- (Center);

\draw[arrow1] (Center) -- (R1-0);
\draw[arrow1] (R1-13) -- (Center);
\draw[arrow1] (R1-1) -- (Center);
\draw[arrow1] (B-0) -- (R2-13);
\draw[arrow1] (B-0) -- (R2-1);
\draw[arrow1] (R2-0) -- (B-0);
\draw[arrow1] (R2-1) -- (R2-0);
\draw[arrow1] (R2-13) -- (R2-0);
\draw[arrow1] (R2-0) -- (R1-0);
\draw[arrow1] (R1-13) -- (R2-13);
\draw[arrow1] (R1-1) -- (R2-1);
\draw[arrow1] (R1-0) -- (R1-13);
\draw[arrow1] (R1-0) -- (R1-1);
\draw[arrow1] (R1-12) -- (R1-13);
\draw[arrow1] (R1-2) -- (R1-1);
\draw[arrow1] (R2-1) -- (R1-2);
\draw[arrow1] (R2-13) -- (R1-12);
 \node[blackNode]  at (90:\rI) {};

\end{tikzpicture}
\hspace{6mm}
\begin{tikzpicture}[
    baseNode/.style={
        circle, 
        draw=black, 
        fill=cyan!40, 
        inner sep=0pt, 
        minimum size=5pt, 
        line width=0.6pt
    },
    redNode/.style={
        fill=red!40 
    },
    arrow/.style={
        -{Stealth[scale=1.0]}, 
        thick,
        shorten >=1.5pt, 
        shorten <=1.5pt
    },
     blackNode/.style={
        circle, draw=black, fill=black!100, inner sep=0pt, minimum size=5pt, line width=0.4pt
    },
    arrow1/.style={
        -{Stealth[scale=1.0]}, yellow, thick, shorten >=1.5pt, shorten <=1.5pt
    }
]

    \node[baseNode] (N1) at (0,0) {};

    \def\n{14}          
    \def\rI{2.0}        
    \def\rII{3.8}       
    \def\rB{5.5}        

    \foreach \i in {0,1,...,13} {
        \pgfmathsetmacro{\angle}{90 - \i*(360/\n)} 
        
        \node[baseNode, redNode] (R1-\i) at (\angle:\rI) {};
        \node[baseNode, redNode] (R2-\i) at (\angle:\rII) {}; 
        \node[baseNode] (B-\i) at (\angle:\rB) {};
    }

    \foreach \i in {0,1,...,13} {
        \pgfmathsetmacro{\next}{int(mod(\i+1,\n))}
        
        \ifodd\i
            \draw[arrow] (R2-\i) -- (B-\i);
            \draw[arrow] (R2-\i) -- (R1-\i);
            \draw[arrow] (N1) -- (R1-\i);
            
            \draw[arrow] (R1-\i) -- (R1-\next);
            \draw[arrow] (R2-\next) -- (R2-\i);
        \else
            \draw[arrow] (R1-\i) -- (N1);
            \draw[arrow] (R1-\i) -- (R2-\i);
            \draw[arrow] (B-\i) -- (R2-\i);
            
            \draw[arrow] (R1-\next) -- (R1-\i);
            \draw[arrow] (R2-\i) -- (R2-\next);
        \fi
    }

\draw[white, ->, line width=5pt] (Center) -- (R1-0);
\draw[white, ->, line width=5pt] (Center) -- (R1-12);
\draw[white, ->, line width=5pt] (R1-13) -- (Center);
\draw[white, ->, line width=5pt] (R1-13) -- (R1-0);
\draw[white, ->, line width=5pt] (R1-13) -- (R1-12);
\draw[white, ->, line width=5pt] (R2-13) -- (R1-13);
\draw[white, ->, line width=5pt] (R1-1) -- (R1-0);
\draw[white, ->, line width=5pt] (R1-0) -- (R2-0);
\draw[white, ->, line width=5pt] (R1-1) -- (R1-2);
\draw[white, ->, line width=5pt] (Center) -- (R1-2);
\draw[white, ->, line width=5pt] (R2-1) -- (R1-1);
\draw[white, ->, line width=5pt] (B-0) -- (R2-0);
\draw[white, ->, line width=5pt] (R2-0) -- (R2-1);
\draw[white, ->, line width=5pt] (R2-0) -- (R2-13);
\draw[white, ->, line width=5pt] (R1-1) -- (Center);

\draw[arrow1] (Center) -- (R1-0);
\draw[arrow1] (R1-13) -- (Center);
\draw[arrow1] (R1-1) -- (Center);
\draw[arrow1] (B-0) -- (R2-13);
\draw[arrow1] (B-0) -- (R2-1);
\draw[arrow1] (R2-0) -- (B-0);
\draw[arrow1] (R2-1) -- (R2-0);
\draw[arrow1] (R2-13) -- (R2-0);
\draw[arrow1] (R2-0) -- (R1-0);
\draw[arrow1] (R1-13) -- (R2-13);
\draw[arrow1] (R1-1) -- (R2-1);
\draw[arrow1] (R1-0) -- (R1-13);
\draw[arrow1] (R1-0) -- (R1-1);
\draw[arrow1] (R1-12) -- (R1-13);
\draw[arrow1] (R1-2) -- (R1-1);
\draw[arrow1] (R2-1) -- (R1-2);
\draw[arrow1] (R2-13) -- (R1-12);

\draw[white, ->, line width=5pt] (Center) -- (R1-0);
\draw[white, ->, line width=5pt] (R1-13) -- (Center);
\draw[white, ->, line width=5pt] (R1-1) -- (Center);
\draw[white, ->, line width=5pt] (R1-0) -- (R1-1);
\draw[white, ->, line width=5pt] (R1-0) -- (R1-13);
\draw[white, ->, line width=5pt] (R2-0) -- (R1-0);

\draw[arrow1] (R1-0) -- (R2-0);
\draw[arrow1] (R1-1) -- (R1-0);
\draw[arrow1] (R1-13) -- (R1-0);
\draw[arrow1] (R2-0) -- (R1-1);
\draw[arrow1] (R2-0) -- (R1-13);
\draw[arrow1] (R1-0) -- (Center);
\node[blackNode]  at (90:\rI) {};

\end{tikzpicture}
}
\caption{Quivers $Q_{\mathcal{CM}_7}$ and $\mu_{7,v}(Q_{\mathcal{CM}_7})$}
\label{Fig:pyramid-7}
\end{figure}

\begin{figure}
\vspace{5mm}
\centering
\begin{tikzpicture}[scale=0.55,
    blueNode/.style={
        circle, 
        draw=black, 
        fill=cyan!40, 
        inner sep=0pt, 
        minimum size=5pt, 
        line width=0.6pt
    },
    redNode/.style={
        circle, 
        draw=black, 
        fill=red!40, 
        inner sep=0pt, 
        minimum size=5pt, 
        line width=0.6pt 
    },
    arrow/.style={
        -{Stealth[scale=1.0]}, 
        thick,
        shorten >=1.5pt, 
        shorten <=1.5pt
    }
]

    \node[blueNode] (Center) at (0,0) {};

    \def\n{18}          
    \def\rI{4}        
    \def\rII{6}       
    \def\rIII{8}      
    \def\rB{10}        

    \foreach \i in {0,1,...,17} {
        \pgfmathsetmacro{\angle}{90 - \i*(360/\n)} 
        
        \node[redNode]  (R1-\i) at (\angle:\rI)   {};
        \node[redNode]  (R2-\i) at (\angle:\rII)  {};
        \node[redNode]  (R3-\i) at (\angle:\rIII) {};
        \node[blueNode] (B-\i)  at (\angle:\rB)   {};
    }

    \foreach \i in {0,1,...,17} {
        \pgfmathsetmacro{\next}{int(mod(\i+1,\n))}
        
        \ifodd\i
            \draw[arrow] (B-\i)  -- (R3-\i);
            \draw[arrow] (R2-\i) -- (R3-\i) ;
            \draw[arrow] (R2-\i) -- (R1-\i);
            \draw[arrow] (Center) -- (R1-\i) ;
            
            \draw[arrow] (R1-\i) -- (R1-\next);
            \draw[arrow] (R2-\next) -- (R2-\i);
            \draw[arrow] (R3-\i) -- (R3-\next);
        \else
            \draw[arrow] (R1-\i) -- (Center) ;
            \draw[arrow] (R1-\i)  -- (R2-\i);
            \draw[arrow] (R3-\i) -- (R2-\i) ;
            \draw[arrow] (R3-\i)  -- (B-\i);
            
            \draw[arrow] (R1-\next) -- (R1-\i);
            \draw[arrow] (R2-\i)    -- (R2-\next);
            \draw[arrow] (R3-\next) -- (R3-\i);
        \fi
    }

\end{tikzpicture}
\end{figure}

\begin{figure}
\vspace{5mm}
\centering
\begin{tikzpicture}[scale=0.55,
    blueNode/.style={
        circle, 
        draw=black, 
        fill=cyan!40, 
        inner sep=0pt, 
        minimum size=5pt, 
        line width=0.6pt
    },
    redNode/.style={
        circle, 
        draw=black, 
        fill=red!40, 
        inner sep=0pt, 
        minimum size=5pt, 
        line width=0.6pt 
    },
    arrow/.style={
        -{Stealth[scale=1.0]}, 
        thick,
        shorten >=1.5pt, 
        shorten <=1.5pt
    },
     blackNode/.style={
        circle, draw=black, fill=black!100, inner sep=0pt, minimum size=5pt, line width=0.4pt
    },
    arrow1/.style={
        -{Stealth[scale=1.0]}, yellow, thick, shorten >=1.5pt, shorten <=1.5pt
    }
]

    \node[blueNode] (Center) at (0,0) {};

    \def\n{18}          
    \def\rI{4}        
    \def\rII{6}       
    \def\rIII{8}      
    \def\rB{10}        

    \foreach \i in {0,1,...,17} {
        \pgfmathsetmacro{\angle}{90 - \i*(360/\n)} 
        
        \node[redNode]  (R1-\i) at (\angle:\rI)   {};
        \node[redNode]  (R2-\i) at (\angle:\rII)  {};
        \node[redNode]  (R3-\i) at (\angle:\rIII) {};
        \node[blueNode] (B-\i)  at (\angle:\rB)   {};
    }

    \foreach \i in {0,1,...,17} {
        \pgfmathsetmacro{\next}{int(mod(\i+1,\n))}
        
        \ifodd\i
            \draw[arrow] (B-\i)  -- (R3-\i);
            \draw[arrow] (R2-\i) -- (R3-\i) ;
            \draw[arrow] (R2-\i) -- (R1-\i);
            \draw[arrow] (Center) -- (R1-\i) ;
            
            \draw[arrow] (R1-\i) -- (R1-\next);
            \draw[arrow] (R2-\next) -- (R2-\i);
            \draw[arrow] (R3-\i) -- (R3-\next);
        \else
            \draw[arrow] (R1-\i) -- (Center) ;
            \draw[arrow] (R1-\i)  -- (R2-\i);
            \draw[arrow] (R3-\i) -- (R2-\i) ;
            \draw[arrow] (R3-\i)  -- (B-\i);
            
            \draw[arrow] (R1-\next) -- (R1-\i);
            \draw[arrow] (R2-\i)    -- (R2-\next);
            \draw[arrow] (R3-\next) -- (R3-\i);
        \fi
    }
\draw[white, ->, line width=5pt] (R1-1) -- (R1-0);
\draw[white, ->, line width=5pt] (R1-17) -- (R1-0);
\draw[white, ->, line width=5pt] (R1-1) -- (R1-2);
\draw[white, ->, line width=5pt] (Center) -- (R1-1);
\draw[white, ->, line width=5pt] (R1-0) -- (Center);
\draw[white, ->, line width=5pt] (Center) -- (R1-17);
\draw[white, ->, line width=5pt] (R1-16) -- (Center);
\draw[white, ->, line width=5pt] (Center) -- (R1-15);
\draw[white, ->, line width=5pt] (R1-17) -- (R1-16);
\draw[white, ->, line width=5pt] (R1-15) -- (R1-16);
\draw[white, ->, line width=5pt] (R1-16) -- (R2-16);
\draw[white, ->, line width=5pt] (R2-17) -- (R1-17);
\draw[white, ->, line width=5pt] (R1-0) -- (R2-0);
\draw[white, ->, line width=5pt] (R2-17) -- (R3-17);
\draw[white, ->, line width=5pt] (R2-16) -- (R2-17);
\draw[white, ->, line width=5pt] (R2-0) -- (R2-17);

\draw[arrow1] (Center) -- (R1-17);
\draw[arrow] (R1-14) -- (Center);
\draw[arrow] (R1-2) -- (Center);
\draw[arrow] (R1-1) -- (R1-2);
\draw[arrow1] (R1-0) -- (R2-17);
\draw[arrow1] (R1-16) -- (R2-17);
\draw[arrow1] (R1-15) -- (R2-16);
\draw[arrow1] (R2-16) -- (R3-17);
\draw[arrow1] (R1-1) -- (R2-0);
\draw[arrow1] (R2-0) -- (R3-17);
\draw[arrow1] (R3-17) -- (R2-17);
\draw[arrow1] (R2-17) -- (R1-17);
\draw[arrow1] (R2-17) -- (R2-0);
\draw[arrow1] (R2-17) -- (R2-16);
\draw[arrow1] (R1-17) -- (R1-0);
\draw[arrow1] (R1-0) -- (R1-1);
\draw[arrow1] (R1-17) -- (R1-16);
\draw[arrow1] (R1-16) -- (R1-15);
\draw[arrow1] (R2-0) -- (R1-0);
\draw[arrow1] (R2-16) -- (R1-16);
\node[blackNode] at (-250: 4) {};

\end{tikzpicture}
\end{figure}

\begin{figure}
\centering
\begin{tikzpicture}[scale=0.55,
    blueNode/.style={
        circle, 
        draw=black, 
        fill=cyan!40, 
        inner sep=0pt, 
        minimum size=5pt, 
        line width=0.6pt
    },
    redNode/.style={
        circle, 
        draw=black, 
        fill=red!40, 
        inner sep=0pt, 
        minimum size=5pt, 
        line width=0.6pt 
    },
    arrow/.style={
        -{Stealth[scale=1.0]}, 
        thick,
        shorten >=1.5pt, 
        shorten <=1.5pt
    },
     blackNode/.style={
        circle, draw=black, fill=black!100, inner sep=0pt, minimum size=5pt, line width=0.4pt
    },
    arrow1/.style={
        -{Stealth[scale=1.0]}, yellow, thick, shorten >=1.5pt, shorten <=1.5pt
    }
]

    \node[blueNode] (Center) at (0,0) {};

    \def\n{18}          
    \def\rI{4}        
    \def\rII{6}       
    \def\rIII{8}      
    \def\rB{10}        

    \foreach \i in {0,1,...,17} {
        \pgfmathsetmacro{\angle}{90 - \i*(360/\n)} 
        
        \node[redNode]  (R1-\i) at (\angle:\rI)   {};
        \node[redNode]  (R2-\i) at (\angle:\rII)  {};
        \node[redNode]  (R3-\i) at (\angle:\rIII) {};
        \node[blueNode] (B-\i)  at (\angle:\rB)   {};
    }

    \foreach \i in {0,1,...,17} {
        \pgfmathsetmacro{\next}{int(mod(\i+1,\n))}
        
        \ifodd\i
            \draw[arrow] (B-\i)  -- (R3-\i);
            \draw[arrow] (R2-\i) -- (R3-\i) ;
            \draw[arrow] (R2-\i) -- (R1-\i);
            \draw[arrow] (Center) -- (R1-\i) ;
            
            \draw[arrow] (R1-\i) -- (R1-\next);
            \draw[arrow] (R2-\next) -- (R2-\i);
            \draw[arrow] (R3-\i) -- (R3-\next);
        \else
            \draw[arrow] (R1-\i) -- (Center) ;
            \draw[arrow] (R1-\i)  -- (R2-\i);
            \draw[arrow] (R3-\i) -- (R2-\i) ;
            \draw[arrow] (R3-\i)  -- (B-\i);
            
            \draw[arrow] (R1-\next) -- (R1-\i);
            \draw[arrow] (R2-\i)    -- (R2-\next);
            \draw[arrow] (R3-\next) -- (R3-\i);
        \fi
    }
\draw[white, ->, line width=5pt] (R1-1) -- (R1-0);
\draw[white, ->, line width=5pt] (R1-17) -- (R1-0);
\draw[white, ->, line width=5pt] (R1-1) -- (R1-2);
\draw[white, ->, line width=5pt] (Center) -- (R1-1);
\draw[white, ->, line width=5pt] (R1-0) -- (Center);
\draw[white, ->, line width=5pt] (Center) -- (R1-17);
\draw[white, ->, line width=5pt] (R1-16) -- (Center);
\draw[white, ->, line width=5pt] (Center) -- (R1-15);
\draw[white, ->, line width=5pt] (R1-17) -- (R1-16);
\draw[white, ->, line width=5pt] (R1-15) -- (R1-16);
\draw[white, ->, line width=5pt] (R1-16) -- (R2-16);
\draw[white, ->, line width=5pt] (R2-17) -- (R1-17);
\draw[white, ->, line width=5pt] (R1-0) -- (R2-0);
\draw[white, ->, line width=5pt] (R2-17) -- (R3-17);
\draw[white, ->, line width=5pt] (R2-16) -- (R2-17);
\draw[white, ->, line width=5pt] (R2-0) -- (R2-17);
\draw[white, ->, line width=5pt] (R1-3) -- (R1-2);
\draw[white, ->, line width=5pt] (R1-2) -- (Center);
\draw[white, ->, line width=5pt] (Center) -- (R1-3);
\draw[white, ->, line width=5pt] (R2-1) -- (R1-1);
\draw[white, ->, line width=5pt] (R1-2) -- (R2-2);
\draw[white, ->, line width=5pt] (R2-0) -- (R2-1);
\draw[white, ->, line width=5pt] (R3-0) -- (R2-0);
\draw[white, ->, line width=5pt] (R3-17) -- (R3-0);
\draw[white, ->, line width=5pt] (R3-1) -- (R3-0);
\draw[white, ->, line width=5pt] (R3-0) -- (B-0);
\draw[white, ->, line width=5pt] (R2-2) -- (R2-1);
\draw[white, ->, line width=5pt] (R2-1) -- (R3-1);

\draw[arrow] (R1-14) -- (Center);
\draw[arrow] (R1-4) -- (Center);

\draw[arrow1] (R1-15) -- (R2-16);
\draw[arrow1] (R2-16) -- (R3-17);
\draw[arrow1] (R3-17) -- (B-0);
\draw[arrow1] (R1-3) -- (R2-2);
\draw[arrow1] (R2-2) -- (R3-1);
\draw[arrow1] (R3-1) -- (B-0);

\draw[arrow1] (R1-16) -- (R2-17);
\draw[arrow1] (R2-17) -- (R3-0);
\draw[arrow1] (R1-2) -- (R2-1);
\draw[arrow1] (R2-1) -- (R3-0);


\draw[arrow1] (R3-0) -- (R3-17);
\draw[arrow1] (R3-0) -- (R3-1);

\draw[arrow1] (R2-0) -- (R2-17);
\draw[arrow1] (R2-17) -- (R2-16);
\draw[arrow1] (R2-0) -- (R2-1);
\draw[arrow1] (R2-1) -- (R2-2);

\draw[arrow1] (R1-17) -- (R1-0);
\draw[arrow1] (R1-1) -- (R1-0);
\draw[arrow1] (R1-17) -- (R1-16);
\draw[arrow1] (R1-16) -- (R1-15);
\draw[arrow1] (R1-1) -- (R1-2);
\draw[arrow1] (R1-2) -- (R1-3);

\draw[arrow1] (R2-16) -- (R1-16);
\draw[arrow1] (R3-17) -- (R2-17);
\draw[arrow1] (R2-17) -- (R1-17);

\draw[arrow1] (R2-1) -- (R1-1);
\draw[arrow1] (R3-1) -- (R2-1);
\draw[arrow1] (R2-1) -- (R1-1);

\draw[arrow1] (B-0) -- (R3-0);
\draw[arrow1] (R3-0) -- (R2-0);
\draw[arrow1] (R1-0) -- (R2-0);
\draw[arrow1] (R1-0) -- (Center);
\draw[arrow1] (Center) -- (R1-1);
\draw[arrow1] (Center) -- (R1-17);
\node[blackNode] at (90: 4) {};

\end{tikzpicture}
\end{figure}

\begin{figure}
\vspace{1mm}
\centering
\begin{tikzpicture}[scale=0.55,
    blueNode/.style={
        circle, 
        draw=black, 
        fill=cyan!40, 
        inner sep=0pt, 
        minimum size=5pt, 
        line width=0.6pt
    },
    redNode/.style={
        circle, 
        draw=black, 
        fill=red!40, 
        inner sep=0pt, 
        minimum size=5pt, 
        line width=0.6pt 
    },
    arrow/.style={
        -{Stealth[scale=1.0]}, 
        thick,
        shorten >=1.5pt, 
        shorten <=1.5pt
    },
     blackNode/.style={
        circle, draw=black, fill=black!100, inner sep=0pt, minimum size=5pt, line width=0.4pt
    },
    arrow1/.style={
        -{Stealth[scale=1.0]}, yellow, thick, shorten >=1.5pt, shorten <=1.5pt
    }
]

    \node[blueNode] (Center) at (0,0) {};

    \def\n{18}          
    \def\rI{4}        
    \def\rII{6}       
    \def\rIII{8}      
    \def\rB{10}        

    \foreach \i in {0,1,...,17} {
        \pgfmathsetmacro{\angle}{90 - \i*(360/\n)} 
        
        \node[redNode]  (R1-\i) at (\angle:\rI)   {};
        \node[redNode]  (R2-\i) at (\angle:\rII)  {};
        \node[redNode]  (R3-\i) at (\angle:\rIII) {};
        \node[blueNode] (B-\i)  at (\angle:\rB)   {};
    }

    \foreach \i in {0,1,...,17} {
        \pgfmathsetmacro{\next}{int(mod(\i+1,\n))}
        
        \ifodd\i
            \draw[arrow] (B-\i)  -- (R3-\i);
            \draw[arrow] (R2-\i) -- (R3-\i) ;
            \draw[arrow] (R2-\i) -- (R1-\i);
            \draw[arrow] (Center) -- (R1-\i) ;
            
            \draw[arrow] (R1-\i) -- (R1-\next);
            \draw[arrow] (R2-\next) -- (R2-\i);
            \draw[arrow] (R3-\i) -- (R3-\next);
        \else
            \draw[arrow] (R1-\i) -- (Center) ;
            \draw[arrow] (R1-\i)  -- (R2-\i);
            \draw[arrow] (R3-\i) -- (R2-\i) ;
            \draw[arrow] (R3-\i)  -- (B-\i);
            
            \draw[arrow] (R1-\next) -- (R1-\i);
            \draw[arrow] (R2-\i)    -- (R2-\next);
            \draw[arrow] (R3-\next) -- (R3-\i);
        \fi
    }
\draw[white, ->, line width=5pt] (R1-1) -- (R1-0);
\draw[white, ->, line width=5pt] (R1-17) -- (R1-0);
\draw[white, ->, line width=5pt] (R1-1) -- (R1-2);
\draw[white, ->, line width=5pt] (Center) -- (R1-1);
\draw[white, ->, line width=5pt] (R1-0) -- (Center);
\draw[white, ->, line width=5pt] (Center) -- (R1-17);
\draw[white, ->, line width=5pt] (R1-16) -- (Center);
\draw[white, ->, line width=5pt] (Center) -- (R1-15);
\draw[white, ->, line width=5pt] (R1-17) -- (R1-16);
\draw[white, ->, line width=5pt] (R1-15) -- (R1-16);
\draw[white, ->, line width=5pt] (R1-16) -- (R2-16);
\draw[white, ->, line width=5pt] (R2-17) -- (R1-17);
\draw[white, ->, line width=5pt] (R1-0) -- (R2-0);
\draw[white, ->, line width=5pt] (R2-17) -- (R3-17);
\draw[white, ->, line width=5pt] (R2-16) -- (R2-17);
\draw[white, ->, line width=5pt] (R2-0) -- (R2-17);
\draw[white, ->, line width=5pt] (R1-3) -- (R1-2);
\draw[white, ->, line width=5pt] (R1-2) -- (Center);
\draw[white, ->, line width=5pt] (Center) -- (R1-3);
\draw[white, ->, line width=5pt] (R2-1) -- (R1-1);
\draw[white, ->, line width=5pt] (R1-2) -- (R2-2);
\draw[white, ->, line width=5pt] (R2-0) -- (R2-1);
\draw[white, ->, line width=5pt] (R3-0) -- (R2-0);
\draw[white, ->, line width=5pt] (R3-17) -- (R3-0);
\draw[white, ->, line width=5pt] (R3-1) -- (R3-0);
\draw[white, ->, line width=5pt] (R3-0) -- (B-0);
\draw[white, ->, line width=5pt] (R2-2) -- (R2-1);
\draw[white, ->, line width=5pt] (R2-1) -- (R3-1);

\draw[arrow] (R1-14) -- (Center);
\draw[arrow] (R1-4) -- (Center);

\draw[arrow1] (R1-15) -- (R2-16);
\draw[arrow1] (R2-16) -- (R3-17);
\draw[arrow1] (R3-17) -- (B-0);
\draw[arrow1] (R1-3) -- (R2-2);
\draw[arrow1] (R2-2) -- (R3-1);
\draw[arrow1] (R3-1) -- (B-0);

\draw[arrow1] (R1-16) -- (R2-17);
\draw[arrow1] (R2-17) -- (R3-0);
\draw[arrow1] (R1-2) -- (R2-1);
\draw[arrow1] (R2-1) -- (R3-0);

\draw[arrow1] (R3-0) -- (R3-17);
\draw[arrow1] (R3-0) -- (R3-1);

\draw[arrow1] (R2-0) -- (R2-17);
\draw[arrow1] (R2-17) -- (R2-16);
\draw[arrow1] (R2-0) -- (R2-1);
\draw[arrow1] (R2-1) -- (R2-2);

\draw[arrow1] (R1-17) -- (R1-16);
\draw[arrow1] (R1-16) -- (R1-15);

\draw[arrow1] (R2-16) -- (R1-16);
\draw[arrow1] (R3-17) -- (R2-17);
\draw[arrow1] (R2-17) -- (R1-17);

\draw[arrow1] (R2-1) -- (R1-1);
\draw[arrow1] (R3-1) -- (R2-1);
\draw[arrow1] (R2-1) -- (R1-1);

\draw[arrow1] (B-0) -- (R3-0);
\draw[arrow1] (R3-0) -- (R2-0);
\draw[arrow1] (R1-17) -- (R2-0);
\draw[arrow1] (R1-1) -- (R2-0);
\draw[arrow1] (R1-0) -- (R1-17);
\draw[arrow1] (R1-0) -- (R1-1);
\draw[arrow1] (R1-1) -- (R1-2);
\draw[arrow1] (R1-2) -- (R1-3);
\draw[arrow1] (Center) -- (R1-0);
\draw[arrow1] (R2-0) -- (R1-0);
\node[blackNode]  at (90:\rI) {};
\end{tikzpicture}
\caption{Quivers $Q_{\mathcal{CM}_9}$ and $\mu_{9,v}(Q_{\mathcal{CM}_9})$}
\end{figure}

\end{document}